\documentclass[reqno]{amsart}

\usepackage[all]{xy}
\usepackage{graphicx}

\usepackage{amssymb}
\usepackage{amsmath}
\usepackage{mathrsfs}
\usepackage{epsfig}
\usepackage{amscd}
\usepackage{graphicx, color}
\usepackage{marginnote}
\usepackage{mathtools}
\usepackage{hyperref}
\usepackage{tikz-cd}

\def\E{\ifmmode{\mathbb E}\else{$\mathbb E$}\fi} 
\def\N{\ifmmode{\mathbb N}\else{$\mathbb N$}\fi} 
\def\R{\ifmmode{\mathbb R}\else{$\mathbb R$}\fi} 
\def\Q{\ifmmode{\mathbb Q}\else{$\mathbb Q$}\fi} 
\def\C{\ifmmode{\mathbb C}\else{$\mathbb C$}\fi} 
\def\H{\ifmmode{\mathbb H}\else{$\mathbb H$}\fi} 
\def\Z{\ifmmode{\mathbb Z}\else{$\mathbb Z$}\fi} 
\def\P{\ifmmode{\mathbb P}\else{$\mathbb P$}\fi} 
\def\T{\ifmmode{\mathbb T}\else{$\mathbb T$}\fi} 
\def\SS{\ifmmode{\mathbb S}\else{$\mathbb S$}\fi} 
\def\DD{\ifmmode{\mathbb D}\else{$\mathbb D$}\fi} 

\newcommand{\del}{\partial}

\newcommand{\ben}{\begin{enumerate}}
\newcommand{\een}{\end{enumerate}}
\newcommand{\be}{\begin{equation}}
\newcommand{\ee}{\end{equation}}
\newcommand{\bea}{\begin{eqnarray}}
\newcommand{\eea}{\end{eqnarray}}
\newcommand{\beastar}{\begin{eqnarray*}}
\newcommand{\eeastar}{\end{eqnarray*}}
\newcommand{\bc}{\begin{center}}
\newcommand{\ec}{\end{center}}

\theoremstyle{theorem}
\newtheorem{thm}{Theorem}[section]

\newtheorem{lem}[thm]{Lemma}
\newtheorem{prop}[thm]{Proposition}

\newtheorem{rmk}[thm]{Remark}

\theoremstyle{definition}
\newtheorem{defn}[thm]{Definition}
\newtheorem{rem}[thm]{Remark}

\newtheorem*{thm*}{Theorem}

\numberwithin{equation}{section}

\def\R{{\mathbb R}}

\def\E{{\mathbb E}}
\def\Z{{\mathbb Z}}
\def\C{{\mathbb C}}
\def\R{{\mathbb R}}
\def\P{{\mathbb P}}

\def\N{{\mathbb N}}

\def\11{{\mathbb I}}

\def\n{{\noindent}}

\def\Int{\operatorname{Int}}

\def\C{\mathbb{C}}
\def\Z{\mathbb{Z}}

\def\T{\mathbb{T}}

\def\D{\mathbb{D}}
\def\Q{\mathbb{Q}}

\def\E{\ifmmode{\mathbb E}\else{$\mathbb E$}\fi} 
\def\N{\ifmmode{\mathbb N}\else{$\mathbb N$}\fi} 
\def\R{\ifmmode{\mathbb R}\else{$\mathbb R$}\fi} 
\def\Q{\ifmmode{\mathbb Q}\else{$\mathbb Q$}\fi} 
\def\C{\ifmmode{\mathbb C}\else{$\mathbb C$}\fi} 
\def\H{\ifmmode{\mathbb H}\else{$\mathbb H$}\fi} 
\def\Z{\ifmmode{\mathbb Z}\else{$\mathbb Z$}\fi} 
\def\P{\ifmmode{\mathbb P}\else{$\mathbb P$}\fi} 
\def\SS{\ifmmode{\mathbb S}\else{$\mathbb S$}\fi} 
\def\DD{\ifmmode{\mathbb D}\else{$\mathbb D$}\fi} 

\def\R{{\mathbb R}}

\def\E{{\mathbb E}}
\def\Z{{\mathbb Z}}
\def\C{{\mathbb C}}
\def\R{{\mathbb R}}

\def\N{{\mathbb N}}
\def\MM{{\mathcal M}}

\def\n{\nu}

\def\CA{{\mathcal A}}
\def\CB{{\mathcal B}}
\def\CC{{\mathcal C}}

\def\CF{{\mathcal F}}
\def\CG{{\mathcal G}}
\def\CH{{\mathcal H}}

\def\CJ{{\mathcal J}}

\def\CL{{\mathcal L}}
\def\CM{{\mathcal M}}
\def\CN{{\mathcal N}}

\def\CS{{\mathcal S}}

\def\CW{{\mathcal W}}

\def\darr#1{\raise1.5ex\hbox{$\leftrightarrow$}
\mkern-16.5mu #1}

\def\roughly#1{\raise.3ex\hbox{$#1$\kern-.75em
\lower1ex\hbox{$\sim$}}}

\def\opname#1{\mathop{\kern0pt{\rm #1}}\nolimits}

\def\Im{\opname{Im}}

\def\dim{\opname{dim}}

\def\coker{\operatorname{Coker}}
\def\Span{\operatorname{Span}}

\begin{document}

\quad \vskip1.375truein

\def\mq{\mathfrak{q}}
\def\mp{\mathfrak{p}}
\def\mH{\mathfrak{H}}
\def\mh{\mathfrak{h}}
\def\ma{\mathfrak{a}}
\def\ms{\mathfrak{s}}
\def\mm{\mathfrak{m}}
\def\mn{\mathfrak{n}}
\def\mz{\mathfrak{z}}
\def\mw{\mathfrak{w}}
\def\Hoch{{\tt Hoch}}
\def\mt{\mathfrak{t}}
\def\ml{\mathfrak{l}}
\def\mT{\mathfrak{T}}
\def\mL{\mathfrak{L}}
\def\mg{\mathfrak{g}}
\def\md{\mathfrak{d}}
\def\mr{\mathfrak{r}}

\title[A wrapped Fukaya category of knot complement]
{A wrapped Fukaya category of knot complement}

\author{Youngjin Bae, Seonhwa Kim, Yong-Geun Oh}

\thanks{SK and YO are supported by the IBS project IBS-R003-D1.
YO is also partially supported by the National Science Foundation under Grant No. DMS-1440140 during his residence at the Mathematical Sciences Research Institute in Berkeley, California in the fall of 2018.
YB was partially supported by IBS-R003-D1 and JSPS International Research Fellowship Program.}

\address{Youngjin Bae\\
   Research Institute for Mathematical Sciences, Kyoto University\\
   Kyoto Prefecture, Kyoto, Sakyo Ward, Kitashirakawa Oiwakecho, Japan 606-8317}
\email{ybae@kurims.kyoto-u.ac.jp}
\address{Seonhwa Kim\\
Center for Geometry and Physics, Institute for Basic Sciences (IBS), Pohang, Korea} \email{ryeona17@ibs.re.kr}
\address{Yong-Geun Oh\\
Center for Geometry and Physics, Institute for Basic Sciences (IBS), Pohang, Korea \& Department of Mathematics,
POSTECH, Pohang, Korea} \email{yongoh1@postech.ac.kr}

\date{January 2, 2019; Revision on March 10, 2019}

\begin{abstract}
This is the first of a series of two articles where we construct a
version of wrapped Fukaya category $\CW\CF(M\setminus K;H_{g_0})$ of the cotangent bundle $T^*(M \setminus K)$
of the knot complement $M \setminus K$ of a compact 3-manifold $M$,
 and do some calculation for the case of
hyperbolic knots $K \subset M$. For the construction, we use the wrapping
induced by the kinetic energy Hamiltonian $H_{g_0}$ associated to the cylindrical adjustment $g_0$
on $M \setminus K$ of a smooth metric $g$ defined on $M$. We then consider
the torus $T = \del N(K)$ as an object in this category and its wrapped Floer
complex $CW^*(\nu^*T;H_{g_0})$ where $N(K)$ is a tubular neighborhood of $K \subset M$.
We prove that the quasi-equivalence class of the category
and the quasi-isomorphism class of the $A_\infty$ algebra $CW^*(\nu^*T;H_{g_0})$
are independent of the choice of cylindrical adjustments of such metrics depending only on
the isotopy class of the knot $K$ in $M$.

In a sequel \cite{BKO}, we give constructions of a wrapped Fukaya category
$\CW\CF(M\setminus K;H_h)$ for hyperbolic knot $K$ and of $A_\infty$ algebra
$CW^*(\nu^*T;H_h)$ directly using the hyperbolic metric $h$ on $M \setminus K$, and
prove a formality result for the asymptotic boundary of $(M \setminus K, h)$.
\end{abstract}


\maketitle

\setcounter{tocdepth}{1}
\tableofcontents

\section{Introduction}
\label{sec:intro}

Idea of using the conormal lift of a knot (or a link) in $\R^3$ or $S^3$ as a Legendrian submanifold
in the unit cotangent bundle has been  exploited by Ekholm-Etynre-Ng-Sullivan \cite{EENS}
in their construction of knot contact homology and proved that this analytic invariants
recovers Ng's combinatorial invariants of the knot which is an isomorphism class of
certain differential graded algebras \cite{ng:knot}.
On the other hand, Floer homology of conormal bundles of submanifolds of a compact
smooth manifold in the full cotangent bundle were studied as a quantization
of singular homology of the submanifold (see \cite{oh:jdg}, \cite{kasturi-oh}).
Such a construction has been extended in the $A_\infty$ level by Nadler-Zaslow
\cite{nadler-zaslow}, \cite{nadler} and also studied by Abbondandolo-Portaluri-Schwarz \cite{APS}
in its relation to the singular homology of the space of cords of the submanifold.

The present article is the first of a series of two articles where we
construct a version of wrapped Fukaya category of $T^*(M \setminus K)$
of the knot-complement $M\setminus K$, which is \emph{noncompact},
as an invariant of knot $K$ (or more generally of links)
and do some computation of the invariant for the case of
hyperbolic knot $K \subset M$ by relating the
(perturbed) pseudoholomorphic triangles in $T^*(M \setminus K)$ to the hyperbolic
geodesic triangles of the base $M \setminus K$. (See \cite{BKO} for the latter.)

For our purpose of investigating the effect on the topology of $M\setminus K$
of the special metric behavior such as the existence of a
hyperbolic metric in the complement $M \setminus K$, it is important to directly deal with
the cotangent bundle of the full complement $M \setminus K$ equipped with the wrapping
induced by the kinetic energy Hamiltonian of a metric on $M\setminus K$.
To carry out necessary analysis of the relevant perturbed Cauchy-Riemann
equation, we need to impose certain tame behavior of the associated Hamiltonian and almost complex
structure at infinity of $M \setminus K$. Because of noncompactness of
$M \setminus K$, the resulting $A_\infty$ category a priori depends on the
Liptschitz-equivalence class of such metrics modulo conformal equivalence and requires a
choice of such equivalence class in the construction.

In the present paper, we will consider the restricted class of Hamiltonians that asymptotically
coincide with the \emph{kinetic energy Hamiltonian} denoted by $H_g$ associated Riemannian metric $g$ on the complement and the Sasakian almost complex structure $J_g$ associated to the metric $g$. We will need a
suitable `tameness' of the metric near the end of complement $M \setminus K$
so that a uniform \emph{horizontal} $C^0$ bound holds for the relevant Cauchy-Riemann equation
with any given tuple $(L^0, \ldots, L^k)$ of Lagrangian boundary conditions from the
given collection of \emph{admissible} Lagrangians.
As long as such a $C^0$ bound
is available, one can directly, construct a wrapped Fukaya category using such a metric.
It turns out that such a horizontal bound can be proved in general only for
the metric with suitable tame behavior such as  those with \emph{cylindrical ends} or with certain type of homogeneous behavior at the end of $M \setminus K$ like a complete hyperbolic metric.
We will consider the case of hyperbolic knot $K \subset M$ in a sequel \cite{BKO}
to the present paper. It turns out that $T^*(M\setminus K)$ is convex at infinity
in the sense that there is a $J_h$-pluri-subharmonic exhaustion function in a neighborhood at infinity of
$T^*(M\setminus K)$ when $K$ is a hyperbolic knot. We refer to \cite{BKO} for the explanation of
this latter property of the hyperbolic knot.

\subsection{Construction of wrapped Fukaya category  $\CW\CF(M \setminus K)$}

The main purpose of the present paper is to define a Fukaya-type category canonically associated to
the knot complement $M \setminus K$. To make our definition of wrapped Fukaya category of
knot complement flexible enough,
we consider a compact oriented Riemannian manifolds $(M,g)$ without boundary.
(We remark that the case of hyperbolic knot $K \subset M$ does not belong to this case
because the hyperbolic metric on $M \setminus K$ cannot be smoothly extended to the whole $M$.)

For this purpose, we take a tubular neighborhood $N(K) \subset M$ of $K$ and decompose
\be\label{eq:decompose}
M \setminus K = N^{\text{\rm cpt}} \cup (N(K) \setminus K)
\ee
and equip a cylindrical metric on $N(K) \setminus K \cong [0, \infty) \times T_K$
with $T_K = \del N(K)$. We call such a metric a cylindrical adjustment of the
given metric $g$ on $M$, and denote by $g_0$ the cylindrical adjustment of $g$
on $N(K) \setminus K$.  An essential analytical reason why we take such an
cylindrical adjustment of the metric and its associated kinetic energy Hamiltonian is because
it enables us to obtain the \emph{horizontal $C^0$-estimates} for the relevant
perturbed Cauchy-Riemann equation with various boundary conditions.
To highlight importance and nontriviality of such $C^0$-estimates,
we collect all the proofs of relevant $C^0$-estimates in Part 2. As a function
on $T^*(M \setminus K)$, the kinetic energy of such a metric on $M \setminus K$,
blows up outside the zero section as one approaches to the knot $K$.

Denote the associated kinetic energy Hamiltonian of $g_0$ on $T^*(M\setminus K)$ by
\be\label{eq:Hg0}
H_{g_0} = \frac12|p|^2_{g_0}.
\ee
The first main theorem is a construction the following $A_\infty$ functor between
two different choices of various data involved in the construction of
$\CW\CF (M \setminus K; H_{g_0})$.

\begin{thm}\label{thm:functoriality} Let $N(K)$ be a tubular neighborhood of $K$ be given.
For any two smooth metrics $g, \, g'$ on $M$, denote by $g_0, \, g_0'$ the associated
cylindrical adjustments thereof as above. Then there exist a natural $A_\infty$ quasi-equivalence
$$
\Phi_{gg'} : \CW\CF (M \setminus K; H_{g_0}) \to \CW\CF (M \setminus K; H_{g_0'})
$$
for any pair $g, \, g'$ such that $\Phi_{gg} = id$. Furthermore its quasi-equivalence class
depends only on the isotopy type of the knot $K$ independent of the choice of tubular neighborhoods
and other data.
\end{thm}

We denote by
$$
\CW\CF (M \setminus K)
$$
any such $A_\infty$ category $\CW\CF (M \setminus K; H_{g_0})$.
Each $A_\infty$ category $\CW\CF (M\setminus K)$ will be constructed on $T^*(M\setminus K)$
using the \emph{kinetic energy Hamiltonian} $H_{g_0}$
and the Sasakian almost complex structure $J_{g_0}$ on $T^*(M \setminus K)$ of $g_0$ on $M\setminus K$: In general a Sasakian almost complex structure $J_h$ associated to the metric $h$ on a
Riemannian manifold $(N,h)$ is given by
\be\label{eq:SasakianJg}
J_h (X) = X^\flat, \quad J_h(\alpha) = - \alpha^\sharp
\ee
under the splitting $T(T^*N) \simeq TN \oplus T^*N$ via the Levi-Civita connection of $h$.

Then the proof of Theorem \ref{thm:functoriality} is relied on construction of various
$A_\infty$ operators, $A_\infty$ functors and $A_\infty$ homotopies in the
current context of Floer theory on $T^*(M\setminus K)$. The general strategy of
such construction is by now standard in Floer theory. (See, for example, \cite{seidel:book}.)
In fact, our construction applies to any arbitrary tame orientable 3-manifold with
 boundary and similar computational result applies when the 3-manifold admits
a complete hyperbolic metric of finite volume. Our construction is given in this generality.
 (See Theorem \ref{thm:invariant} for the precise statement.)

\begin{rem}
In the proofs of Theorems \ref{thm:functoriality} and \ref{thm:algebra},
we need to construct various $A_\infty$ functors and $A_\infty$ homotopies between them
which enter in the  invariance proofs. We adopt the definitions of them given in \cite{hasegawa} for
the constructions of the $A_\infty$ functor and the $A_\infty$ homotopy directly using
the continuations of either Hamiltonians or of Lagrangians or others. Construction of $A_\infty$
functors appear in the literature in various circumstances, but we could not locate
a literature containing geometric construction of an $A_\infty$ homotopy
in the sense of \cite{hasegawa,seidel:book} for the case of geometric continuations
such as Hamiltonian isotopies. (In \cite{fooo:book1,fooo:book2,fukaya:cyclic},
 the notion of $A_\infty$ homotopy is defined via a suspension model, the notion of pseudo-isotopy,
of the  chain complex.) Because of these reasons and for the convenience of readers,
we provide full details, in the categorical context, of the construction $A_\infty$ homotopy
associated to a continuation of Hamiltonians in Section \ref{sec:homotopy}. Our construction
of $A_\infty$ functor is the counterpart of the standard Floer continuation equation also applied to
the higher $\frak m^k$ maps with $k \geq 2$. It turns out
that actual constructions of the associated $A_\infty$ homotopy as well as of $A_\infty$ functors are
rather subtle and require some thought on the correct moduli spaces that enter
in definitions of $A_\infty$ functors and homotopies associated to geometric continuations.
(See Subsection \ref{subsec:functor} and Section \ref{sec:homotopy} for the definitions of relevant moduli spaces.) We also call readers' attention to Savelyev's relevant construction in the context of $\infty$-category in \cite{savelyev}.
\end{rem}

\subsection{Construction of Knot Floer algebra}

For a concrete computation we do in \cite{BKO}, we focus on a particular object
in this category $\CW\CF(M \setminus K)$ canonically associated to the knot.
For given tubular neighborhood $N(K)$ of $K$, we consider the conormal
$$
L = \nu^*T, \,\quad T = \del N(K)
$$
and then the wrapped Fukaya algebra of the Lagrangian $L$ in $T^*(M\setminus K)$.
We remark that the wrapped Fukaya algebra of $\nu^*T$ in $T^*M$ for a closed manifold $M$
can be described
by purely topological data arising from the base space, more specifically
that of the space of paths attached to $T$ in $M$. (See \cite{APS} for some relevant result.)
So it is important to consider $L$ as an object for $T^*(M\setminus K)$
to get more interesting knot invariant.

On the other hand, since we restrict the class of our Hamiltonians to that of
kinetic energy Hamiltonian $H_{g_0}$ associated to a Riemannian metric $g_0$
on $M\setminus K$, the pair $(\nu^*T, H_{g_0})$ is not a nondegenerate pair but
a \emph{clean pair} in that the set of Hamiltonian chords contains a continuum of
constant chords valued at points of $T \cong \T^2$.

We denote by $\mathfrak X(L;H_{g_0})= \mathfrak X(L,L;H_{g_0})$ the set of Hamiltonian chords of $H_{g_0}$
attached to a Lagrangian submanifold $L$ in general.
We have
$$
\mathfrak X(L;H_{g_0}) = \mathfrak X_0(L;H_{g_0}) \coprod \mathfrak X_{ < 0}(L;H_{g_0})
$$
where the subindex of $\mathfrak X$ in the right hand side denotes the action of
the Hamiltonian chord of $H_{g_0}$ which is the negative of the
length of the corresponding geodesic. (See \eqref{eq:CA-E} below for the explanation on the relevant
convention used in the present paper.)
It is shown in \cite[Proposition 2.1]{BKO} that the constant component $\mathfrak X_0(L;H_{g_0})$
of the critical set of the action functional $\CA_{g_0}$ is clean
in the sense of Bott and so diffeomorphic to $\T^2$ as a smooth manifold.
We take
$$
CW(\nu^*T,\nu^*T; T^*(M\setminus K); H_{g_0}) = C^*(T) \oplus \Z\{\mathfrak X_{< 0}(L;H_{g_0})\}
$$
where $C^*(T)$ is any model of cochain complex of $T$, e.g., $C^*(T) = \Omega^*(T)$ the de Rham complex
and associate an $A_\infty$ algebra following the construction from \cite{fooo:book1}.

\begin{thm}\label{thm:algebra} Let $N(K)$ be a tubular neighborhood of $K$ and
let $T = \del (N(K))$. The $A_\infty$ algebra
$$
(CW(\nu^*T, T^*(M\setminus K); H_{g_0}), {\frak m), \quad \frak m
= \{\frak m}^k\}_{1 \leq k < \infty}
$$
can be defined, and its isomorphism class does not depend on the various
choices involved such as tubular neighborhood $N(K)$ and
the metric $g$ on $M$.
\end{thm}

\begin{rem}
Due to the presence of Morse-Bott component of constant chords, there are two routes
toward construction of wrapped Floer complex, which we denote by
$
CW_g(\nu^*T, T^*(M \setminus K)):
$
One is to take the model
$$
CW(\nu^*T,\nu^*T; T^*(M\setminus K); H_{g_0}) = C^*(T) \oplus \Z\langle \mathfrak X_{< 0}(L;H_h)\rangle
$$
where $C^*(T)$ is a chain complex of $T$ such as the singular chain complex
as in \cite{fooo:book1} or the de Rham complex, and the other is to take
$$
CW(\nu_k^*T, \nu_k^*T; T^*(M\setminus K); H_{g_0}) \cong \Z\langle \text{\rm Crit }k\rangle \oplus \Z\langle\mathfrak X_{< 0}(L;H_h)\rangle
$$
where $\nu_k^*T$ the fiberwise translation of $\nu^*T$ by $dk$,
suitably interpolated with $\nu^*T$ away from the zero section,
for a sufficiently $C^2$-small compactly supported Morse function
$k: M \setminus K \to \R$ such that $\nu^*T \pitchfork \text{\rm Image } dk$.
We refer readers to \cite[Section 2]{BKO} for the detailed explanation
on the latter model.
\end{rem}

The isomorphism class of $CW_g(T, M \setminus K)$ independent of $g$ then provides a knot-invariant of $K$
in $M$ for an arbitrary knot $K$.
To emphasize the fact that we regard $L = \nu^*T$ as an object in the cotangent bundle of a knot complement
$M \setminus K$, not as one in the full cotangent bundle $T^*M$, we denote the cohomology group of
$CW_g(T, M \setminus K)$ as follows.

\begin{defn}[Knot Floer algebra]\label{defn:knotFloeralgebra}
We denote the cohomology of $CW_g(T, M \setminus K)$ by
$$
HW(\del_\infty(M \setminus K))
$$
which carries a natural product arising from $\frak m^2$ map. We call this
\emph{Knot Floer algebra} of $K \subset M$.
\end{defn}
By letting the torus converge to the ideal boundary of $M \setminus K$, we may
regard $HW(\del_\infty(M \setminus K))$ as the wrapped Floer cohomology of the `ideal boundary'
of the hyperbolic manifolds $M \setminus K$, which is the origin of the notation
$\del_\infty(M \setminus K)$ we are adopting.

In a sequel \cite{BKO}, we introduce a reduced version of
the $A_\infty$ algebra, denoted by $\widetilde{CW}^d(\del_\infty(M \setminus K))$,
by considering the complex generated by \emph{non-constant}
Hamiltonian chords, and  prove that for the case of hyperbolic knot $K \subset M$
this algebra can be also directly calculated by
considering a horo-torus $T$ and the wrapped Floer complex $CW(\nu^*T, T^*(M\setminus K);H_h)$
of the hyperbolic metric $h$ although $h$ cannot be smoothly extended to the whole manifold $M$.
We also prove a formality result of its $A_\infty$ structure
$\widetilde{CW}(\nu^*T,T^*(M\setminus K);H_h)$ for any hyperbolic knot $K$.
The following is the main result we prove in \cite{BKO}.

\begin{thm}[Theorem 1.6 \cite{BKO}]\label{thm:compare-with-h} Suppose $K$ is a hyperbolic knot on $M$. Then
we have an (algebra) isomorphism
$$
HW^d(\del_\infty(M \setminus K)) \cong HW^d(\nu^*T;H_h)
$$
for all integer $d \geq 0$. Furthermore the reduced cohomology $\widetilde{HW}^d(\del_\infty(M \setminus K)) = 0$ for all $d \geq 1$.
\end{thm}

\subsection{$C^0$ estimates}

One crucial new ingredient in the proofs of the above theorems is to establish the \emph{horizontal}
$C^0$ estimates of solutions of the perturbed Cauchy-Riemann equation
\be\label{eq:duXH01=0}
(du -  X_{\bf H} \otimes \beta)_{\bf J}^{(0,1)} = 0
\ee
mainly for the energy Hamiltonian $H_{g_0}(q,p) = \frac{1}{2}|p|^2_{g_0}$. For this purpose, we have to require
the one-form $\beta$ to satisfy `co-closedness'
$$
d(\beta \circ j) = 0
$$
in addition to the usual requirement imposed in \cite{abou-seidel}.
Together with the well-known requirement of \emph{subclosedness} for the \emph{vertical} $C^0$ estimates
\cite{abou-seidel}, we need to require $\beta$ to satisfy
\be\label{eq:co-closed}
d\beta \leq 0, \quad d(\beta \circ j) = 0, \, i^*\beta = 0
\ee
for the inclusion map $i: \del \Sigma \to \Sigma$. This brings the question whether
such one-forms exist in the way that the choice is compatible with
the gluing process of the relevant moduli spaces.
We explicitly construct such a $\beta$ by pulling back a one form from a \emph{slit domain}.
This choice of one form is consistent under the gluing of moduli spaces.

Roughly speaking, this $C^0$-estimates guarantees that under the above preparation,
whenever the test objects $L^0, \ldots, L^k$ are all contained in $W_i = T^*N_i$,
the images of the solutions of the Cauchy-Riemann equation
is also contained in $W_i$, i.e., \emph{they do not approach the knot $K$.}
We refer to Part 2 for the proofs of various $C^0$ estimates
needed for the construction.
Furthermore for the proof of independence of the $A_\infty$ category $\CW\CF_g(M \setminus K)$
under the choice smooth metric $g$, one need to construct an
$A_\infty$ functor between them for two different choices of $g, \, g'$.
Because the maximum principle applies only in the increasing direction of the associated
Hamiltonian from $H_{g}$ to $H_{g'}$, we have to impose the monotonicity
condition
$$
g \geq g' \quad \text{ or equivalently } H_{g'} \leq H_g.
$$

\subsection{Further perspective}

Putting the main results of the present paper in perspective, we compare
 our category $\CW\CF(M \setminus K)$ is a version of partially wrapped
Fukaya category on $T^*M$ with the `stop' given by
$$
\Lambda^\infty_K : = \del^\infty(T^*M)|_{K} \subset \del^\infty(T^*M)
$$
for $\Lambda_K = T^*M|_{K}$. This is a 3-dimensional \emph{coisotropic submanifold}
of the asymptotic contact boundary
$$
\del^\infty(T^*M) = \bigcup_{q \in M} \del^\infty(T^*_q M)
$$
of $T^*M$ which is of 5 dimension.  In this regard, our construction can be compared
with the partially wrapped Fukaya categories as follows.

Our category avoids the entire colostropic submanifold $\Lambda_K \subset T^*M$
while partially wrapped Fukaya category with relevant coisotropic stop avoids
$\Lambda^\infty_K \subset \del^\infty(T^*M)$ only asymptotically. In this regard,
M. Abouzaid pointed out that  Sylvan's partially wrapped Fukaya category with stop $\Lambda^\infty_K$ can be
regarded as a bulk deformation of our category via \emph{coisotropic} submanifold of
codimension 2. (See \cite{seidel:quartic} for the usage of such
bulk deformation via a \emph{complex divisor}. We also refer to \cite{fooo:bulk} for
a general theory of bulk deformations by \emph{toric divisors}.)
The way how we avoid this stop is
by attaching the cylindrical end on $M \setminus N(K)$ along its boundary
$\del(M \setminus N(K)) = \del N(K)$ and considering the kinetic energy Hamiltonian
of a cylindrical adjustment $g_0$ of a smooth metric $g$ defined on $M$. As mentioned before,
this Hamiltonian blows up as we approach to $K$ off the zero section, which will prevent
the image of any finite energy solution of \emph{perturbed} Cauchy-Riemann equation \eqref{eq:duXH01=0}
from touching the fiber of $T^*M$ restricted to $K$.
Our approach directly working with $T^*(M\setminus K)$, which is well-adapted to the metric structures
on the base manifold, will be important for our later purpose of studying hyperbolic knots
in a sequel \cite{BKO}.

We would also like to mention that in the arXiv version of \cite[\S 6.5]{ENS}, Ekholm-Ng-Shende
considered a wrapped Fukaya category on the Weinstein manifold denoted by $W_K$ to $K$ that is obtained by
attaching a punctured handle to $DT^*\R^3$ along the unit conormal bundle of the knot
and altering the Liouville vector field of $T^*M$ along $\nu^*K$. It would be interesting
to compare our category with theirs.

A similar construction of $A_\infty$ algebra for the conormal $\nu^*T$ of $T$ as above
can be carried out for the `conormal' $\nu^*N^{\text{\rm cpt}}$ of $N^{\text{\rm cpt}} := N \setminus N(K)$ or the \emph{micro-support} of
the characteristic function $\chi_{N^{\text{\rm cpt}}}$, which is given by
$$
\nu^*N^{\text{\rm cpt}} = o_{N^{\text{\rm cpt}}} \coprod \nu^*_- (\del N^{\text{\rm cpt}})
$$
where $o_{N^{\text{\rm cpt}}}$ is the zero section restricted to $N^{\text{\rm cpt}}$,
$T = \del N^{\text{\rm cpt}}$ is equipped with boundary orientation of $N^{\text{\rm cpt}}$
and $\nu^*_- (\del N^{\text{\rm cpt}})$ is the negative conormal of $\del N^{\text{\rm cpt}}$.
Combined with the construction of the wrapped version of the natural restriction morphisms
constructed in \cite{oh:natural}, this construction
would give rise to $A_\infty$ morphisms $\widehat{\nu^*N_i} \to \widehat{\nu^*T}$ and
a natural $A_\infty$ functor
\be\label{eq:nuNiTi}
\xymatrix{
& \widehat{\nu^*N_1} \ar[r] \ar[d] & \widehat{\nu^*N_2} \ar[d]\\
& \widehat{\nu^*T_1} \ar[r] & \widehat{\nu^*T_2}  }
\ee
where $N_i = M \setminus N_i(K)$ and $T_i = \del N_i$ for two tubular
neighborhoods $N_1(K) \subset N_2(K)$ of $K$ for $i=1, \, 2$.
Here $\widehat{L}$ denotes the Yoneda image of the Lagrangian $L$ in general.
(Since we will not use this construction in this series, we will leave
further discussion elsewhere not to further lengthen the paper.)

On the other hand it is interesting to see that when we are given a exhaustion sequence
$N_1 \subset N_2 \subset \cdots N_i \subset \cdots $, the union $M \cup \nu^*K$ is a limit
of the above mentioned conormal $\nu^*N_i$ as $i \to \infty$ on
$T^*(M \setminus K)$ in the Gromov-Hausdorff sense.
Furthermore the canonical smoothing of
$\nu^*N_i$ given in \cite[Theorem 2.3]{kasturi-oh} can be made to converge to the $S^1$-family
Lagrangian surgery of $o_M \cup \nu^*K$ denoted by $M_K$ in \cite{AENV}, \cite{ENS} which consider
the case $M = S^3$. While construction of Lagrangian $M_K$
requires some geometric restriction $K$ such as being fibered (see \cite[Lemma 6.12]{AENV}),
construction of exact Lagrangian smoothing of $\nu^*K$ can be done for arbitrary knot $K$. It would be interesting to see what the ramification of this observation will be.

\begin{rem} It is reasonable to
expect that our category is equivalent to the limit of the wrapped Fukaya categories
$\CW(T^*N_i)$ defined in \cite{GPS} for the symplectic manifolds with boundary, or
\emph{Liouville sectors}
$$
X_i = T^*N_i,\quad \del X_i = T^*N_i|_{\del N_i}:
$$
While we put wrapping through by ambient Hamiltonian flows,
\cite{GPS} put wrapping on the Lagrangian itself but considered the unperturbed Cauchy-Riemann
equation with positively moving boundary condition and localization trick of category.
\end{rem}

\subsection{Conventions}
\label{subsec:sasakian}

In the literature on symplectic geometry, Hamiltonian dynamics, contact geometry and
the physics literature, there are various conventions used which are different from
one another one way or the other. Because many things considered in the present paper such as
the energy estimates, the $C^0$ estimates applying the maximum principle and
construction of the Floer continuation map depend on the choice of various conventions,
we highlight the essential components of our convention that affect their validity.

The major differences between different conventions in the literature
lie in the choice of the following three definitions:
\begin{itemize}
\item {\bf Definition of Hamiltonian vector field:} On a symplectic manifold $(P,\omega)$,
the Hamiltonian vector field associated to a function $H$ is given by the formula
$$
\omega(X_H,\cdot) = dH \, (\text{ resp. } \, \omega(X_H,\cdot ) = - dH),
$$
\item{\bf Compatible almost complex structure:} In both conventions, $J$ is compatible to $\omega$ if
the bilinear form $\omega(\cdot, J \cdot)$ is positive definite.
\item{\bf Canonical symplectic form:} On the cotangent bundle $T^*N$, the canonical symplectic form
is given by
$$
\omega_0 = dq \wedge dp, \, (\text{resp. }\,   dp \wedge dq).
$$
\end{itemize}

In addition, we would like to take
\be\label{eq:CRJH}
\frac{\del u}{\del \tau} + J\left(\frac{\del u}{\del t} - X_H(u)\right) = 0
\ee
as our basic perturbed Cauchy-Riemann equation on the strip.
Since we work with the cohomological version of Floer complex, we would like to regard
this equation as the \emph{positive} gradient flow of an action functional $\CA_H$ as in \cite{fooo:book2}.
This, under our convention laid out above, leads us to our choice of the action functional associated to Hamiltonian $H$ on $T^*N$ given by
$$
\CA_H(\gamma) = - \int \gamma^*\theta + \int_0^1 H(t, \gamma(t))\, dt,
$$
which is the negative of the classical action functional. With this definition, Floer's continuation
map is defined for the homotopy of Hamiltonian $s \mapsto H^s$ for which the
inequality $\frac{\del H^s}{\del s} \geq 0$ is satisfied, i.e., in the direction for which
the Hamiltonian is \emph{increasing}. (See Section \ref{sec:construction of functor} for the relevant discussion.)

For the kinetic energy Hamiltonian $H = H_g(x)$, we have
\be\label{eq:CA-E}
\CA_H(\gamma_c) = - E_g(c)
\ee
where $\gamma_c$ is the Hamiltonian chord associated to the geodesic $c$ and $E_g(c)$ is the energy of
$c$ with respect to the metric $g$.

\bigskip

\noindent{\bf Acknowledgement:} Y. Bae thanks Research Institute for Mathematical Sciences, Kyoto
University for its warm hospitality. Y.-G. Oh thanks H. Tanaka for his interest in the present work
and explanation of some relevance of Savelyev's work \cite{savelyev} to our construction of $A_\infty$
homotopy associated to the homotopy of Floer continuation maps.
The authors also thank M. Abouzaid for making an interesting comment on
 the relationship between our wrapped Fukaya category and
Sylvan's partially wrapped Fukaya category with stop given by $\Lambda_K^\infty$ in
Fukaya's 60-th birthday conference in Kyoto in February 2019.

\section{Geometric preliminaries}
\label{sec:prelim}

In this section, we consider the cotangent bundle $W = T^*N$ with
the canonical symplectic form
$$
\omega_0 = \sum_{i=1}dq^i \wedge dp_i
$$
which is nothing but $\omega_0 = -d \theta$, where $\theta$ is the Liouville one-form
$\theta = \sum p_idq^i$. (Our convention of the canonical symplectic form
on the cotangent bundle is different from that of \cite{seidel:biased}.)
Then the radial vector field
\be\label{eq:radial}
Z = \sum_{i=1}  p_i \frac{\del}{\del p_i}
\ee
satisfies
\be\label{eq:Z}
Z \rfloor \omega_0 = - \theta,\qquad \CL_Z \omega_0  = \omega_0
\ee
In particular its flow $\phi^t$ satisfies
\be\label{eq:flow}
(\phi^t)^*\omega_0 = e^t \omega_0.
\ee
Therefore $T^*N$ is convex at infinity in the sense of \cite{eliash-gromov}.

Let us consider a Riemannian metric $g_0$ of $N$. We denote by
$$
\sharp: T^*N \to TN, \quad \flat: T^*N \to TN
$$
the `raising' and the `lowering' operations associated to the metric $g_0$.
Then we also equip $T^*N$ with the metric
$
\widehat g_0 = g_0 \oplus g_0^{\flat}
$
with respect to the splitting
$$
T(T^*N) = TN \oplus T^*N
$$
induced from the Levi-Civita connection of $g_0$ on $N$.

\subsection{Tame manifolds and cylindrical adjustments}\label{sec:Tame manifolds and the Busemann function}

Let us consider a noncompact \emph{tame} 3-manifold $N$, which means that there exists an \emph{exhaustion} $\{N_i\}$, a sequence of compact manifolds $N_i$ with $\bigcup_i N_i = N$ such that
\be\label{eq:exhaustion}
N_1 \subset \text{int}(N_2) \subset N_2 \subset \cdots \subset N_i \subset \text{int}(N_{i+1}) \subset \cdots
\ee
and  each $N_{i+1} \setminus \text{int}(N_i)$ is homeomorphic to $\del N_i \times [0,1]$.
It is convenient to consider a compact manifold with boundary $N'$ such that $N$ is homeomorphic to $\Int N'$.
Let us choose a neighborhood $U=U(\partial N')$ of $\partial N'$ inside $N'$, called an \emph{end} of $N$.
Then there is a homeomorphism $\phi:U\to \del N \times (0,+\infty]$, where $\phi(\partial N')=\del N\times\{+\infty\}$ is called the \emph{asymptotic boundary} of $N$.

A typical example of tame manifold is a knot complement $N=M\setminus K$ with an ambient closed 3-manifold $M$. In this case, we can take an exhaustion by simply choosing a sequence of nested tubular neighborhoods of a knot $K$.

Now we focus on the knot complement, i.e. $\del N^{\text{\rm cpt}}$ is a 2-dimensional torus $\\T^2$, specially denoted by $T_i$. Although most of statements in this article also work for arbitrary tame 3-manifolds in general, the results in the sequel \cite{BKO} requires the torus boundary condition of $N^{\text{\rm cpt}}$  in order to exploit hyperbolic geometry using complete hyperbolic metric of finite volume.

We can also consider a more general Riemannian metric $g_0$ of $N$ possibly incomplete. For example, if we consider a knot complement $N=M\setminus K$, a natural choice of such a metric is a restriction of a smooth metric $g=g_M$ of a closed ambient manifold $M$.

\begin{defn}[Cylindrical adjustment]\label{defn:cyl-adj}
We define the cylindrical adjustment $g_0$ of the metric $g$ on $M$
with respect to the exhaustion \eqref{eq:NiNi'} by
\be\label{eq:gi}
g_0 = \begin{cases} g \quad & \text{on } N_i' \\
da^2 \oplus g|_{\del N_i} & \text{on } N_i\setminus K
\end{cases}
\ee
for some $i$, which is suitably interpolated on $N_i \setminus N_i'$ which is fixed.
\end{defn}

Here $a$ is the coordinate for $[0,+\infty)$ for the following decomposition
\[
N=N^{\text{\rm cpt}} \cup_{T} \big(T\times [0,+\infty)\big).
\]
We call the metric $g_0$ {\em an asymptotically cylindrical adjustment of $g$ on $N$}, which are all complete Riemannian metrics of $N$.
We denote $N^{\text{\rm end}} = T\times [0,+\infty)$ in the above decomposition equipped with
the cylindrical metric $g|_{T_i} \oplus da^2$.

We remark that the following property of asymptotically cylindrical adjustments
$g_0$ of smooth metrics $g$ on $M$ restricted to $M \setminus K$
will be important in Section \ref{sec:homlimit} later.
\begin{prop}\label{prop:equivalence} Suppose there is given an exhaustion \eqref{eq:exhaustion}.
	For any choice of two smooth Riemannian metrics $g$ and $g'$ defined on $M$,
denote by $g_0$ and $g_0'$ the cylindrical adjustment made on $N_i$ for a same $i$.
for every pair of $g_0$ and $g'_0$ are Lipschitz equivalent, i.e., there is a constant $C = C(g,g') \geq 1$
such that
	 $$
	\frac{1}{C} g_0 \leq g'_0\leq C g_0
	$$	
on $M \setminus K$.
\end{prop}
\begin{proof} Let $N(K)$ be a tubular neighborhood of $K$. By choosing sufficiently large $i$,
we may assume $M \setminus N_{i_0} \subset N(K)$ and
both $g_0, \, g_0'$ are cylindrical outside $N_{i_0}$.

Using the normal exponential map of $K$, we can parameterize the tubular neighborhood $T(K)$ of
$K$ by $(r, \theta, \varphi)$ where $\varphi \in S^1$ parameterizes the knot $K$ and $(r,\theta)$ is
the polar coordinates of $N_x K$ at $x \in K$ for the metric $g$. Similarly we denote by $(r',\theta',\varphi')$
the coordinates associated to $g'$.

More explicitly, we take a normal frame $\{X,Y\}$ along $K$ along $K$.
We denote by $\exp^{g,\perp}$ the normal exponential map along $K$ of the metric $g$ and
by $\exp^{g',\perp}$ that of $g'$. Using this frame, we have
an embedding of $\iota_g: D^2(\epsilon_0) \times S^1 M$ defined by
$$
\iota_g(r,\theta,\varphi) = \exp_{x(\phi)}^{g,\perp}(r(\cos \theta X(x(\varphi)) + \sin \theta Y(x(\varphi))))
$$
and $\iota_{g_0'}$ in a similar way. Then the composition map
$$
\iota_{g_0'}^{-1} \circ \iota_g: D^2(\epsilon_1) \times S^1 \to D^2(\epsilon_0') \times S^1
$$
is well-defined and smooth if we choose a smaller $\epsilon_1 = \epsilon_1(g,g')$
whose size depends only on $g, \, g'$. In particular, we have
$$
\|D(\iota_{g_0'}^{-1} \circ \iota_g) - id\|_{C^1} < C = C(g,g')
$$
on $D^2(\epsilon_1) \times S^1$ for some constant $C > 0$. This in particular proves
$$
\sup \left\{\frac{g'(v,v)}{g(v,v)}\, \Big \vert\, v \in T(M)|_{N(K)}, \, \|v\|_g = 1\right\}  < C'= C'(g,g').
$$

We recall that the cylindrical adjustment of $g_0$ is defined
$$
g_0 \sim \begin{cases} e^{-2a} da^2 + e^{-2a} \theta^2 + d\varphi^2 \quad & \text{for } 0 \leq r < \epsilon_0 \\
\epsilon_1^2 (da^2 + d\theta^2) + d\varphi^2 \quad & \text{for } a \geq i
\end{cases}
$$
in coordinates with the coordinate change $r = e^{-a}$ and
$\epsilon_0 = e^{-i + 1/2}$ and $\epsilon_1 = e^{-i}$. Exactly the same formula
holds for the cylindrical adjustment of $g_0'$ in the coordinates replaced by
$(a', \theta', \varphi')$, those with primes. In particular all the metric
coefficients appearing in the formulae are exactly the same for both adjustments.
This proves
$$
\sup \left\{\frac{g_0'(v,v)}{g_0(v,v)}\, \Big\vert\, v \in T(M\setminus K)|_{\iota_g^{-1}(N(K)\setminus K)},
\, \|v\|_{g_0} = 1\right\}  < C'= C'(g,g').
$$
for all $i$'s if we choose the tubular neighborhood $N(K)$ of $K$ sufficient small.
This finishes the proof.
\end{proof}

\begin{rem}\label{rem:liouville-cylindrical}
\begin{enumerate}
\item
It may be worthwhile to examine the behavior of Hamiltonian vector field
$X_{H_{g_0}}(q,p)$ on $M \setminus K$ as $q$ approach to $K$. For this purpose, let
$(r, \theta, \varphi)$ be the coordinate system on $N(K)$. The metric $g$ can be written
as
$$
g = dr^2 + r^2 d\varphi^2 + d\theta^2 + o(r^2).
$$
Define a cylindrical adjustment with respect to the coordinate $(a,\theta, \varphi)$ with $r = e^{-a}$ for $a \in [0,\infty)$.
Ignoring $o(r^2)$, the associate Hamiltonian of the cylindrical adjustment (for $r =1$) is given by
$$
H_{g_0} = \frac12 \left(p_a^2 + p_\varphi^2 + p_\theta^2\right)
= \frac12\left(r^2 p_r^2 + p_\varphi^2 + p_\theta^2\right).
$$
Therefore
\beastar
X_{H_{g_0}} & = & p_a \frac{\del}{\del a} + p_\varphi\frac{\del}{\del \varphi} + p_\theta\frac{\del}{\del \theta}\\
& = & - r p_r^2 \frac{\del}{\del p_r} +  r^2 p_r \frac{\del}{\del r} +  p_\varphi\frac{\del}{\del \varphi} + p_\theta\frac{\del}{\del \theta}.
\eeastar
We highlight the last two summands which makes the associated Hamiltonian flow
rotates around the torus with higher and higher speed as $|p_\varphi|^2 + |p_\theta|^2 \to \infty$
around the knot $K$. This asymptotic behavior is different from that of the
Hamiltonian vector field used in defining the partially wrapped Fukaya category whose
stop is given by the Liouville sector $T^*(\del N(K)) \subset T^*(M \setminus N(K))$ whose
horizontal component is parallel to the radial vector field $\frac{\del}{\del r}$.

\item

 The above discussion and construction of wrapped Fukaya category can be extended to
a general tame manifold $N$ whose end may have more than one connected component, as long as
we fix a Lipschitz-equivalence class of metrics on $N$.
An example of such $N$ is the complement $M \setminus L$ where $L$ is a link.
\end{enumerate}
\end{rem}

\subsection{Kinetic energy Hamiltonian and almost complex structures}\label{sec:Ham-J}

We now fix a complete Riemannian manifold with cylindrical end. We denote the resulting
Riemannian manifold by $(N, g)$.
\begin{defn}\label{defn:gi-Hamiltonian}
We denote by $\mathcal H=\mathcal H(T^*N)$
the space of smooth functions $H : T^*N \to \R$ that satisfies
$$
H(q,p) = \frac{1}{2}|p|^2_{g^\flat}
$$
on
$$
\{(q,p):q\in N,\ |p|_{g^\flat}\geq R \} \cup T^*N^{\text{\rm end}}
$$
for a sufficiently large $R> 0$,
where $g^\flat$ is the dual metric of $g$.
We call such map on $T^*N$ {\em admissible Hamiltonian} with respect to the metric $g$.
The {\em Hamiltonian vector field} $X_H$ on $T^*N$ is defined to satisfy
$$
X_H \rfloor \omega =dH
$$
for a given $H\in \mathcal H_i$.
\end{defn}

We next describe the set of adapted almost complex structures.
From now on, we use $g$ instead of $g^\flat$, when there is no danger of confusion.
For each given $R > 0$, the level set
$$
Y_R : = \{|p|_{g} = R\} = H^{-1}({R^2/2})
$$
is a hypersurface of contact type in $T^*N$, and the domain
$$
W_R: =  \{|p|_{g}\leq R\}
$$
becomes a Liouville domain in the sense of \cite{seidel:biased}.

For each $R > 0$, the level set $i_R:Y_R \hookrightarrow T^*N$ admits a contact from
\[
\lambda_R :=- i_R^*\theta
\]
by the restriction. We denote the associated contact distribution by
\[
\xi_R = \ker \lambda_R \subset TY_R.
\]
Note that the corresponding Reeb flow is nothing but a reparametrization of
the $g$-geodesic flow on $Y_R$.

If we denote by $(r,y)$ the cylindrical coordinates given by $r = |p|_{g}$
$$
T^*N \setminus \{0\} \cong ST^*N \times \R_+.
$$
Including the zero section, we have decomposition
$$
T^*N \cong W_1 \cup_{Y_1} \big([0,\infty) \times Y_1\big).
$$
To highlight the metric dependence, we use $Y_{g}, \lambda_{g}, \xi_{g}$ instead of $Y_1, \lambda_1, \xi_R$, respectively.
On $T^*N \setminus 0_{N}$ with metric $g$,
we have the natural splitting
$$
T(T^*N\setminus 0_{N}) \cong \R\cdot \frac{\del}{\del r} \oplus TY_{g_0} \cong \Span\left\{\widetilde X_{g}, \frac{\del}{\del r}\right\} \oplus \xi_{g}
\cong \R^2 \oplus \xi_{g},
$$
where  $\widetilde X_{g}$ is a vector field which generates the $g$-geodesic flow on $Y_{g}$.
We recall that we have a canonical almost complex structure $J_g$ on $T^*N$ associated to
a metric on $N$ defined by
\be\label{eq:sasakian-J}
J_g(X) = X^\flat, \, J_g(\alpha) = -\alpha^\sharp
\ee
in terms of the splitting $T(T^*N) = TN \oplus T^*N$ with respect to the Levi-Civita
connection of $g$. We call this $J_g$ the \emph{Sasakian almost complex structure} on $T^*N$.

\begin{defn}[Admissible almost complex structure]\label{def:Admissible almost complex structure}
We say an almost complex structure $J$ on $T^*N$ is $g$-admissible if
$J = J_{g}$ on $T^*N^{\text{\rm end}}$ and on $\{(q,p) \mid |p|_g \geq R\}$
for the Sasakian almost complex structure $J_{g_0}$ associated to $g_0$
for a metric decomposition $N = N^{\text{cpt}} \cup N^{\text{end}}$ with
$N^{\text{end}} \cong [0, \infty) \times T$ with $T = \del N^{\text{cpt}}$.
Denote by $\CJ_g$ the set of $g$-admissible almost complex structures.
\end{defn}

All admissible almost complex structures satisfy the following important property which enters in the study of Floer theoretic construction on Liouville manifolds in general.

\begin{defn}\label{def:almost_complex_structure} Let $H_g$ be the energy Hamiltonian on $T^*N$ associated metric $g$ on
the base $N$ given above. An almost complex structure $J$ on $T^*N$ is called
\emph{contact type} if it satisfies $(-\theta) \circ J =  dH_g$.
\end{defn}

\begin{rem} Appearance of the negative sign here is because
our convention of the canonical symplectic form on the cotangent bundle is $\omega_0 = -d\theta$.
Compared with the convention of \cite{abou-seidel}, our choice of the one-form $\lambda$ therein
is $\lambda = -\theta$.
\end{rem}

\section{A choice of one-form $\beta$ on $\Sigma$}
\label{sec:minimal}

In Abouzaid-Seidel's construction of wrapped Floer cohomology
given in \cite[Section 3.7]{abou-seidel}, \cite{abou:fiber}, they start from
a \emph{compact} Liouville domain with contact type boundary and consider
the perturbed Cauchy-Riemann equation of the type
\be\label{eq:CRXH}
(du- X_H \otimes \beta)_J^{(0,1)}=0.
\ee
Because of the compactness assumption, their setting
does not directly apply to our current cotangent bundle $T^*N$
where the tame base manifold $N$ is noncompact:
To perform analytical study of the relevant
moduli spaces, the first step is to establish suitable $C^0$-estimate.

\subsection{Co-closed and sub-closed one-forms}\label{sec:sub-closed}

We first recall Abouzaid-Seidel's construction of what they call a \emph{sub-closed one-form}
in the context of Liouville domain with compact contact-type boundary.
For each given $k \geq 1$, let us consider a Riemann surface $(\Sigma,j)$ of genus zero with $(k+1)$-ends.
This is isomorphic to the closed unit disk $\D^2$ minus $k+1$ boundary points ${\bf z}=\{z^0, \dots, z^k\}$ in the counterclockwise way.
Each end admits a holomorphic embedding
\begin{align*}
\begin{cases}
\epsilon^0: Z_+ :=\{\tau \geq 0\} \times [0,1] \to \Sigma;\\
\epsilon^i: Z_- :=\{\tau \leq 0\} \times [0,1] \to \Sigma, \quad\text{ for }i=1,\dots, k
\end{cases}
\end{align*}
preserving the boundaries and satisfying $\lim_{\tau \to \pm \infty}\epsilon^\ell=z^\ell$, for $\ell=0,\dots,k$.
We call the distinguished point at infinity $z^0$ a {\em root}.

For a given weight ${\bf w} = \{w^0,\dots,w^k\}$ satisfying
\be\label{eq:balancing}
w^0 = w^1 + \cdots +w^k,
\ee
a \emph{total sub-closed one-form} $\beta \in \Omega^1(\Sigma)$ is considered in \cite{abou:wrapped} satisfying
\begin{align*}
\begin{cases}
d\beta \leq 0;\\
(\epsilon^\ell)^*\beta=w^\ell dt.
\end{cases}
\end{align*}
This requirement is put to establish both (geometric) energy bound and (vertical) $C^0$-bound.

It turns out that in our case of $T^*(M \setminus K)$ where the base is
\emph{noncompact}, sub-closedness of $\beta$ is not enough to control the behavior of
the Floer moduli space of solutions of \eqref{eq:CRXH}: we need the following restriction
\be\label{eq:ebetaj=0}
d(\beta \circ j) = 0
\ee
in addition to the sub-closedness. We refer readers to Lemma \ref{lem:Deltaa} in Subsection
\ref{sec:horizontalC0} for the reason why such a condition is needed.

\subsection{Construction of one-forms $\beta$}\label{sec:one-form}
For each given $k \geq 1$ equipped with weight datum ${\bf w}=\{w^0,\dots, w^k\}$ satisfying the balancing condition (\ref{eq:balancing}), we will choose a one-form $\beta$ on $\Sigma$ that satisfies
\begin{align}\label{eq:beta-cond}
\begin{cases}
d\beta \leq 0\\
i^*\beta  = 0 \quad & \text{for the inclusion map }\, i: \del \Sigma \to \Sigma\\
d\beta = 0 \quad & \text{near }\, \del \Sigma \\
d(\beta \circ j) = 0 \\
(\epsilon^j)^*\beta = w^j dt \quad & \text{on a subset of $Z_\pm$ where $\pm \tau \gg 0$.}
\end{cases}
\end{align}
For this purpose, we first consider the {\em slit domain} representation of the conformal structure
$(\Sigma,j)$ whose explanation is in order.

Consider domains
\begin{align*}
Z^1 &=\{ \tau+\sqrt{-1}\, t\in \C \mid \tau\in\R,\ t \in [w^0-w^1, w^0] \};\\
Z^2 &=\{ \tau+\sqrt{-1}\, t\in \C \mid \tau\in\R,\ t \in [w^0-w^1-w^2, w^0-w^1] \};\\
&\vdots\\
Z^k &=\{ \tau+\sqrt{-1}\, t\in \C \mid \tau\in\R,\ t \in [0, w^k] \},
\end{align*}
and its gluing along the inclusions of the following rays
\begin{align*}
R^{\ell}=\{\tau+\sqrt{-1}\, t_\ell\in \C \mid \tau\geq s^{\ell},\ t_\ell=w^0-(w^1+\cdots+ w^\ell) \};\\
j_-^\ell:R^\ell\hookrightarrow Z^\ell,\qquad j_+^\ell:R^\ell\hookrightarrow Z^{\ell+1},
\end{align*}
for some $s^\ell\in \R$, $\ell=1,\dots,k-1$.
In other words, for a collection ${\bf s}=\{s^1,\dots, s^{k-1}\}$, the glued domain becomes
\[
Z^{\bf w}({\bf s})=Z^1 \coprod Z^2 \coprod \cdots \coprod Z^k/ \sim.
\]
Here, for $(\zeta,\zeta')\in Z^{\ell}\times Z^{\ell+1}$, $\zeta\sim \zeta'$ means that there exists $r\in R^\ell$ such that $\zeta=j^\ell_-(r), \zeta'=j^\ell_+(r)$.
We may regards
\begin{align*}
Z^0\subset \{\tau+\sqrt{-1}\, t\in \C \mid \tau\in\R,\  t\in[0,w^0]\},
\end{align*}
with $(k-1)$-slits
\begin{align*}
S^{\ell}=\{\tau + \sqrt{-1}\, t_\ell \in \C \mid \tau \leq s^\ell,\ t_\ell=w^0-(w^1+\cdots+w^{\ell}) \}
\end{align*}
where $\ell=1,\dots,k-1$.

Then there is a conformal mapping
\[
\varphi:\Sigma=\D^2\setminus \{z^0,\dots,z^k\}\to Z^{\bf w}({\bf s})
\]
with respect to ${\bf s}$ satisfying that
\begin{enumerate}

\item Let $\partial \Sigma^\ell$ be a connected boundary component of $\D^2\setminus \{z^0,\dots,z^k\}$ between $z^\ell$ and $z^{\ell+1}$, for $\ell\in \Z_{k+1}$. Then
\begin{align*}
\begin{cases}
\varphi(\partial \Sigma^\ell)=w^0\sqrt{-1}+\R &\text{ for }\ell=0;\\
\varphi(\partial \Sigma^\ell)=S^\ell &\text{ for }\ell=1,\dots,k-1;\\
\varphi(\partial \Sigma^\ell)=\R &\text{ for }\ell=k.
\end{cases}
\end{align*}

\item A restriction $\varphi|_{\mathring\Sigma}:\mathring\Sigma \to \mathring Z^0$ is a conformal diffeomorphism.

\item The following asymptotic conditions
\begin{align*}
\begin{cases}
\lim_{\tau\to-\infty}\mathring\varphi^{-1}(\{\tau\}\times(w^0-\sum_{j=1}^{\ell}w^j,w^0-\sum_{j=1}^{\ell-1}w^j))=z^\ell \quad \text{ for } \ell=1,\dots,k;\\
\lim_{\tau\to+\infty}\mathring\varphi^{-1}(\{\tau\}\times(0,w^0))=z^0.
\end{cases}
\end{align*}
\end{enumerate}
Let us denote such a slit domain by $Z^{\bf w}({\bf s})$ or simply $Z$ if there is no confusion.

\begin{figure}[ht]
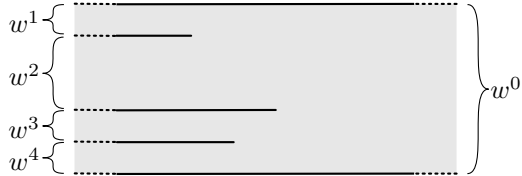
\caption{An example of slit domain}
\end{figure}

Now we consider $dt\in \Omega^1(Z)$ and we define the one-form $\beta$
\begin{lem}\label{lem:beta} Define
\begin{equation}\label{eq:beta}
\beta = \varphi^*dt\in \Omega^1(\Sigma).
\end{equation}
Then $\beta$ satisfies all the requirements in \eqref{eq:beta-cond}.
\end{lem}
\begin{proof} Observe that the holomorphic embedding provides a natural strip-like representation at
each puncture. More precisely, we have
\begin{align*}
\begin{cases}
\varphi \circ \epsilon^0: Z_+\to Z:(\tau,t)\mapsto (\tau+K,w^0 t);\\
\varphi \circ \epsilon^\ell: Z_-\to Z:(\tau,t)\mapsto (\tau-K,w^\ell t+w^0-\sum_{j=1}^{\ell}w^j) \quad\text{ for }\ell=1,\dots, k
\end{cases}
\end{align*}
for some $K\in\R^+$ large enough. So it is direct to check that $(\epsilon^j)^*\beta = w^j dt$ and other conditions in (\ref{eq:beta}). In fact, this form satisfies the stronger condition begin closed rather
than being sub-closed. Furthermore since $\varphi$ is holomorphic, we compute
$$
\beta \circ j = (\varphi^*dt) \circ j = dt d\varphi \circ j = dt j \circ d\varphi = -d\tau \circ d\varphi= -d(\tau \circ \varphi)
$$
which is obviously closed. Therefore we have proved $d(\beta \circ j) = 0$. This finishes the proof.
\end{proof}

Let us wrap up this section by recalling the gluing result of the slit domains.

\begin{prop}
Let $k_1$, $k_2$ be positive integers, and
let ${\bf u}=(u^0,\dots,u^{k_1})$ and ${\bf v}=(v^0,\dots,v^{k_2})$ be weight datum satisfying the balancing condition and $u^0=v^i$ for some $i\geq 1$. Then there is a one-parameter gluing of $Z(\bf u)$ and $Z(\bf v)$ which become a slit domain of
\[
{\bf w}={\bf u}\#^i{\bf v}:=(v^0,\dots,v^{i-1},u^1,\dots, u^{k_1},v^{i+1},\dots,v^{k_2}).
\]
\end{prop}

\part{Construction of a wrapped Fukaya category for $M \setminus K$}

In this part, we carry our construction of our wrapped Fukaya category of
the knot complement $M \setminus K$.

The standard Liouville structure given by the vector field
$$
Z_\lambda = p \frac{\del}{\del p} = \sum_{i=1}^n p_i \frac{\del}{\del p_i}, \quad n = \dim N
$$
has noncompact contact type boundary $\del (DT_{\leq1}^*N)$ for $N = M \setminus K$.
Because of this, construction of wrapped Floer cohomology for the sequence of conormal bundles
of the type $\nu^*T$ with closed submanifold $T \subset N$ meets a new obstruction arising from
the \emph{noncompactness} of $N$: One must examine the \emph{horizontal} $C^0$
bound away from the `infinity' in $N$.
\begin{rem}
For the construction of a wrapped Fukaya category $\CW\CF(T^*N)$, we do not need to restrict
ourselves to the case of knot complement but
can consider the general context of tame 3-manifolds $M$
with several asymptotic boundaries. Since we will not use this general
construction, we do not pursue it in this paper.
\end{rem}

\section{Admissible Lagrangians and associated wrapped Floer complexes}

We first describe the set of admissible Lagrangian submanifolds in such
a manifold $N$  for the set of objects of $\CW\CF(M \setminus K)$.
Since the base manifold $N$ is noncompact, we need to restrict the class of Lagrangian manifolds as follows:
\begin{defn}\label{def:admissible Lagrangian}
A Lagrangian submanifold $L\subset(W,\omega_0)$ is called {\em admissible} if
\begin{enumerate}

\item $L$ is exact and embedded;

\item Its image under the projection $\pi:W \to  N$ is compact in $M \setminus K$.

\item The relative first Chern class $2c_1(W,L)$ vanishes on $H_2(W,L)$ and the second Stiefel-Whitney class $w_2(L)$ vanishes.

\item Under the decomposition $N \cong N^{\text{cpt}} \cup ([0,\infty) \times \T^2)$ and the
associated cylindrical adjustment $g_0$ of $g$, there exists $R>0$ such that
\begin{align*}
L \cap \{ |p|_{g_0} \geq R\} = \Lambda_R \cdot [R,+\infty)
\end{align*}
where $\Lambda_R = L \cap \{ |p|_{g_0} = R\}$ given as in \eqref{eq:gi}.
\end{enumerate}
\end{defn}

For the purpose of controlling the horizontal $C^0$ estimates of solutions of \eqref{eq:CRXH},
we fix a compact exhaustion sequence  $N_1 \subset N_2 \subset \cdots \subset N_i \subset \cdots  $ be
of $N = M \setminus K$. We also take another exhaustion $N_1' \subset N_2' \subset \cdots \subset N_i'
\subset \cdots$ such that
\be\label{eq:NiNi'}
N_1' \subset N_1 \subset N_2' \subset N_2 \subset \cdots \subset N_i' \subset N_i \subset \cdots.
\ee
We mostly denote any such $N_i$ by $N^{\text{\rm cpt}}$ when we do not need to specify the subindex.

We let
\[
W_1 \subset W_2 \subset \cdots \subset W_i \subset \cdots, \quad W_i = T^*N_i
\]
be the exhaustion $T^*(M \setminus K) = \bigcup_{i =1}^\infty W_i$ induced by \eqref{eq:NiNi'}.
We have
\be\label{eq:hatWi}
T^*N = T^*(N_i \cup [0,\infty) \times \del N_i)= W_i \cup T^*N_i^{\text{\rm end}}.
\ee
(See Section \ref{sec:Lagrangian brane} for the implication of the condition (4).)
For any two given admissible Lagrangian $L^0$ and $L^1$, they are
contained inside $W_i$ for some $i$.

We consider a Hamiltonian
$H$ from $\mathcal H$ and define the set
\be\label{eq:frakX}
{\frak X}(H;L^0,L^1) = \{ x:[0,1]\to W \mid
\dot x(t)=X_H(x(t)), \, x(0)\in L^0, x(1)\in L^1\}
\ee
of time-one Hamiltonian chords of $H$.
It can be decomposed into
$$
{\frak X}(H;L^0,L^1) = \coprod_i {\frak X}^i(H;L^0,L^1)
$$
where ${\frak X}^k(H;L^0,L^1)$ is the set of Reeb chords with degree $k$.
By the bumpy metric theorem \cite{abraham} (or rather its version with free boundary condition),
we may assume that all Hamiltonian chords of the Hamiltonian $H \in \CH_g$ are nondegenerate
by considering a generic metric $g$, as long as we consider a countable family of
Lagrangian submanifolds.

\begin{prop} Consider a decomposition $N = N^{\text{\rm cpt}} \cup N^{\text{\rm end}}$ as above
assume that $N^{\text{\rm end}}$ is equipped with the cylindrical metric
$da^2 \oplus g|_{\del N^{\text{\rm end}}}$.
Denote $W=T^*N^{\text{\rm cpt}}$ and let $T^*N = W \cup T^*N^{\text{\rm end}}$ be the associated decomposition.
Suppose $L^0,\, L^1 \subset W$. Then for all $x \in {\frak X}(H;L^0,L^1)$,
we have $\Im  x \subset W$.
\end{prop}
\begin{proof} We recall that the geodesic $x$ satisfies
the (elliptic) second order ODE $\nabla_t \dot x = 0$. This implies that the $a$-component of $x$ in
$N^{\text{\rm end}}$ satisfies $\frac{d^2 a}{dt} = 0$.
Applying the easy maximum principle to $a$, $a$ cannot have maximum at a point $(a,p) \in N^{\text{\rm end}}$.
Since $x(0), \, x(1) \in W$, this finishes the proof.
\end{proof}

We take the {\em action functional}
\begin{align}\label{eq:offshell-action}
\mathcal A(\gamma)= - \int_{0}^{1}\gamma^*\theta  + \int_{0}^{1} H(\gamma(t))\, dt + f^1(\gamma(1))
 - f^0(\gamma(0))
\end{align}
on the path space
$$
\Omega((L^0,L^1) = \{\gamma:[0,1] \to T^*N \mid \gamma(0) \in L^0, \, \gamma(1) \in L^1\}.
$$
\emph{We alert the readers that this is the negative of the classical action functional.}

We assign a Maslov index $\mu(x)$ to each chord $x$.
See Section \ref{sec:Dimensions and orientations of moduli spaces}.
Consider the graded module
\begin{align*}
CW(L^1,L^0; H)&=\bigoplus_{i}CW^i(L^1,L^0;H);\\
CW^k(L^1,L^0;H)&=\bigoplus_{\mu(x)=k}\Z\langle x \rangle
\end{align*}
for $H \in \CH_g$ and $x\in {\frak X}^k(H;L^0,L^1)$. We would like to define a
boundary map $\frak m^1$ on this module.

\begin{rem}\label{rem:notational convention}
Here we adopt the notation $CW(L',L; H)$ for the complex associated to
the geometric path space and to the set of Hamiltonian chords from $L$ to $L'$
which we denote by $\Omega(L,L')$ and ${\frak X}(H;L,L')$ respectively,
following the notational convention of \cite{fooo:book1,fooo:book2}
in that $CW(L',L; H)$ is the \emph{cohomological complex}, not the homological one.
\end{rem}

We take a $t$-dependent family of almost complex structures $\{J_t\}_{t\in[0,1]}$ contained in admissible almost complex structures defined in Definition~\ref{def:Admissible almost complex structure}.
For each given pair $x^0,x^1$ of Hamiltonian chords in $CW(L^1,L^0; H)$ for some $\tau_0 \in \R_+$,
we consider the moduli space $\widetilde {\CM}(x^0;x^1)$ which consist of  maps
$u:\R\times [0,1]\to W$ satisfying the following  perturbed $J$-holomorphic equation
with the boundary and the asymptotic conditions given by
\be\label{eq:CRJXH}
\begin{cases}
\partial_\tau u+J_t(\partial_t u- X_H)=0\\
u(\R\times \{1\})\subset L^0,u(\R\times \{0\})\subset L^1\\
 u(-\infty,t)= x^1(t),\quad u(+\infty,t)= x^0(t).\\
\end{cases}
\ee

\begin{rem} \begin{enumerate}
\item In the point of view of \cite{oh:fredholm}, we may fix the Hamiltonian $H$ and
$J$ and perturb the boundary Lagrangians to achieve this kind of transversality result.
\emph{This is the strategy that we
adopt in the present paper and its sequel \cite{BKO}.}
\item Note that \eqref{eq:CRJXH} is action increasing as $\tau \to \infty$ for the action
functional \eqref{eq:offshell-action}. On the other hand
we put the output at $\tau = +\infty$ and the input at $\tau = -\infty$. Combination of these two
implies that our Floer complex is the \emph{cohomological} version.
\end{enumerate}
\end{rem}

For a generic choice of almost complex structures,
the moduli space $\widetilde {\CM}(x^0;x^1)$ is a manifold of dimension $\mu(x^0)-\mu(x^1)$.
It admits a free $\R$-action as long as $x^0 \neq x^1$. We write ${\CM}(x^0;x^1)$ for the quotient space.
Note that ${\CM}(x^0;x^1)$ is a set of oriented points, when $\mu(x^0)-\mu(x^1)-1=0$.
Here the sign of the rigid solution is determined by the sign of the induced isomorphism
\[
o(x^1) \to o(x^0),
\]
where $o(x^i)$ is an orientation space associated to each Hamiltonian chord which is described in Section \ref{sec:Dimensions and orientations of moduli spaces}.
Then the differential  ${\frak m}^1:CW^{i}(L^1,L^0)\to CW^{i+1}(L^1,L^0)$ is defined by counting rigid solutions of ${\CM}(x^0;x^1)$,
\[
{\frak m}^1(x^1)=\sum_{\mu(x^0)=\mu(x^1)+1}(-1)^{\mu(x^1)}\#\CM(x^0;x^1)x^0,
\]
where $\#$ denotes the signed sum.

For readers' convenience, let us recall from \cite{abou:wrapped} a canonical isomorphism between
the two wrapped Floer complexes before and after conformal fiberwise rescaling.
We denote by $\psi_\lambda$ is time-$log(\lambda)$ Liouville flow on $M$, which can be expressed by $\psi_s(r,y)=(sr,y)$ with respect to the cylindrical coordinate discussed in Section\ref{sec:Ham-J}.
For the simplicity of notation, when a pair $(L^0, L^1)$ and the Hamiltonian $H$ are given, we
just write $CW^*(\psi(L^1),\psi(L^0))$ for any of
$CW^*(\psi(L^1),\psi(L^0);\omega,\frac{1}{w^2}H \circ \psi,\psi^* J_t)$
for $\psi = \psi_w$.

\begin{lem}\label{lem:vertical rescaling}
 Let $g_0$ be the metric on $N$  as above.
Denote by $H$ the kinetic energy Hamiltonian on $W$ of the metric $g_0$ on $N$.
Let $\psi: W \to W$ be a conformally symplectic diffeomorphism satisfying $\psi^*\omega= w^2 \,\omega$ for some non-zero constant $w$. Then there is a canonical isomorphism
\[
CW(\psi): CW^*(L^1,L^0)\cong CW^*(\psi(L^1),\psi(L^0)).
\]
\end{lem}

\section{Floer data for $A_\infty$ structure}

Now consider a $(k+1)$-tuple of admissible Lagrangian
\[
\mathbf L=(L^0,\dots,L^k)
\]
inside $W_i \subset W$. Pick a $(k+1)$-tuple of Hamiltonian chords $(x^0;\bf{x})$
with ${\bf x} = (x^1, \ldots, x^k)$ where
\be\label{eq:chord}
x^0 \in \frak X(H;L^0,L^k);\quad
x^i \in \frak X(H;L^{i-1}, L^{i}) \quad \text{for $i = 1, \ldots, k$}.
\ee

Recall the Riemann surface $(\Sigma,j)\in\mathcal M{}^{k+1}$ is a genus zero disk $(k+1)$ boundary points $\mathbf z=\{z^0,\dots,z^{k}\}$ removed in a counter-clockwise way.
Each end near $z^\ell$ is equipped with a holomorphic embedding
\be\label{eq:strip-end}
\begin{cases}
\epsilon^0:Z_+ \to  \Sigma \quad &\\
\epsilon^i:Z_- \to \Sigma \quad &\text{for } \, i = 1,\dots,k
\end{cases}
\ee
satisfying the boundary and the asymptotic conditions.
We take these choices so that they become
consistent over the universal family
\[
\CS{}^{k+1}:=\{(s,\Sigma,j)\,|\,s\in \Sigma,\  (\Sigma,j)\in \CM^{k+1}\} \to \CM{}^{k+1}
\]
and its compactification
$$
\overline \CS{}^{k+1} \to \overline \CM{}^{k+1}
$$
as in \cite{abou-seidel}. Usage of
universal family will enter in a more significant way when we consider construction of
$A_\infty$ homotopy later.

\begin{defn}[Floer data for $A_\infty$ map]\label{def:Floer data}
A {\em Floer datum} $\mathcal D_{\frak m}=\mathcal D_{\frak m}{(\Sigma,j)}$ on a stable disk $(\Sigma,j)\in \overline{\mathcal M}{}{}^{k+1}$ is the following:
\begin{enumerate}
\item {\bf Weights}: A $(k+1)$-tuple of positive real numbers $\mathbf w=(w^0,\dots,w^k)$ which is assigned to the end points $\mathbf z$ satisfying
$$
w^0=w^1+\cdots+w^k.
$$
\item {\bf One-form}: $\beta\in\Omega^1(\Sigma)$ constructed in Section \ref{sec:one-form} satisfying $\epsilon^{j*}\beta$ agrees with $w^jdt$.
\item {\bf Hamiltonian}: A map $\mathbf H:\Sigma\to\mathcal H_i$ whose pull-back under $\epsilon^j$ uniformly converges to $\frac{H}{(w^{j})^2}\circ \psi_{w^j}$ near each $z^j$ for some $H\in\mathcal H_i$.
\item {\bf Almost complex structure}: A map $\mathbf J:\Sigma\to\mathcal J(T^*N)$ whose pull-back under $\epsilon^j$ uniformly converges to $\psi_{w^j}^*J_t$ near each $z^j$ for some $J_t\in\mathcal J_i$.
\item {\bf Vertical moving boundary}: A map $\eta:\partial \Sigma\to [1,+\infty)$ which converges to $w^j$ near each $z^j$.
\end{enumerate}

\begin{rem}\label{rem:quadratic} The condition (3) is automatically holds for any (fiberwise)globally  quadratic Hamiltonian
such as the kinetic energy Hamiltonian that we are using in the present paper. This is because of the
equality
\be\label{eq:quadratic-H}
\frac{H}{w^2}\circ \psi_w = H
\ee
for such a Hamiltonian. For the consistency with that of \cite{abou:fiber}, we leave the condition as it is.
On the other hand, to exploit this special situation, we extend the function $\eta$ to whole $\Sigma$
so that
$$
\eta\circ \epsilon^i(\infty_i,t) \equiv w^i.
$$
We then choose ${\bf J}$ so that
\be\label{eq:choiceofJ}
(\psi_{\eta})_*{\bf J} = J_g.
\ee
Such a choice will be useful later in our study of $C^0$ estimates.
\end{rem}

Two Floer data $({\bf w}_1,\beta_1,{\bf H}_1, {\bf J}_1, \eta_1)$ and $({\bf w}_2,\beta_2,{\bf H}_2, {\bf J}_2, \eta_2)$ are {\em conformally equivalent} if there exist a constant $C>0$ such that
\begin{align*}
{\bf w}_1 = C {\bf w}_2,\quad \beta_1 = C \beta_2,\quad {\bf H}_1=\frac{{\bf H}_2 \circ \psi^C}{C^2},\quad {\bf J}_1=\psi^{C*}{\bf J_2},\quad \eta_1=C \eta_2.
\end{align*}
\end{defn}

\begin{defn}\label{def:universal Floer data}
A {\em universal} choice of Floer data $\frak{D}_{\frak m}$ for the $A_\infty$ structure consists of $\mathcal D_{\frak m}(\Sigma,j)$ for every $(\Sigma,j)$ in $\overline{\mathcal M}{}^{k}$ for all $k \geq 2$ satisfying the following:
\begin{enumerate}
\item The data is smooth with respect to the moduli space $\mathcal M^k$.
\item The data on $\partial \overline{\mathcal M}{}^{k}$ is conformally equivalent the one on the lower strata $\mathcal M^{k_1}\times \mathcal M^{k_2}$, $k_1+k_2=k+2$.
\item The data are compatible with the gluing process in an infinite order.
\end{enumerate}
\end{defn}

Let us consider a moduli space $\CM_{(\Sigma,j)}^{\bf w}(x^0;\mathbf x)=\CM(x^0;\mathbf x;\mathcal D_{\frak m}(\Sigma,j))$ which consists of a map $u:\Sigma\to T^*N$ satisfying
a perturbed $(j,{\bf J})$-holomorphic equation with a vertical moving boundary condition and a shifted asymptotic condition:
\be\label{eq:duXHJ}
\begin{cases}
(du- X_{\mathbf H}\otimes \beta)_{\mathbf J}^{(0,1)}=0,\\
\text{$u(z)\in\psi_{\eta(z)}(L^i)$, for $z\in\partial \Sigma$ between $z^i$ and $z^{i+1}$ where $i\in \Z_{k+1}$.}\\
u\circ \epsilon^j(-\infty,t)=\psi_{w^j}\circ x^j(t), \text{ for }j=1,\dots,k.\\
u\circ \epsilon^0(+\infty,t)=\psi_{w^0}\circ x^0(t).
\end{cases}
\ee
Here precise meaning of the first equation is
\begin{align}\label{eq:J-holeqn}
(du(z)-X_{\mathbf H(z)}\otimes \beta(z))+{\bf J}(z)\circ(du(z)-X_{{\bf H}(z)}\otimes \beta(z))\circ j=0.
\end{align}
An immediate but important observation is that the pull-back of \eqref{eq:J-holeqn} under $\epsilon^j$ becomes
the standard Floer equation
\[
\partial_\tau u+J_t\left(\partial_t u-X_{\frac{H}{w^j}\circ \psi_{w^j}}(u)\right)=0,
\]
on the strip-like ends
which can be obtained by applying $\psi_{w^j}$ to (\ref{eq:CRJXH}).
Now we consider a parameterized moduli space
\begin{align*}
\mathcal M^{\bf w}(x^0;{\bf x})=\bigcup_{(\Sigma,j)\in
\mathcal M{}^{k+1}} \CM_{(\Sigma,j)}^{\bf w}(x^0;\mathbf x).
\end{align*}
Then by the standard transversality argument we have

\begin{lem}
For a generic choice of universal Floer data $\frak D_{\frak m}$, the moduli space $\CM(x^0;\mathbf x)$ is a manifold of dimension
$\mu(x^0)-\sum_{i=1}^{k}\mu(x^i)-2+k.$
\end{lem}

Due to our construction of the one-forms $\beta$ using the slit domain,
consistency of the Floer data over all strata is automatic which is needed for
the construction of $A_\infty$ structures later.

\section{Construction of $A_\infty$ structure map}\label{sec:energy estimate}

In this section, we would like to construct an $A_\infty$ structure map
\begin{align*}
{\frak m}^k:CW^*(L^{1},L^0;H) \otimes \cdots \otimes CW^*(L^k,L^{k-1};H)\to  CW^*(L^k,L^0;H)[2-k].
\end{align*}
We recall that the module $CW^*(L,L';H)$ is generated by \emph{time-one Hamiltonian chords} of $X_H$ from $L$ to $L'$.

The starting point of compactification of the moduli space relevant to the $A_\infty$ Floer
structure map is to establish a uniform energy bound
and $C^0$ estimates for the solutions of perturbed pseudo holomorphic
equations associated to the corresponding moduli spaces.

\subsection{Energy bound and $C^0$ estimates}
\label{subsec:energyC0}

The uniform energy bound for the solutions of perturbed pseudo holomorphic
equations is necessary for the compactification of the corresponding moduli spaces.

Let $L^i$ be an admissible Lagrangian and $f^i: L \to \R$ be its potential function, i.e.,
a function $\iota^*(-\theta) = df^i$.
We first recall the definition of {\em action} of $x\in{\frak X}(H;L^i,L^j)$
\begin{align}\label{eqn:action}
\mathcal A(x)= - \int_{0}^{1}x^*\theta  + \int_{0}^{1} H(x(t))\, dt + f^j(x(1)) - f^i(x(0))
\end{align}
from \eqref{eq:offshell-action}.

The energy $E(u)$ is defined by
\begin{align*}
\int_{\Sigma}\frac{1}{2}|du-X_H(u)\otimes \beta|_J^2
\end{align*}
for general smooth map $u$.
For any solution $u:\Sigma\to T^*N$ of \eqref{eq:duXHJ},
we have the following estimate
\begin{align}
E(u)&=\int_{\Sigma}u^*\omega-u^*d{\bf H} \wedge \beta{\nonumber}\\
&\leq\int_{\Sigma}u^*\omega-u^*d{\bf H} \wedge \beta-\int_{\Sigma}u^*{\bf H}\cdot d\beta\label{eq:energy_estimate_1}\\
&=\int_{\Sigma}u^*\omega-d(u^*{\bf H}\cdot \beta)\label{eq:energy_estimate_2},
\end{align}
where the inequality in (\ref{eq:energy_estimate_1}) comes from $\mathbf H \geq 0$ and sub-closedness of $\beta$.
By Stokes' theorem with the fixed Lagrangian boundary condition in (\ref{eq:duXHJ}) and $\beta|_{\partial\Sigma}=0$ implies that (\ref{eq:energy_estimate_2}) becomes

\begin{align*}
\mathcal A(x^0)- \sum_{j=1}^{k}\mathcal A(x^j).
\end{align*}

The moduli spaces $\mathcal M^{\bf w}(x^0;\mathbf x)$ admit compactification with respect to the Floer datum $\mathcal D_{\frak m} =\mathcal D_{\frak m}{(\Sigma,j)}$ on a stable disk
$(\Sigma,j)\in\ \mathcal M{}^{k+1}$ stated in Definition \ref{def:Floer data}.

Since our underlying manifold $N$ is noncompact, the $C^0$-estimate for the $J$-holomorphic maps has to be preceded before starting the process of compactification.
This is where the `co-closedness' of $\beta$ enters in an essential way
which is needed for an application of
maximum principle.

The following is the main proposition in this regard whose proof is
postponed until Part 2.

\begin{prop}[Horizontal $C^0$ bound]\label{prop:horizontal}
Let $(\Sigma,j)$ be equipped with strip-like ends at each puncture $z^i$ as before.
Assume the one-form $\beta$ is as in Definition~\ref{def:Floer data}.
Let ${\bf L}=(L^0,\cdots, L^k)$ be $(k+1)$-pair of admissible Lagrangians in $W_\ell\subset T^*N$
for some $\ell \in \N$, see Definition~\ref{def:admissible Lagrangian}.

Let $x^j \in {\frak X}(w^j H; L^{j-1},L^{j})$ for $j=1, \dots, k$ and
$x^0 \in {\frak X}(w^0 H; L^0,L^k)$. Then for any solution $u$ of \eqref{eq:duXHJ}, we have
\be\label{eq:horizontalC0}
\Im u \subset W_\ell.
\ee
\end{prop}

We also need to establish the vertical bound. Such a vertical bound
is established in \cite[section 7c]{abou-seidel} using an integrated form of
the strong maximum principle. Here, partially because similar arguments are needed
for the (horizontally) moving boundary condition, we will provide a more standard argument of
using the pointwise strong maximum principle in Part 2. Both conditions
$d\beta \leq 0$ and $i^*\beta=0$ will be used in a crucial way similarly as
in Abouzaid-Seidel's proof

\begin{prop}[Vertical $C^0$ bound]\label{prop:vertical}
Let $\mathbf x=(x^0,\dots,x^{k})$ be given as in (\ref{eq:chord}). Then
$$
\max_{z \in \Sigma}|p(u(z))| \leq {\frak{ht}}(\mathbf x;H,\{L^i\})
$$
for any solution $u$ of \eqref{eq:duXHJ}.
\end{prop}

\begin{rem} In this remark, we summarize the standard arguments on how the compactness arguments and
dimension counting enter in the construction of moduli operator such as $A_\infty$ maps.
For the practical purpose, we restrict ourselves to the cases of the moduli spaces
$\mathcal M^{\bf w}(x^0;\mathbf x)$ of expected dimension one and zero.
Consider a sequence of pseudo-holomorphic maps $\{u_\nu\}_{\nu\in\N}$ with a sequence of Floer data $\{\mathcal D_\nu\}_{\nu\in\N}$ defined on $\{(\Sigma_{\nu},j_{\nu})\}_{\nu\in\N}$ diverging to one end.
Let $\mathcal D_\infty$ be the corresponding limit of the Floer datum which is defined over a broken stable disk $(\Sigma_\infty,j_{\infty})=(\Sigma_{\mathrm I},j_{\mathrm I})*_{n+1}(\Sigma_{\rm II},j_{\rm II})$. Here $*_{n+1}$ means $0$-th puncture of $\Sigma_{\rm I}$ and $(i+1)$-th puncture of $\Sigma_{\rm II}$ correspond to the breaking point.
We already mentioned about the $C^0$-estimate of the moduli space $\mathcal M^{\bf w}(x^0;\mathbf x)$.
If the gradient of $\mathcal M^{\bf w}(x^0;\mathbf x)$ is not uniformly bounded, then we have sphere or disk bubbling phenomenon. But these cannot happen since the symplectic manifold $(T^*N,\omega)$ and Lagrangian submanifold $(L^i,\iota^*\theta)$ are exact.
Then by Arzel\`a-Ascoli theorem, there is a subsequence of $\{u_\nu\}$ which $C^\infty_{loc}$-converges to $u_\infty$ with Floer datum $\mathcal D_\infty$.
The uniform energy estimate guarantees that the broken strip (or points) should be mapped to a Hamiltonian chord with a matching weight condition.
So the boundary points correspond to broken pseudo-holomorphic maps and each component of the broken map is a part of zero-dimensional moduli space.
Note also that the moduli space $\mathcal M^{\bf w}(x^0;\mathbf x)$ of dimension zero has empty boundary, and hence itself is a finite set.
\end{rem}

\subsection{Definition of $A_\infty$ structure map}

Construction of the $A_\infty$ structure map $\frak m= \{\frak m^k\}$ proceeds in two steps as in \cite{abou:fiber}:

First we consider the moduli space $\mathcal M^{\bf w}(x^0;{\bf x})$ which consist of
solutions $u:\Sigma\to T^*N$ of (\ref{eq:duXHJ}) satisfying the asymptotic conditions
\begin{align*}
(\psi_{w^0}(x^0);\psi_{w^1}(x^1),\dots,\psi_{w^k}(x^k))
\end{align*}
at each ends ${\bf z}=(z^0; z^1, \dots , z^k)$ of $\Sigma$, where
\begin{align*}
\begin{cases}
\psi_{w^0}(x^0) \in CW^*(\psi_{w^0}(L^k),\psi_{w^0}(L^0)); \\
\psi_{w^j}(x^j) \in CW^*(\psi_{w^j}(L^{j}),\psi_{w^{j}}(L^{j-1})) \quad \text{ for } j=1,\dots, k
\end{cases}
\end{align*}
for some Floer datum $\mathcal D=\mathcal D_\frak m(\Sigma,j)$ given in Definition~\ref{def:Floer data}.
Here ${\bf w} = (w^0; w^1, \ldots w^k)$ is the given tuple of weights of asymptotic
conformal shifting of boundary Lagrangians.

The count of such elements directly defines a map
\beastar
\frak m_{\mathcal D}^k & : & CW^*(\psi_{w^1}(L^{1}),\psi_{w^1}(L^{0})) \otimes \cdots \otimes CW^*(\psi_{w^k}(L^{k}),\psi_{w^k}(L^{k-1}))\\
&{}& \qquad \qquad \longrightarrow CW^*(\psi_{w^0}(L^{k}),\psi_{w^0}(L^0))[2-k].
\eeastar
Then we compose the tensor of vertical scaling maps
\[
CW(\psi): CW^*(L,L')\cong CW^*(\psi(L),\psi(L')).
\]
in Lemma~\ref{lem:vertical rescaling} and pre-compose its inverse to $\frak m^k_{\mathcal D}$, i.e.,
the map $\frak m^k$ is defined by
\be\label{eq:mk}
\frak m^k = (CW(\psi))^{-1} \circ \frak m_{\mathcal D}^k \circ (CW(\psi))^{\otimes k}
\ee
In conclusion, we have
\begin{align}\label{eqn:structure map m}
{\frak m}^k(x^1 \otimes \cdots \otimes x^k)=\sum_{x^0} (-1)^\dagger\#\CM^{\bf w}(x^0;\mathbf x)\, x^0,
\end{align}
where $\dagger=\sum_{i=1}^k i \cdot \mu(x^i)$ and $\#$ denotes the signed count of zero dimensional
moduli space $\CM^{\bf w}(x^0;\mathbf x)$.

\begin{rmk}
Note that the $A_\infty$ structure map $\{\frak m^k \}_{k \in \N}$ is defined on the chain complex generated by time-1 Hamiltonian chords of the originally given $H$, while the intermediate maps
$\frak m^k_{\mathcal D}$ is defined on the complex generated by time-one Hamiltonian chords of
the weighted Hamiltonian $w^iH$.
\end{rmk}

Now suppose that the moduli space $\CM^{\bf w}(x^0;\mathbf x)$ is one dimensional. Then by the standard argument in Gromov-Floer compactification, the boundary strata $\partial \overline{\CM}(x^0;\mathbf x)$ consist of
\[
\coprod_{\bar x} \overline{\mathcal M}^{\bf w}(x^0;{\bf x^1}) \times
\overline{\mathcal M}^{\bf w}(\bar x;{\bf x^2}).
\]
Here, $0 \leq n \leq k-m$ and
\begin{align*}
\bar x \in \frak{X}(|{\bf w^2}|H; L^{n+m},L^{n}),\quad |{\bf w^2}|=w^{n+1}+\cdots+w^{n+m};\\
{\bf x^1}=(x^1,\dots, x^{n}, \bar x, x^{n+m+1},\dots, x^{k}),\quad {\bf w^1}={\bf w}|_{\bf x^1};\\
{\bf x^2}=(x^{n+1},\dots,x^{n+m}),\quad {\bf w^2}={\bf w}|_{\bf x^2}.
\end{align*}
We conclude the following relation. Here we follow the sign convention from \cite{seidel:book}, see also Section~\ref{sec:Dimensions and orientations of moduli spaces}.
\begin{prop}\label{eqn:A_infty relation}
The maps $\{\frak m^k \}_{k\in\N}$ define an $A_\infty$ structure i.e., satisfy
\[
\sum_{m,n}(-1)^\ddagger\frak m^{k-m+1} (x^1,\dots, x^{n},\frak m^{m} (x^{n+1},\dots,x^{n+m}),x^{n+m+1},\dots,x^{k})=0,
\]
where $\ddagger=\ddagger_n=\sum_{i=1}^{n}\mu(x^i)-n$.  We denote the resulting category  by
$$
\mathcal{WF}(T^*(M \setminus K);H_{g_0}).
$$

\end{prop}

\section{Construction of $A_\infty$ functor}
\label{sec:construction of functor}

In this section, we consider a pair of metrics $g, \, g'$ on $M$ with $g \geq g'$.

We shall construct a (homotopy) directed system of $A_\infty$ functors
\[
\mathcal F_\lambda:\CW\CF (M\setminus K; H_{g_0}) \to \mathcal{WF} (M\setminus K; H_{g_0'})
\]
for any \emph{monotone} path $\lambda:[0,1] \to \mathcal C(N)$ with $\lambda(0) = g, \, \lambda(1) = g'$.
By Definition \ref{def:admissible Lagrangian} of the objects in the wrapped Fukaya category, there is a natural inclusion
$$
Ob(\CW\CF (M\setminus K; H_{g_0})) \hookrightarrow Ob(\CW\CF (M\setminus K;H_{g_0'})).
$$
We take it as data of the maps
\[
Ob(\mathcal F_\lambda):Ob(\CW\CF (M\setminus K;H_{g_0})) \to Ob(\CW\CF (M\setminus K;H_{g_0'})).
\]

Note that even though $L^0,L^1$ are objects of $\CW\CF (M\setminus K;H_{g_0})$ and hence of
$\CW\CF (M\setminus K;H_{g_0'})$, the corresponding morphism spaces could be different. In fact, the Hamiltonian used for
$\CW\CF (M\setminus K)$ is $H_{g_0}$ and that for $\CW\CF (M\setminus K)$ is $H_{g_0'}$ and so
the corresponding morphism spaces are $CW^*(L^0,L^1;H_{g_0})$ and $CW^*(L^0,L^1;H_{g_0'})$ respectively.
Then we will show that the quasi-equivalence class of the $A_\infty$ category
$
\CW \CF(T^*(M\setminus K);H_{g_0})
$
does not depend on $g$. We denote the $A_\infty$ category by
$$
\CW \CF_g(M \setminus K):= \CW \CF(T^*(M\setminus K);H_{g_0})
$$
suppressing $g$ from its notation.

More precisely, we will construct a  $A_\infty$ functor $\frak f_{g'g} = \{\frak f_\lambda^k\}$ with
$$
\frak f_\lambda^k: CW^*(L^{1},L^0;H_{g_0}) \otimes\ldots \otimes CW^*(L^{k},L^{k-1};H_{g_0}) \to CW^*(L^{k},L^{0}; H_{g_0'})
$$
associated to a homotopy of metrics from $g$ to $g'$. Such an $A_\infty$ functor will be
constructed by an $A_\infty$ version of Floer's continuation map under the
homotopy of associated Hamiltonians $H_{g_0}$ to $H_{g_0'}$.

\begin{rem}
It turns out that this construction of $A_\infty$ morphism or $A_\infty$ functor
under the change of Hamiltonians, at least in the form given in the present paper,
 has not been given in the existing literature as far as we are aware of.
(See \cite{savelyev}, however, for some construction which should be relevant to such a study.)
It took us some effort and time to arrive at the right definition we are presenting here.
At the end of the day, our definition is motivated by the construction given in \cite[Section 4.6]{fooo:book1}
which defines an $A_\infty$ morphism under its Hamiltonian isotopy in the context
of \emph{singular chain complexes}
associated to a given Lagrangian submanifold $L$ in the Morse-Bott context.
\end{rem}

\subsection{Moduli space of time-allocation stable curves}

In this subsection, we recall the moduli space, denoted by ${\overline \CN}{}^{k+1}$ in \cite{fooo:book2},
of a decorated stable curves of genus 0. This moduli space was used in the construction of
$A_\infty$ homomorphisms therein which we generalize for the construction of $A_\infty$ functors.
(See also \cite{abou-seidel} for the similar moduli space named as
the `popsicle moduli space' for the more elaborate version thereof.)

Let $\MM{}^{k+1}$ be the moduli space of
$(\Sigma;{\bf z})$ where $\Sigma$ is a genus zero bordered
Riemann surface and ${\bf z}=(z^0,\dots,z^k)$ are the boundary punctured points,
ordered anti-clockwise, such that $(\Sigma;{\bf z})$ is
stable. Let $\Sigma = \bigcup \Sigma_i$ be the decomposition into the irreducible components.
The stability implies that there is no sphere component, since we do not put any interior
marked points.

We define a partial order on the set of irreducible components.
Let $\{\Sigma_\alpha\,|\,\alpha\in\frak A\}$ be the set of components of $\Sigma$ and,
by the definition of $\mathcal M{}^{k+1}$, it admits a rooted tree structure, see Figure 7.1.21 in \cite{fooo:book2}.
We assign a partial order $\prec$ on $\frak A$ with respect to the rooted tree structure as follows:

\begin{defn}\label{def:partial order}
If every path joining $\Sigma_{\alpha_1}$ to the rooted component $\Sigma_{\alpha_0}$, corresponding to $z_0$, intersect with $\Sigma_{\alpha_2}$, then we write
$\alpha_1\prec \alpha_2$.
\end{defn}

We recall the following definition of time-allocation.

\begin{defn}[Definition 7.1.53 \cite{fooo:book2}]\label{def:rho}
Let $(\Sigma,\prec)$ be such a pair. We define the {\it time allocation}
$\frak \rho:\frak A \to [0,1]$ to $(\Sigma,\prec)$ so that if $\alpha_i \prec \alpha_j$ then $\rho(\alpha_i) \le \rho(\alpha_j)$.
\end{defn}
See Figure 7.1.19 \cite{fooo:book2}.
\begin{defn}[Definition 7.1.54 \cite{fooo:book2}]\label{def:parameterN}
For $k \geq 2$, we define ${\overline \CN}{}^{k+1}$
to be the set of pairs $((\Sigma ; {\bf z}), \rho)$ of $(\Sigma;{\bf z}) \in {\overline \CM}{}^{k+1}$ and
its time allocation $\rho$.
\end{defn}

We equip ${\overline \CN}{}^{k+1}$ with the topology induced from
that of ${\overline \CM}{}^{k+1}$ in an obvious way.
We now describe the stratification of ${\overline \CN}{}^{k+1}$.
Following \cite{fooo:book2}, we denote
$$
\text{\bf N} = (\Sigma,{\bf z},\rho)
\in \mathcal N{}^{k+1}.
$$
We consider the union of all irreducible components $\Sigma_i$ with $\rho_i=\rho(i) \in (0,1)$.
Let us decompose and label the index set $\frak A$ of ${\bf N}\in \mathcal N^{k+1}$ as follows:
\begin{align}\label{eqn:interior vertex data}
\frak A &= \rho^{-1}(0) \sqcup \rho^{-1}((0,1)) \sqcup \rho^{-1}(1);\\
\rho^{-1}(0) &= \{\frak m_1, \dots, \frak m_\lambda \};\nonumber\\
\rho^{-1}((0,1)) &= \{\frak f_1,\dots, \frak f_j \};\nonumber\\
\rho^{-1}(1) &= \{\frak m'_1 ,\dots, \frak m'_i\}.\nonumber
\end{align}
Also denote
Note that $\ell$ could be $0$, and if $\rho^{-1}(1)$ is non-empty then $\bigcup_{\alpha\in \rho^{-1}(1)}\Sigma_\alpha$ is connected and contains $z_0$.

\begin{defn}\label{def:ribbon graph} The combinatorial type $\Gamma=\Gamma({\bf N})$ of $\text{\bf N}$ is a ribbon graph with decorations as follows:
\begin{itemize}
\item Interior vertices $V^{\rm int}$ of $\Gamma$ correspond to the components set $\alpha\in\frak A$.
\item Exterior vertices $V^{\rm ext}$ of $\Gamma$  correspond to the punctured points $\bf z$.
\item The matching condition of components determines interior edges of $\Gamma$.
\item An exterior edge is assigned when a component contains $\bf z$.
\item The ribbon structure is determined by the cyclic order of the marked or
singular points on the boundary of each component.
\item There is a decomposition
\begin{align}\label{eqn:int vertex decomposition}
V^{\rm int}(\Gamma) = V^{\frak f}(\Gamma) \sqcup V^{\frak m}(\Gamma) \sqcup V^{\frak m'}(\Gamma)
\end{align}
with respect to (\ref{eqn:interior vertex data}).
\end{itemize}
We denote by $G^{k+1}$ the set of such combinatorial types of $\overline{\mathcal N}{}^{k+1}$.
\end{defn}

Let $\mathcal N_\Gamma=\{{\bf N}\in \overline{\mathcal N}{}^{k+1}:\Gamma({\bf N})=\Gamma\}$, then
we have the decomposition
$$
{\overline \CN}{}^{k+1} = \bigsqcup_{\Gamma \in G{}^{k+1}} \CN_\Gamma.
$$
Let $\Gamma_0$ be the graph that has
only one interior vertex.
For each $k$, there are three different types of them as in (\ref{eqn:int vertex decomposition}).

\begin{lem}[Lemma 7.1.55 \cite{fooo:book2}]\label{lem:dim of Nspace} For any $\Gamma \in G{}^{k+1}$
with $|V(\Gamma)|>1$ then $\mathcal N_\Gamma$
is diffeomorphic to $D^{|\Gamma|}$ where
\be\label{eq:|Gamma|}
|\Gamma| := k-1 - |V(\Gamma) \setminus V^{\frak f}(\Gamma)|.
\ee
\end{lem}

\begin{prop}[Proposition 7.1.61 \cite{fooo:book2}]\label{prop:struc of Nspace} $\mathcal N{}^{k+1}$ has a structure of
smooth manifold (with boundary or corners) that is compatible
with the decomposition (according to the combinatorial types).
\end{prop}

It was shown in the proof of this proposition in \cite{fooo:book2}, there exists a
map
$$
\frak I: {\overline \CM}{}^{k+1} \times [0,1] \to {\overline \CN}{}^{k+1}
$$
that gives a homeomorphism onto the set of all $\text{\bf N}$ whose associated graph
$\Gamma$ has only one interior vertex and its time allocation $\rho$ is constant.

We then quote the following basic structure theorem of ${\overline \CN}{}^{k+1}$
from \cite{fooo:book2}

\begin{thm}[Theorem 7.1.51 \cite{fooo:book2}]\label{thm:f-boundary} For each $k \in \N$,
${\overline \CN}{}^{k+1}$ carries the structure of a cell complex which is
diffeomorphic to $k-1$ dimensional disc $D^{k-1}$ such that its boundary is decomposed to the union of cells
described as follows:
\begin{enumerate}
\item $\displaystyle{\MM^{(i+1,\dots,i+\ell)} \times
{\overline \CN}{}^{(1,\dots,i,*,i+\ell+1,\dots,k)}}$.
\item $\displaystyle{\prod_{i=1}^{m}{\overline \CN}{}^{(\ell_{i-1}+1,\dots,\ell_{i})}
\times \MM^{m+1}}$, where $1 = \ell_0$, $\ell_m = k$,
$\ell_i < \ell_{i+1} -1$.
\end{enumerate}
In the above, $\mathcal M^{(i_1,\dots,i_k)}$ is $\mathcal M^k$ with a new label $(i_1,\dots,i_k)$ on $k$ punctures. The same for $\mathcal N$.
\end{thm}
We remark that Case (1) contains the case $i=1$, $1+\ell =k$.
This is the case of
$\MM{}^{k+1} \times \mathcal N^{1+1} \cong \MM{}^{k+1}$.
It corresponds to the
case when time allocation $\rho$ is 1 everywhere.
Case (2) contains the case $m = k$.
This is the case of $\mathcal N^{1+1} \times (\MM{}^{k+1})'$. It corresponds to the case when the time allocation $\rho$ is $0$
everywhere.

Similarly as in the construction of $A_\infty$ homomorphism given in Chapter 7 \cite{fooo:book2},
we will use ${\overline \CN}{}^{k+1}$ for the construction of $A_\infty$ functors.

\subsection{Definition of $A_\infty$ functor and energy estimates}
\label{subsec:functor}
Firstly, consider cylindrical adjustments $g_0, \, g_0'$  on $N = M \setminus K$
of Riemannian metrics defined on $M$ as introduced in Section
\ref{sec:Tame manifolds and the Busemann function}. We assume $g_0 \geq g_0'$ and
consider a homotopy $\lambda:[0,1]\to \mathcal C(N)$ of Riemannian metrics
from $g_0'$ to $g_0$ given by
\[
\lambda: r \mapsto g(r) = g_0 + r(g_0'-g_0)
\]
(or by any other homotopy $r \mapsto g(r)$ connecting $\lambda(0) = g_0$ and $g_0'$). We denote by
$H_\lambda$ the time-dependent Hamiltonian associated to $\lambda$ defined by
$H_\lambda(s,x) = H_{\lambda(s)}(x)$.
We will also impose additional monotonicity restriction
\be\label{eq:gigj}
g_0 \geq g_0' \quad \text{or equivalently } \, H_{g_0} \leq H_{g_0'}
\ee
and $\lambda$ is monotone,
which is needed, for example, for the energy estimates for the
perturbed Cauchy-Riemann equation relevant to the definition of
Floer's continuation map for the wrapped Fukaya category of the cotangent bundle
in general. We refer to Remark \ref{rem:reason} to see how this monotonicity
condition enters in a crucial way.

\begin{rem}\label{rem:scaling}
For our purpose, we will obtain such an inequality by multiplying a sufficiently large constant
$\lambda> 0$ to one of the metrics say $g_1$. Because the base manifold $M \setminus K$ is not compact,
it is possible to achieve $g_0' \leq \lambda g_0$ everywhere on $M \setminus K$,
only when $g_0$ and $g_0'$ are Lipschitz-equivalent, i.e., when there is a constant $C = C(g,g) '> 0$
such that $\frac1C g_0 \leq g_0' \leq C g_0$. This is precisely the reason why we
look at those metrics $g, \, g'$ on $M \setminus K$ that are smoothly extendable to whole $M$
in the present paper.
We refer readers to Proposition \ref{prop:equivalence} for the latter statement.
\end{rem}
For given $\lambda =\{g(r)\}_{r \in [0,1]}$, we consider the fiber bundles
$$
\mathcal H_\lambda \to [0,1], \quad \mathcal H_{g(r)}=\mathcal H(T^*N,\widehat g(r))
$$
and
$$
\mathcal J_\lambda \to [0,1], \quad \mathcal J_{g(r)}=\mathcal J(T^*N,\widehat g(r)).
$$
We will also use the elongation function $\chi : \R \times [0,1]$ satisfying
\beastar
\chi(\tau) =
\begin{cases}
0 \quad \tau \leq 0; \\
1 \quad \tau \geq 1,
\end{cases}\quad 0 \leq \chi' \leq 2.
\eeastar

The morphism of $A_\infty$ functor $\mathcal F_\lambda$ consists of the following data of $A_\infty$ homomorphism
\begin{align*}
{\frak f}^k_\lambda:CW^*(L^{1},L^{0};H_{g_0}) \otimes \cdots \otimes CW^*(L^{k},L^{k-1};H_{g_0})\to  CW^*(L^k,L^0;H_{g_0'})[1-k].
\end{align*}
In order to construct ${\frak f}^k$ we need to use another moduli space of perturbed $J$-holomorphic curves of genus zero.
Here we consider the case where components are perturbed $J$-holomorphic with respect to the compatible almost complex structures and Hamiltonian functions which could {\em vary} on the components. We basically follow the construction described in \cite[Section 4.6]{fooo:book1} in the current wrapped context.

Note that each component $\Sigma_\alpha$ is an element of $\mathcal M^{\ell+1}$ for some $\ell\geq 1$, and two components have a matching asymptotic condition at most {\em one} point. Since our underlying manifold is exact and the Lagrangian submanifolds are exact, there is no possibility for sphere bubbles and disk bubbles.

Now we are ready to give the definition of $A_\infty$ functor. We start with
listing the data necessary for the construction. We would like to highlight
that in the item ($3(\rho)$) below, we put the time-dependent Hamiltonian $H_\lambda$
while we put the time-independent Hamiltonian in Definition \ref{def:Floer data}.

\begin{defn}[Floer data for $A_\infty$ functor]\label{def:Floer data for functor}
 A {\em Floer datum} $\mathcal D_{\frak f}=\mathcal D_{\frak f}(\Sigma,j;\rho)$ for $(\Sigma,j;\rho) \in \overline{\mathcal N}{}^{k+1}$ is defined by  replacing the properties (3), (4) in Definition \ref{def:Floer data} as follows:
\begin{enumerate}
\item {\bf Weights}: A $(k+1)$-tuple of non-negative integers $\mathbf w=(w^0,\dots,w^k)$ which is assigned to the marked points $\mathbf z$ satisfying
$$
w^0=w^1+\cdots+w^k.
$$
\item {\bf One forms}: A collection of one-forms $\beta_\alpha \in\Omega^1(\Sigma_\alpha)$ constructed in Section \ref{sec:one-form} satisfying $\epsilon^{j*}\beta$ agrees with $w^jdt$.
\item[($3_{\rho}$)] {\bf Hamiltonians}: a collection $\mathbf H^{\rho}$ of maps $\mathbf H_\alpha:\Sigma_\alpha\to\mathcal H_{\rho(\alpha)}$ assigned to each irreducible component $\alpha\in\frak A$
whose pull-back under $\epsilon^{j_\alpha}$ uniformly converges to $\frac{H^{j_\alpha}_\lambda}{(w^{j_\alpha})^2}\circ \psi_{w^{j_\alpha}}$ near each $z^{j_\alpha}$
for the path
$$
H^{j_\alpha}_\lambda:[0,1] \to \mathcal H_{\rho(\alpha)}
$$
dictated as described above.
\item[($4_{\rho}$)] {\bf Almost complex structures}: $\mathbf J^{\rho}$ which consist of maps $\mathbf J_{\alpha}:\Sigma_\alpha \to\mathcal J_{\rho(\alpha)}$ for each $\alpha\in\frak A$, whose pull-back under $\epsilon^{j_\alpha}$ uniformly converges to $\phi_{w^{j_\alpha}}^*J^{j_\alpha}$ near each $z^{j_\alpha}$ for the path
$$
J^{j_\alpha}_\lambda: [0,1] \to \mathcal J_{\rho(\alpha)}
$$
associated to $\lambda$.
\item[(5)] {\bf Vertical moving boundary}: A map $\eta:\partial \Sigma\to [1,+\infty)$ which converges to $w^j$ near each $z^j$.
\end{enumerate}

A universal choice of Floer data $\frak D_{\frak f}$ for the $A_\infty$ functor consists of the collection of $\mathcal D_{\frak f}$ for every $(\Sigma,j;\rho) \in \coprod_{k\geq1}\overline{\mathcal N}{}^{k+1}$ satisfying the corresponding conditions as in Definition \ref{def:universal Floer data} for $\overline{\mathcal N}{}^{k+1}$ and its boundary strata.
\end{defn}

Now we consider the system of $\mathbf H^{\rho}$ perturbed $\mathbf J^{\rho}$ holomorphic maps $\{u_\alpha:\Sigma_\alpha \to T^*N\}_{\alpha\in\frak A}$ for the data $\mathcal D_{\frak f}(\Sigma_\alpha,j_\alpha;\rho)$ and  a $(k+1)$-pair of Hamiltonian chords
\begin{align*}
\mathbf x=(x^1,\dots, x^k)&\in {\frak X}(H_{g_0};L^{0},L^{1}) \times \cdots \times {\frak X}(H_{g_0};L^{k-1},L^{k});\\
y&\in {\frak X}(H_{g_0'};L^0,L^k)
\end{align*}
satisfying stability, shifted asymptotic conditions, and the perturbed $J$-holomorphic equation
with respect to $\mathcal H_s$ and $\mathcal J_s$ for each parameter $s$ as follows.

With this preparation, we describe the structure of
the collection of maps $\{u_\alpha\}_{\alpha \in \frak A}$.

As a warm-up, we first consider the case $k=1$, i.e., with one input and one output puncture
whose domain is unstable. When the domain is smooth, it is isomorphic to
$\R \times [0,1]$. For a given pair of Hamiltonians $H^-, \, H^+$ and homotopy
$\{H_s\}_{s \in [0,1]}$ with $H_0 = H^-, \, H_1 =  H^+$,
we consider the \emph{non-autonomous} Cauchy-Riemann equation
\be\label{eq:CR-chi}
\begin{cases}
\frac{\del u}{\del \tau} + J_{\chi(\tau)}\left(\frac{\del u}{\del t} - X_{H_{\chi(\tau)}}(u)\right) = 0\\
u(\tau,0) \in L, \, u(\tau,1) \in L'.
\end{cases}
\ee
\begin{rem}\label{rem:reason} This appearance of non-autonomous Cauchy-Riemann equation in our
construction of $A_\infty$ functor is the reason why we imposed the
monotonicity hypothesis \eqref{eq:gigj} for the energy estimates.
We refer to \cite{oh:chain} for the energy formula
for the non-autonomous equation:
\bea\label{eq:lemma4.1}
\int \left|\frac{\del u}{\del \tau}\right|_{J_{\chi(\tau)}}^2 \, dt\, d\tau
& = &
 \mathcal A_{H^+}(z^+) - \mathcal A_{H^-}(z^-) \nonumber \\
&{}& -
\int_{-\infty}^\infty \chi'(\tau) \left(\int_0^1
\frac{\del H_s}{\del s}\Big|_{s = \chi(\tau)}(u(\tau,t))dt\right)  d\tau.
\eea
For readers' convenience, we give its derivation in Appendix.
It follows that we have uniform energy bound for general Hamiltonians
on non-compact manifolds \emph{provided
the homotopy $s \mapsto H_s$ is a monotonically increasing homotopy}.
\end{rem}
Denote by $\CN(x';x)$ the moduli space of finite energy solution with $u(-\infty) = x, \, u(\infty) = x'$.
We denote by $\overline{\CN}(x';x)$ its compactification. An element of $\overline{\CN}(x';x)$
is a linear chain of maps elements
$$
u^-_1, \, \ldots, \, u^-_{k^-}, u_0, u^+_1, \, \ldots, \, u^+_{k^+}
$$
such that $u^-_i \in \CM(x^-_{i-1},x^-_i)$ with $x^-_i \in \frak X(H^-;L,L')$ for $0 \leq i \leq k^-$,
$u^+_j \in \CM(x^+_{j-1},x^+_j)$ $x^+_j \in \frak X(H^+;L,L')$ for $0 \leq j \leq k^+$ and $u_0 \in \CN(x^-_{k^-},x^+_0)$. We denote the concatenation of such linear chain by
\be\label{eq:u-chain}
u = (u^-_1, \, \ldots, \, u^-_{k^-}, u_0, u^+_1, \, \ldots, \, u^+_{k^+}).
\ee
Then for given ordered sequence $0 \leq \rho_0 \leq \rho_1 \leq  \ldots \leq \rho_\lambda \leq 1$,
we perform the above construction for each consecutive pair $(H_{\rho_i}, H_{\rho_{i+1}})$
iteratively to form a `stair-case' chain which is a concatenation of
the linear chains $u^i$ over $i = 0, \ldots, \ell$. We denote by $\overline\CN(x';x)$
the set of such stair-case chains.

Now we turn to the case $k \geq 2$.
Let $\Sigma_\alpha$ be an irreducible component of $\Sigma$.
Denote by $\{z_0^\alpha, \ldots, z_{k_\alpha}^\alpha\}$ $k_\alpha \geq 1$ the special points of
$\alpha$ which is the union of nodal points and marked points of $\Sigma$
on $\Sigma_\alpha$.  We denote by $L^\alpha_i$ the Lagrangian submanifold that
we originally put around the vertex $\alpha$ in the chambers given by the dual graph of
$\Sigma$ in the beginning. Denote by $\Sigma_{\alpha^+}$ the component attached to
$z_0^\alpha$. Then we have $\alpha \prec \alpha^+$ and so $\rho(\alpha) \leq \rho(\alpha^+)$.

We now describe the equation for $u^\alpha$ on the strip-like region at $z_0^\alpha$.
For given $\alpha, \, \beta$, we consider the function homotopy $s \mapsto H_s^{(\alpha,\beta)}$ defined by
$$
H^{(\alpha,\beta)}_s = (1-s)H_{\rho(\alpha)} + s H_{\rho(\beta)}
$$
or any homotopy $\{H_s\}_{s \in [0,1]}$ connecting $H_{\rho(\alpha)}$ and $H_{\rho(\beta)}$.
In the strip-like region near $z_0^\alpha$ on $\Sigma_\alpha$,
we consider the non-autonomous Cauchy-Riemann equation on $(-\infty,0] \times [0,1]$
$$
\frac{\del u}{\del \tau} + J_{\chi(\tau_0^\alpha - \tau)}\left(\frac{\del u}{\del t} -
X_{H^{(\alpha,\alpha^+)}_{\chi(\tau_0^\alpha - \tau)}}(u)\right) = 0
$$
for some $\tau^\alpha_0 \leq -1$.

On the other hand, at the input punctures $z_i^\alpha$,
we denote by $\Sigma_{\alpha^-}$ the component attached to $z^i_\alpha$ at the
nodal point $z_0^{\alpha^-}$. Then we put the non-autonomous equation
on $[0, \infty) \subset [0,1]$
$$
\frac{\del u}{\del \tau} + J_{\chi(\tau - \tau_i^\alpha)}\left(\frac{\del u}{\del t} -
X_{H^{(\alpha^-,\alpha)}_{\chi(\tau - \tau_i^\alpha)}}(u)\right) = 0
$$
for some $\tau_i^\alpha \geq 1$.

At each nodal point of $\Sigma$ associated with $\alpha \prec \beta$, we insert a linear chain
associated to $H^- = H_{\rho(\alpha)}, \, H^+ = H_{\rho(\beta)}$ of the type
$$
u^-_1, \, \ldots, \, u^-_{k^-}, u_0
$$
on the strip-like regions of $z_i^\alpha$ with $u^-_i$ as above but $u_0$ is a map
on $[0, \infty) \times [0,1]$ or of the type
$$
u_0, u^+_1, \, \ldots, \, u^+_{k^+}
$$
on the strip-like regions of $z_0^{\beta}$ with $u^+_j$ as above but $u_0$ is a map
on $(-\infty, 0] \times [0,1]$. \emph{We put the chain
at most one of the two regions.}

In summary, we require the map $u_\alpha$ on $\Sigma_\alpha$ to satisfy
\begin{enumerate}
\item
\be\label{eq:CRJHalpha}
(du_\alpha- X_{{\bf H}_{\alpha}}\otimes \beta_\alpha)_{{\bf J}_{\alpha}}^{(0,1)}=0
\ee
for each $\alpha\in\frak A$. Here we emphasize the fact that the Hamiltonian ${\bf H}_\alpha$ is
autonomous away from the strip-like regions of each component. \emph{We do not exclude the
case where the component is of the $\frak m$-type, i.e., one that
satisfies the autonomous equation of $H_{\rho(\alpha)}$.}
\item $u_\alpha\circ \epsilon^{j_\alpha}(-\infty,t)=\psi_{w^{j_\alpha}}\circ x^{j_\alpha}(t)$,
for the input punctures $j_\alpha$ of $\Sigma_\alpha$.
\item $u\circ \epsilon^{0_\alpha}(+\infty,t)=\psi_{w^{0_\alpha}}\circ x^{0_\alpha}(t)$,
for the output puncture $0_\alpha$ of $\Sigma_\alpha$.
\item $\big((\Sigma_\alpha,\mathbf z),(u_\alpha)\big)_{\alpha\in\frak A}$ is stable.
\end{enumerate}

\begin{figure}[ht]
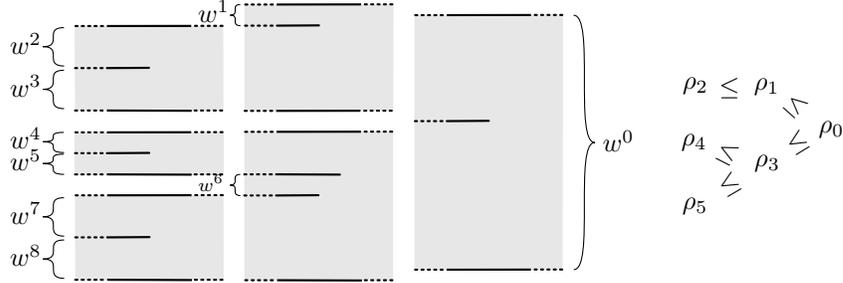
\caption{An example of slit domains for the $A_\infty$ functor}
\end{figure}

To be able to construct the relevant compactified moduli space of solutions \eqref{eq:CRJHalpha},
we need the uniform energy bound and $C^0$ estimates both of which require the monotonicity condition
\eqref{eq:gigj}. The proof of the following energy bound is given by combining the
energy bound obtained in Section \ref{subsec:energyC0} and that of \eqref{eq:lemma4.1} in
Remark \ref{rem:reason}.
\begin{lem} For any finite energy solution of \eqref{eq:CRJHalpha}, we have the energy identity
\be\label{eq:functor-energy}
E(u) \leq  \CA_{H^+}(x^0) - \sum_{j=1}^k \CA_{H^-}(x^j)
- \sum_{\alpha \in \frak A} \int_{-\infty}^\infty \chi'(\tau) \left(\int_0^1
\frac{\del H_{\lambda^\chi}^{\alpha}}{\del s}\Big|_{s = \chi(\tau)}(u(\tau,t))dt\right)  d\tau
\ee

where the map $\lambda^\chi: \R_\pm \to [0,1]$ is the elongated path defined by
$
\lambda^\chi(\tau) = \lambda(\chi(\tau)).
$
In particular, if $\lambda$ is a monotonically decreasing path, i.e, if
$\frac{\del H_{\lambda}}{\del s} \geq 0$, then we have
$$
E(u) \leq  \CA_{H^+}(x^0) - \sum_{j=1}^k \CA_{H^-}(x^j).
$$
\end{lem}
Then we denote the space of such a system of maps satisfying the above conditions by
\begin{align}\label{eqn:moduli space N}
\overline{\mathcal N}{}^{k+1}_{(\Sigma,j;\rho)}(y;\mathbf x)=\overline{\mathcal N}{}^{k+1}(y;\mathbf x;\mathcal D_{\frak f}(\Sigma,j;\rho)).
\end{align}
Consider a parameterized moduli space
\[
\overline{\mathcal N}{}^{k+1}(y;\mathbf x) =\bigsqcup_{(\Sigma,j;\rho)\in \overline{\mathcal N}{}^{k+1}} \overline{\mathcal N}{}^{k+1}_{(\Sigma,j;\rho)}(y;\mathbf x).
\]
We have a forgetful map
$$
{\frak{forget}}: \overline{\mathcal N}{}^{k+1}(y;\mathbf x) \to \overline{\mathcal N}{}^{k+1}
$$
for $k \geq 2$ which induces a decomposition
$$
\overline{\mathcal N}{}^{k+1}(y;\mathbf x) = \bigsqcup_{\Gamma \in G{}^{k+1}} {\CN}_{\Gamma}(y;\mathbf x)
$$
where ${\CN}_\Gamma(y;\mathbf x) = {\frak{forget}}^{-1}(\Gamma)$.

\begin{lem}
For a generic choice of universal Floer data $\frak D_{\frak f}$, the moduli space
${\CN}_\Gamma(y;\mathbf x)$ is a manifold of dimension
$$
\mu(y)-\sum_{i=1}^{k}\mu(x^i) - |\Gamma|
$$
where $|\Gamma|$ is given as in \eqref{eq:|Gamma|}.
\end{lem}

In particular, if $V(\Gamma) = V^{\frak f}(\Gamma)$, i.e.,
there is no ${\frak m}$ or ${\frak m'}$ component, then
the dimension becomes
$$
\mu(y) - \sum_{i=1}^{k}\mu(x^i) -1 +k.
$$
The matrix coefficient of the $A_\infty$ homomorphism
\begin{align*}
{\frak f}^k_\lambda:CW^*(L^{1},L^{0};H_{g_0}) \otimes \cdots \otimes CW^*(L^{k},L^{k-1};H_{g_0})\to  CW^*(L^{k},L^{0};H_{g_0'})[1-k].
\end{align*}
is defined by
\begin{align}\label{eqn:A_infty functor F}
\frak f^k_\lambda(x^1 \otimes \cdots \otimes x^k)=\sum_{y}(-1)^\spadesuit\#\overline{\mathcal N}{}^{k+1}(y;{\bf x})\, y,
\end{align}
where $\spadesuit=\sum_{j=1}^{k}j \cdot \mu(x^j)+k$ and $\#$ is the signed sum as before.
Note that $\#(\overline{\mathcal N}{}^{k+1}(y;\mathbf x))$ becomes zero unless
$$
\mu(y)=\sum_{i=1}^{k}\mu(x^i)+1-k.
$$

Since the equation \eqref{eq:CRJHalpha}
for each given $\rho(\alpha)$ does not involve moving boundary condition but only fixed
boundary condition, the $C^0$-bound for \eqref{eq:duXHJ} still applies to prove the following
$C^0$-bounds.

Let $(\Sigma,j)$ be equipped with strip-like ends
$\epsilon^k: Z_\pm \to \Sigma$ at each puncture $z_k$ as before.
\begin{prop}[Vertical $C^0$ estimates for $\frak f$-components]
\label{prop:vertical-f}
Let $\mathbf x=(x^0,\dots,x^{k})$ be a $(k+1)$-tuple of
Hamiltonian chords $x^j \in {\frak X}( H_{g_0}; L^{j-1},L^{j})$. Then
$$
\max_{z \in \Sigma}|p(u(z))| \leq \max_{r \in [0,1]} \max
\{ {\frak{ht}}(\mathbf x;H_{g_{(r)}},\{L^i\})\mid i = 0, \ldots, k \}
$$
for any solution $u$ of \eqref{eq:CRJHalpha}.
\end{prop}

\begin{prop}[Horizontal $C^0$ estimates for $\frak f$-components]
\label{prop:horizontal-f}
Assume the one-form $\beta$ satisfies \eqref{eq:beta}. Suppose $L^j \subset
W_\lambda$ for all $j =0,\ldots, k$, and let
$x^j \in {\frak X}( H_{g_0}; L^{j-1},L^{j})$ for $j=1, \dots, k$ and
$x^0 \in {\frak X}(H_{g_0'}; L^0,L^k)$ be Hamiltonian chords. Then
\be\label{eq:horizontalC0}
\Im u \subset W_\lambda
\ee
for any solution $u$ of \eqref{eq:CRJHalpha}.
\end{prop}

We now describe the structure of boundary (or codimension one) strata
of ${\overline \CN}{}^{k+1}(y;\mathbf x)$. We first form the union
\[
\overline{\mathcal N}{}^{k+1}(y;\mathbf x) =\bigsqcup_{(\Sigma,j;\rho)\in \overline{\mathcal N}{}^{k+1}} \overline{\CN}{}^{k+1}_{(\Sigma,j;\rho)}(y;\mathbf x).
\]
Then consider the asymptotic evaluation maps
\begin{align*}
&ev^\ell_-: \bigsqcup_{(y;{\bf x})}\overline{\mathcal N}{}^{k+1}(y;\mathbf x)\to{\frak X}(H_{\rho(z^\ell)};L^{\ell-1},L^{\ell})\times [0,1] \quad \text{for }\, \ell = 1, \ldots, k\\
&ev^0_+: \bigsqcup_{(y;{\bf x})}\overline{\mathcal N}{}^{k+1}(y;\mathbf x)
\to{\frak X}(H_{\rho(z^0)};L^{0},L^{k}) \times [0,1]
\end{align*}
given by
\begin{align*}
&ev^\ell_-(\{u_\alpha\}_{\alpha\in\frak A}) = (u\circ \epsilon^\ell(-\infty, \cdot \ ), \rho(z^\ell)) \quad \text{for }\, \ell = 1, \ldots, k\\
&ev^0_+(\{u_\alpha\}_{\alpha\in\frak A}) = (u\circ \epsilon^0(+\infty, \cdot \ ), \rho(z^0))
\end{align*}
where $\rho(z^i) \in [0,1]$ is a $\rho$ value of the component having the $i$-th end.
We also consider similar evaluation maps from
\begin{align*}
&ev^\ell_-: \bigsqcup_{(y;{\bf x})}\overline{\mathcal M}{}^{k+1}(y;\mathbf x)\to{\frak X}(H_{g_0};L^{\ell-1},L^{\ell}) \quad \text{for }\, \ell = 1, \ldots, k\\
&ev^0_+: \bigsqcup_{(y;{\bf x})}\overline{\mathcal M}{}^{k+1}(y;\mathbf x)\to{\frak X}(H_{g_0};L^{0},L^{k}).
\end{align*}

\begin{thm}
Let $g, \, g'$ be a pair of smooth metrics on $M$ satisfying $g \geq g'$, and $\lambda$ be a
monotone homotopy between them with $\lambda(0) = g, \, \lambda(1) = g'$.
Then the definition \eqref{eqn:A_infty functor F}
is well-defined and the maps $\{\frak f^k_\lambda\}$ satisfy the $A_\infty$ functor relation.
\end{thm}

\begin{proof}

Let $\big((\Sigma^{(i)},\mathbf z^{(i)}),(u_\alpha^{(i)}),(\rho_\alpha^{(i)})\big)_{\alpha\in\frak A}^{i\in\N}$ be a sequence of elements of ${\overline \CN}{}^{k+1}(y;\mathbf x)$, then
\begin{enumerate}
\item One of the component $\Sigma_\alpha^{(i)}$ splits into two components as $i\to +\infty$.
\item Two component $\Sigma_{\alpha_1}$ and $\Sigma_{\alpha_2}$ sharing one asymptotic matching condition satisfy
$\lim\rho_{\alpha_1}^{(i)}= \lim\rho_{\alpha_2}^{(i)}$ when $i\to +\infty$.
\item $\lim_{i\to +\infty}\rho_{\alpha}^{(i)}=0$.
\item $\lim_{i\to +\infty}\rho_{\alpha}^{(i)}=1$.
\end{enumerate}
(See Figure 7.1.20 \cite{fooo:book2}.)

Both Type (1) and Type (2) above consists of the following fiber product
$$
{\overline \CN}{}^{k_1+1}(y;{\bf x_1})_{ev^i_+}\times_{ev^0
_-}  {\overline \CN}{}^{k_2 +1}(*;{\bf x}_2)
$$
with $k_1+ k_2 = k$ and with opposite sign and so they cancel each other when $\mu(y)=\sum_{i=1}^{k}\mu(x^i)+1-k$.
Here for given ${\bf x} = (x^1,x^2,\ldots, x^k)$,
\beastar
{\bf x_1} & = & (x^1, \ldots, x^{i-1}, *, x^{i + k_2}, \ldots, x^{k_1 + k_2}),\\
{\bf x_2} & = & (x^i, \ldots, x^{i + k_2-1}).
\eeastar

Note that the case (3), (4) occur at {\em leaf components} and the {\em root component} of  $\{\Sigma_\alpha\}_{\alpha\in\frak A}$, which correspond to the two cases of domain
degenerations described in Theorem \ref{thm:f-boundary} respectively.

The case (3) contributes to the following terms
\begin{align}\label{eqn:functor_relation1}
\sum_{\substack{k=k_1+k_2-1\\n,k_2}}(-1)^{\ddagger+1}{\frak f}^{k_1}_\lambda(x^1,\dots,x^n,{\frak m}^{k_2} (x^{n+1},\dots, x^{n+k_2}),x^{n+k_2+1},\dots, x^k),
\end{align}
where $\ddagger=\ddagger_n=\sum_{i=1}^{n}\mu(x^i)-n$.
While the case (4) corresponds to
\begin{align}\label{eqn:functor_relation2}
\sum_{\substack{k=k_1+\cdots+k_\ell \\ \ell,k_1,\dots,k_\ell}}{\frak m}^\ell({\frak f}^{k_1}_\lambda(x^{1},\dots,x^{k_1}),\dots,{\frak f}^{k_\ell}_\lambda(x^{k-k_{\ell}+1},\dots,x^{k})).
\end{align}
This verifies that $\{{\frak f}^k\}_{k\in\N}$ satisfies the $A_\infty$ homomorphism relation, i.e. (\ref{eqn:functor_relation1}) equals to (\ref{eqn:functor_relation2}).
We refer to Chapter 7 Section 1 of \cite{fooo:book2} for the detail of this proof for
the case of $A_\infty$ homomorphism but the same argument applies to the case of $A_\infty$ functors.
\end{proof}

\section{Homotopy of $A_\infty$ functors}
\label{sec:homotopy}

Now let $\lambda_1,\, \lambda_2\in C^\infty([0,1],\mathcal C(N))$ be two paths of admissible metrics connecting
$$
g_0 = \lambda_1(0)=\lambda_2(0), \quad g_0'=\lambda_1(1)=\lambda_2(1)
$$
with $g_0 \geq g_0'$.
Denote by $\CF=\{\frak f^k_\lambda\}_{k\in\N}$ and $\CF'=\{\frak f'^k_\lambda\}_{k\in\N}$ be the corresponding $A_\infty$ functors
from $\CW\CF (M \setminus K)$ to $\CW\CF (M \setminus K)$ constructed
in the previous subsection.

Now suppose we are given a path $\Gamma \in C^\infty([0,1] \times [0,1],\mathcal C(N))$ of paths
connecting $\lambda_1$ and $\lambda_2$. The main objective of this subsection is to construct an
$A_\infty$ homotopy, denoted by $\frak h = {\frak h}_\Gamma$, between $\CF$ and $\CF'$.

\subsection{Definitions of $A_\infty$ homotopy and composition}

To motivate our construction of the relevant decorated moduli space below,
we first recall the definition of an $A_\infty$ homotopy from \cite{hasegawa}, \cite{seidel:book}.

\begin{defn}[$A_\infty$ homotopy]\label{defn:homotopy} Let $(\CA,\frak m_\CA), (\CB,\frak m_\CB)$ be two $A_\infty$ categories,
and  $\CF=\{\frak f^k\}_{k\in\N},\, \CG=\{\frak g^k\}_{k\in\N}$ be two $A_\infty$ functors from $\CA$ to $\CB$. A \emph{homotopy} $\CH$ from $\CF$ to $\CG$
is a family of morphisms
$$
\frak h^k: \CA^{\otimes k} \to \CB, \quad k \geq 1,
$$
of degree $-k$ satisfying the equation that
\bea\label{eq:homotopy}
&{}& ( \frak f^k - \frak g^k ) (a^{1},\dots, a^{k} ) \nonumber\\
& = & \sum_{m,n}(-1)^\dagger \frak h^{k-m+1}(a^1, \dots, a^{n}, \frak m^m_\CA(a^{n+1}, \dots,
a^{n+m}), a^{n+m+1}, \dots, a^{d} )\nonumber\\
&{}& \quad +\sum_{r,i}\sum_{s_1,\dots,s_r} (-1)^\clubsuit \frak m^r_\CB\big(\frak f^{s_1}(a^{1},\dots,a^{s_1}),\dots,\frak f^{s_i-1}(\dots,a^{s_1+\cdots+s_{i-1}}),\nonumber\\
&{}& \qquad \qquad \qquad{\frak h}^{s_i}(a^{s_1+\cdots+s_{i-1}+1},\dots, a^{s_1+\cdots+s_i}),\nonumber\\
&{}& \qquad \qquad \frak g^{s_i+1}(a^{s_1+\cdots+s_i+1},\dots),\dots,\frak g^{s_r}(a^{d-s_r+1},\dots,a^d)\big).
\eea
where $s_1+\cdots+s_r=d$ and
\begin{align*}
\dagger &= \sum_{i=1}^{n}\mu(a^i)-n,\qquad \clubsuit = \sum_{\ell=1}^{s_1+\dots+s_{i-1}} \mu (a^\ell)-\sum_{\ell=1}^{i-1}s_\lambda.
\end{align*}
\end{defn}

We will apply this definition of homotopy in the proof of consistency condition for
the system
$$
\left(\CW\CF (M \setminus K;H_{g_0}), \{\frak m^k \}\right) \stackrel{\Phi_\lambda}\longrightarrow
\left(\CW\CF (M \setminus K;H_{g_0'}), \{\frak m^k \}\right)
$$
to define a homotopy:
For a triple $g \geq g' \geq g''$ of metrics on $M$
we have functors $\mathcal F_\lambda,\mathcal F_{\lambda'}$, and $\mathcal F_{\lambda''}$.
We consider the composition of $A_\infty$~functors and construct a homotopy
between $\mathcal F_{\lambda'}\circ \mathcal F_\lambda$ and $\mathcal F_{\lambda''}$.

The composed $A_\infty$~functor $\mathcal F_{\lambda'}\circ \mathcal F_\lambda$ which we recall
consists of the following data of $A_\infty$~homomorphism:
\[
(\CF_{\lambda'}\circ \CF_\lambda)^d(x^1,\dots, x^d)
=\sum_{r}\sum_{s_1,\dots,s_r}\frak f_{\lambda'}^r(\frak f_\lambda^{s_1}(x^{1},\dots,x^{s_1}),\dots,\frak f_\lambda^{s_r}(x^{d-s_r+1},\dots,x^{d})
 )
\]
We consider two paths of cylindrical Riemannian metric between the
concatenated path $\lambda\ast \lambda':[0,1]\to \mathcal C(N)$ and the direct path $\lambda''$
connecting $g$ and $g''$. It easily follows from contractibility
of the space of cylindrical Riemannian metrics we can construct a path (of paths)
$$
\Gamma:[0,1]\times [0,1]\to \CC(N)
$$
between $\lambda\ast \lambda'$ and $\lambda''$.
Here we use the coordinate $(r,s)$ for $[0,1]\times[0,1]$, where $s$ is the newly adopted one.
Note that $\Gamma(-,s)=: \lambda_s$ gives a homotopy between $g_0$ and $g_0''$ for each $s\in[0,1]$.

Now we go back to the construction of $A_\infty$ homotopy associated to
a geometric homotopy.
For each fixed $s\in [0,1]$, we have the corresponding spaces of admissible almost
complex structures and of admissible Hamiltonians
\begin{align}\label{eqn:parametrized acs and Hamiltonian}
\mathcal J^r_{s}=\mathcal J(T^*N,\Gamma(r,s)),\quad
\mathcal H^r_{s}=\mathcal H(T^*N,\Gamma(r,s)).
\end{align}
Let us recall from Definition \ref{def:partial order} that
 $(\Sigma;\mathbf z;\rho) \in {\overline \CN}{}^{k+1}$ admits a partial order $\prec$
on the set $\frak A$ of irreducible components $\alpha$ of $\Sigma$.

\begin{rem}
\begin{enumerate}
\item With the above preparation, it looks natural to consider the
parameterized ${\frak f}$-moduli space, which was called a \emph{timewise moduli space}
in \cite{fooo:book1,fooo:book2}
\be\label{eq:paraCN}
{\overline \CN}{}_{\text{\rm para}}^{k+1}(y;{\bf x}) = \bigcup_{s \in [0,1]} \{s\} \times
{\overline \CN}{}_{\rho(s)}^{k+1}(y;{\bf x}),
\ee
where $\rho(s)$ is a homotopy between $\rho(0),\rho(1):\frak A \to [0,1]$ which induces
the $A_\infty$ functors $\mathcal F_{\lambda'}\circ \mathcal F_\lambda$
and $\mathcal F_{\lambda''}$, respectively.
It is fibered over $[0,1]$
$$
ev_{[0,1]}: {\overline \CN}{}_{\text{\rm para}}^{k+1}(y;{\bf x}) \to [0,1].
$$
By definition, each fiber of this map is compact. It is tempting to
define the homotopy by setting its matrix coefficient to be
$$
\#\left({\overline \CN}{}_{\text{\rm para}}^{k+1}(y;{\bf x}) \right)
$$
when the moduli space associated to $(y;{\bf x})$ has
its virtual dimension 0. However this definition will not
lead to the relation required for the $A_\infty$ homotopy described in
Definition \ref{defn:homotopy} because the one-dimensional components of
the above defined moduli space has its boundary that is not consistent
with the homotopy relation.

So we need to modify this moduli space so that the deformed
homotopy map whose matrix coefficients are defined via counting the
zero dimensional components of the modified moduli space. One requirement
we impose in this modification is that the strata the domain of each element of which is irreducible
are unchanged from above. We will modify those strata whose elements have
nodal domains, i.e., not irreducible.
\item
We would like to mention that this kind of modification process already
appeared in the definition of $A_\infty$ map ${\frak f}$ where the
time-order decoration $\rho$ is added to the moduli space of bordered stable maps
to incorporate the datum of geometric homotopy $\lambda: r \mapsto g_\lambda(r)$.
We also recall the starting point of Fukaya-Oh-Ohta-Ono's deformation
theory of Floer homology \cite{fooo:book1} in which modifying the definition of
Floer boundary map to cure the anomaly $\del^2 \neq 0$.
\item A general categorial construction of $A_\infty$ homotopy associated to
Lagrangian correspondence is given as a natural transformation between two
$A_\infty$ functors using the quilted setting in \cite{MWW}. It appears to us that
to apply this general construction, we should lift a family of Hamiltonian
to a Lagrangian cobordism arising as the suspension of the associated Lagrangian isotopy.
We avoid using this general set-up and Lagrangian suspension
but instead quickly provide a direct and down-to-earth construction
as a variation of the construction given in \cite[Section 4.6]{fooo:book2}.
\end{enumerate}
\end{rem}

\subsection{Timewise decorated time-allocation stable curves}

We recall the evaluation map
$ev_{[0,1]}: {\overline \CN}{}_{\text{\rm para}}^{k+1}(y;{\bf x}) \to [0,1]$
which naturally induces a map
$$
{\frak A} \to [0,1]; \quad \alpha \mapsto s(\alpha)
$$
where $s(\alpha)$ is the time in $[0,1]$ at which the irreducible component
$((\Sigma_\alpha,z_\alpha), u_\alpha)$ of $((\Sigma,z), u)$ is attached.
For our modification of the above mentioned moduli space
${\overline \CN}{}_{{\bf w},\text{\rm para}}^{k+1}(y;{\bf x})$, we assign a
{\em total order} $\lhd$ on the component set $\frak A$ of $\Sigma$
which is determined by the rooted {\em ribbon} structure on $(\Sigma,{\bf z}) \in {\overline \CM}{}^{k+1}$.
This order will play an important role in establishing the homotopy
relation \eqref{eq:homotopy} between $\CF_{\lambda''}$ and $\CF_{\lambda'} \circ \CF_\lambda$.
\[
\begin{tikzcd}[column sep=scriptsize]
\CW\CF (M\setminus K;H_{g_0}) \arrow[dr, "\mathcal{F}_\lambda"] \arrow[rr, "\mathcal{F}_{\lambda''}",
""{name=U,below}]{}
    & & \mathcal{WF}(M\setminus K;H_{g_0''})  \\
& \CW\CF (M\setminus K; H_{g_0'}) \arrow[Leftarrow, from=U,
"\mathcal {H}"] \arrow[ur,"\mathcal{F}_{\lambda'}"]
\end{tikzcd}
\]

Motivated by this diagram and to achieve the $A_\infty$ homotopy relation given in
Definition \ref{defn:homotopy}, we would like to attach the ${\frak m}$-component only at
$\del [0,1] = \{0,1\}$ in the way that the one with time-allocation $\rho = 0$ at $s=0$ and
$\rho =1$ at $s=1$.  We note that the components with $\rho(\alpha) = 0, \, 1$ are
type $\frak m$-components.

Recall from Definition~\ref{def:ribbon graph} that $\Gamma=\Gamma({\bf N})$ is the rooted ribbon tree associated to ${\bf N}\in \overline{\mathcal N}{}^{k+1}$ or $\overline{\mathcal M}{}^{k+1}$. From now on, we canonically identify $V^{\rm ext}(\Gamma)$ with the boundary punctures ${\bf z}=(z^0,\dots, z^k)$ and $V^{\rm int}(\Gamma)$ with the component $\frak A$. Each edge of $\Gamma$ carries the natural orientation
which flows into the root $z^0$.

Each $z^j$, $j\geq 1$ determines a unique (minimal) edge path $\ell_j = \ell_{z^0 z^j}$ from $z^0$ to $z^j$ respecting the orientation of $\Gamma$ (in the opposite direction). For an interior vertex $v$ of
$\Gamma$, we say $v \in \ell_j$ if $v \in V^{\text{int}}(\ell_j)$ where
$$
V^{\text{int}}(\ell_j) = \{ v \in V^{\text{int}}(\Gamma) \mid v \in \ell_j\}.
$$
We have an obvious  order on each $V^{\text{\rm int}}(\ell_j)$ given by the
edge distance from the root.
Denote $v_\alpha \in V^{\rm int}(\Gamma)$ the interior vertex corresponding to $\alpha \in \frak A$.
Now define
\begin{align*}
j_{\text{\rm tm}}: \frak A &\to \{1,\dots, k\};\\
\alpha &\mapsto \min \{j \mid v_\alpha \in \ell_j \}.
\end{align*}

\begin{defn}\label{defn:<s} Let $\Sigma$ be a bordered stable curve of genus zero with
$k+1$ marked points with its combinatorial type $T$.
Let $\alpha, \, \beta \in {\frak A}$. We say
$\alpha \lhd \beta$ if one of the following holds:
\begin{enumerate}
\item $j_{\text{\rm tm}}(\alpha) < j_{\text{\rm tm}}(\beta)$.
\item $j_{\text{\rm tm}}(\alpha) = j_{\text{\rm tm}}(\beta)$ and $\rho(\alpha) \leq \rho(\beta)$.
\end{enumerate}
\end{defn}
It is easy to check that $\lhd$ defines a total order on $\frak A$
and depends only on the ribbon structure of $\Gamma$.
(We refer to \cite[p.7]{savelyev} for more pictorial description of
this order.) We denote the set of descendants of $\delta$ by
\begin{align*}
{\rm Desc}(\delta):=\{\alpha \in \frak A \mid \alpha \lhd \delta, \ \alpha \neq \delta \}.
\end{align*}

We now recall the universal family $\overline \CS{}^{k+1} \to \overline \CM{}^{k+1}$.
Each element \[{\bf v}=(\Sigma,j;\sigma) \in \overline \CS{}^{k+1}\] picks out a distinguished component
$\delta({\bf v}) \in \frak A$ that contains the point $\sigma \in \Sigma$.

\begin{defn}[Universal parametrization]\label{defn:fraks} Let ${\bf v} \in \overline \CS{}^{k+1}$ and
$\delta({\bf v})$ be the associated component.  We define an $s$-parameterized map
${\frak s}_{{\bf v}}: [0,1] \times \frak A \to [0,1]$ by
\begin{align}\label{eq:s(alpha)}
{\frak s}_{{\bf v}}(s)(\alpha)=
\begin{cases}
1 &\text{ if } \alpha \in \text{\rm Desc}(\delta(\bf v));\\
s &\text{ if } \alpha=\delta(\bf v);\\
0 &\text{ otherwise}.
\end{cases}
\end{align}
\end{defn}
Note that the map $\frak{s}_{\bf v}$ is determined by the ribbon graph $\Gamma(\bf v)$ and the distinguished component $\delta(\bf v)$.
Indeed, $\frak{s}_{\bf v}$ is defined over $\overline \CS{}^{k+1}/\hspace{-1.5mm}\sim$ where
\[
{\bf v} \sim {\bf v'} \quad \Longleftrightarrow \quad \Gamma({\bf v})=\Gamma({\bf v'}),\ \
\delta({\bf v})=\delta({\bf v'}).
\]

We are ready to describe a parameter space $\overline{\mathcal L}{}^{k+1}$ defined over $\overline{\mathcal S}{}^{k+1}/\hspace{-1.5mm}\sim$
which will be used in the construction of $A_\infty$ homotopy or 2-morphism.
\begin{defn}\label{def:Lspace}
For each ${\bf v}=[\Sigma,j;\sigma]\in\overline \CS{}^{k+1}/\hspace{-1.5mm}\sim$ with $k\geq 2$,
we define $\overline{\mathcal L}{}^{k+1}$ to be the set of quadruples
\[
(\Sigma,j;\rho_{\bf v};\frak{s}_{\bf v}),
\]
where
\begin{itemize}
\item $\rho_{\bf v}$ is a time allocation in Definition~\ref{def:rho} satisfying $\rho(\delta({\bf v}))\neq 0$ or $1$,
\item $\frak{s}_{\bf v}$ is a $s$-parameterized map defined in Definition~\ref{defn:fraks}.
\end{itemize}
\end{defn}

By extending Definition~\ref{def:ribbon graph}, let us denote a ribbon graph induced from
 ${\bf L}\in \overline{\mathcal{L}}{}^{k+1}$ by $\Gamma({\bf L})$. The ribbon graph $\Gamma({\bf L})$ has an additional decoration of $\delta(\bf v)$.
 Note that $\delta(\bf v) \in V^{\frak f}(\Gamma)$ by the condition of $\rho_{\bf v}$ in Definition~\ref{def:Lspace}.
Let
\[
\mathcal{L}_\Gamma=\{{\bf L}\in \overline{\mathcal{L}}{}^{k+1} \mid \Gamma({\bf L})=\Gamma\},
\]
then we have the decomposition
\begin{align}\label{eqn:decomposition of L}
\overline{\mathcal{L}}{}^{k+1}=\bigcup_{\Gamma\in G^{k+1}}\mathcal{L}_{\Gamma},
\end{align}
where $G^{k+1}$ be the set of combinatorial types of $\overline{\mathcal{L}}{}^{k+1}$.

As analogues of Lemma~\ref{lem:dim of Nspace} and Proposition~\ref{prop:struc of Nspace} we have
\begin{lem}
For any $\Gamma \in G^{k+1}$, $\mathcal{L}_\Gamma$ is diffeomorphic to $D^{|\Gamma|+1}$.
\end{lem}

\begin{proof}
The space $\mathcal{L}_\Gamma$ additionally has a datum of $s$-parameterized map $\frak{s}_{\bf v}$ compared to $\mathcal N_\Gamma$, so the lemma follows.
\end{proof}

\begin{prop}
The parameter space $\overline{\mathcal L}{}^{k+1}$ has a structure of smooth manifold with boundaries and corners, which are compatible with the decomposition (\ref{eqn:decomposition of L}). Moreover, there is a projection map
\begin{align*}
\overline{\mathcal{L}}{}^{k+1} & \to [0,1]; \\
(\Sigma,j;\rho_{\bf v};\frak{s}_{\bf v}) &\mapsto \frak s_{\bf v}(s)(\delta({\bf v}))=s
\end{align*}
which is a smooth submersion.
\end{prop}

\begin{proof} The map factors through the  map
$
\overline{\mathcal{L}}{}^{k+1} \to \overline \CM{}^{k+1} \times [0,1]
$
defined by
$$
(\Sigma,j;\rho_{\bf v};\frak{s}_{\bf v}) \mapsto (\Sigma,j;\frak s_{\bf v}(s)(\delta({\bf v})))=(\Sigma,j;s),
$$
so the submersion property follows.
\end{proof}

\begin{lem}\label{lem:delNColl}
The boundary $\del \overline{\mathcal L}{}^{k+1}$
is decomposed into the union of the following three types of fiber products:
\begin{enumerate}
\item $
{\CM}{}^{\ell+1}\#\left(\overline{\CN}{}^{k_1+1} \times \cdots \times \overline\CL{}^{k_i+1} \times \cdots \times
\overline{\CN}{}^{k_\lambda+1}\right)$, where $\sum_{i=1}^\ell k_i = k$.
\item $\overline\CL{}^{\ell+1}\#{\CM}^{k-\ell+1}$
where $\ell = 1,\dots , k-1$.
\item $\displaystyle{\{0,1\} \times \overline{\CN}{}^{k+1}}$.
\end{enumerate}
\end{lem}

\subsection{Construction of  $A_\infty$ homotopy}

We first describe the situation we are in for the purpose of constructing an
$A_\infty$ functor.

\begin{defn}[Floer data for $A_\infty$ homotopy]\label{def:Floer data homotopy}
A Floer datum $\mathcal D_\frak h({\bf L})$
for ${\bf L}=(\Sigma,j;\rho;\frak s)\in \overline\CL{}^{k+1}$ is a quintuple
\[
({\bf w}, \{\beta_\alpha\}_{\alpha \in \frak A}, {\bf H}_{\bf L}=\{H_\alpha\}_{\alpha \in \frak A}, {\bf J}_{\bf L}=\{J_\alpha\}_{\alpha \in \frak A}, \eta)
\]
defined by modifying a Floer datum for $A_\infty$ functor in Definition~\ref{def:Floer data for functor} as follows:
\begin{enumerate}
\item A $(k+1)$-tuple of non-negative integers $\mathbf w=(w^0,\dots,w^k)$ which is
assigned to the marked points $\mathbf z$ satisfying
$$
w^0=w^1+\cdots+w^k.
$$
\item A consistent choice of one-forms $\beta_\alpha\in\Omega^1(\Sigma_\alpha)$, for each $\alpha\in \frak A$, constructed in Section \ref{sec:one-form} satisfying $\epsilon^{j*}\beta_\alpha$ agrees with $w^jdt$.
\item[($3_{\rho,{\frak s}}$)] Hamiltonian perturbations $\mathbf H_{\bf L}$
which consist of maps
$$
H_{\alpha}: \Sigma_\alpha \to \mathcal H^{\rho(\alpha)}_{{\frak s}_{\bf v}(s,\alpha)}
$$
for each $\alpha\in{\frak A}$, see (\ref{eqn:parametrized acs and Hamiltonian}). Its pull-back under the map $\epsilon^{j_\alpha}$ uniformly converges to $\frac{H}{(w^{j_\alpha})^2}\circ \psi_{w^{j_\alpha}}$ near each $z^{j_\alpha}$ for some $H\in\mathcal H^{\rho(\alpha)}_{{\frak s}_{\bf v}(s,\alpha)}$.
\item[($4_{\rho,{\frak s}}$)] Almost complex structures $\mathbf J_{\bf L}$
which consist of maps
\[
J_{\alpha}:\Sigma_\alpha \to\mathcal J^{\rho(\alpha)}_{{\frak s}_{\bf v}(s,\alpha)}
\]
for each $\alpha\in{\frak A}$, whose pull-back under $\epsilon^{j_\alpha}$ uniformly converges to $\phi_{w^{j_\alpha}}^*J$
near each $z^{j_\alpha}$ for some
$J\in\mathcal J^{\rho(\alpha)}_{{\frak s}_{\bf v}(s,\alpha)}$, see also (\ref{eqn:parametrized acs and Hamiltonian}).
\item[(5)] A map $\eta:\partial \Sigma\to [1,+\infty)$
which converges to $w^j$ near each $z^j$.
\end{enumerate}

A universal choice of Floer data $\frak D_{\frak h}$ for the $A_\infty$ homotopy consists of the collection of $\mathcal D_{\frak h}$ for every $(\Sigma,j;\rho;\frak s) \in \coprod_{k\geq1}\overline{\mathcal L}{}^{k+1}$ satisfying the corresponding conditions as in Definition \ref{def:universal Floer data} for $\overline{\mathcal L}{}^{k+1}$ and its boundary strata.
\end{defn}

\begin{figure}[ht]
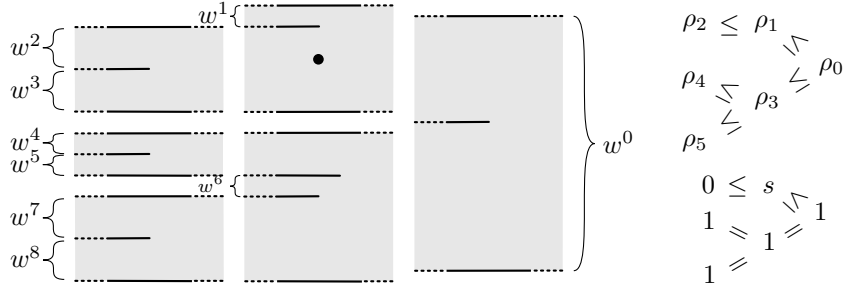
\caption{An example of slit domains for the $A_\infty$ homotopy}
\end{figure}

Here we would like to recall the readers that the current geometric circumstance is
that of $2$-morphisms. The collection ${\bf H} = \{H^r_s\}_{(r,s) \in [0,1]^2}$ satisfies
\be\label{eq:H2morphism}
H^0_s \equiv H_{g_0}, \quad H^1_s \equiv H_{g_\lambda},
\ee
see (\ref{eqn:parametrized acs and Hamiltonian}).

For $(k+1)$-pair of Hamiltonian chords
\begin{align*}
\mathbf x=(x^1,x^2,\dots, x^k)&\in {\frak X}(H_{g_0};L^0,L^1) \times \cdots \times {\frak X}(H_{g_0};L^{k-1},L^{k});\\
y&\in {\frak X}(H_\lambda;L^0,L^k),
\end{align*}
we consider a system of maps $\{u_\alpha:\Sigma_\alpha \to T^*N\}_{\alpha\in\frak A}$ with respect to
$\mathcal D_{\frak h}$ over
\[
{\bf L}=(\Sigma,j;\rho;\frak s)\in \overline \CL{}^{k+1}
\]
satisfying the following:
\begin{enumerate}
\item[(1-a)] For $\alpha\in {\frak A}_\Sigma$ except the distinguished component, $u_\alpha$ satisfies
$$
(du_\alpha - X_{H_{\alpha}}\otimes \beta_\alpha)_{J_{\alpha}}^{(0,1)}=0
$$
for $H_\alpha \in \CH^{\rho(\alpha)}_{*}$ and $J_\alpha \in \CJ^{\rho(\alpha)}_{*}$, where $*=0,1$ depending on $\frak s(\alpha)$.
\item[(1-b)] For the distinguished component $\delta\in\frak A$ in Definition~\ref{defn:fraks} and for each $s\in[0,1]$, let us consider a map $u_{\delta(s)}$ satisfying
$$
(du_{\delta(s)} - X_{H_{\delta(s)}}\otimes \beta_\delta)_{J_{\delta(s)}}^{(0,1)}=0,
$$
where $H_{\delta(s)} \in \CH^{\rho(\alpha)}_{\frak s(s,\delta)}$ and $J_{\delta(s)} \in \CJ^{\rho(\alpha)}_{\frak s(s,\delta)}$.
\end{enumerate}
and the conditions (2)--(6) for $\overline{\mathcal N}{}^{k+1}_{(\Sigma,j;\rho)}(y;\mathbf x)$ in (\ref{eqn:moduli space N}).
For each $s_0\in[0,1]$, let us denote the space of maps
\[\{u_\alpha\}_{\alpha \in \frak A\setminus\{\delta\}}\cup\{u_{\delta(s_0)}\}\]
satisfying the above by
$$
\overline{\CL}{}^{k+1}_{{\bf L}}(s_0)(y;\mathbf x)=\overline{\CL}{}^{k+1}(y;\mathbf x;\mathcal D_{\frak h}(\Sigma,j;\rho;\frak s;s_0)).
$$
Now we consider parameterized moduli spaces given by
\begin{align*}
\overline{\CL}{}^{k+1}_{{\bf L}}(y;\mathbf x)=\bigcup_{s\in[0,1]} \overline{\CL}{}^{k+1}_{{\bf L}}(s)(y;\mathbf x);\\
\overline{\CL}{}^{k+1}(s)(y;\mathbf x)=\bigcup_{{\bf L}\in\overline{\CL}{}^{k+1}} \overline{\CL}{}^{k+1}_{{\bf L}}(s)(y;\mathbf x);\\
\overline{\CL}{}^{k+1}(y;\mathbf x)=\bigcup_{{\bf L}\in\overline{\CL}{}^{k+1}} \overline{\CL}{}^{k+1}_{{\bf L}}(y;\mathbf x)=\bigcup_{s\in[0,1]} \overline{\CL}{}^{k+1}(s)(y;\mathbf x).
\end{align*}
Similarly as for ${\overline \CN}{}^{k+1}(y;\mathbf x) $,
we have a forgetful map
$$
{\frak{forget}}: \overline{\CL}{}^{k+1}(y;\mathbf x)
\to \overline{\CL}{}^{k+1}
$$
for $k \geq 2$ and the evaluation map $ev_{[0,1]}:\overline{\CL}{}^{k+1}(y;\mathbf x)\to[0,1]$ defined by
\begin{align}\label{eqn:ev_s}
\{u_\alpha\}_{\alpha \in \frak A\setminus\{\delta\}}\cup\{u_{\delta(s_0)}\}\mapsto s_0.
\end{align}
Standard parameterized transversality  then proves the following

\begin{lem} Suppose that ${\overline \CL}{}^{k+1}(s_0)(y;\mathbf x)$
with $s_0=0, \, 1$ are transversal. Then for a generic choice Floer data $\frak D_{\frak h}$ for the $A_\infty$ homotopy
 with fixed ends for $s=0, \, 1$, the moduli space
 $\overline{\CL}{}^{k+1}(y;\mathbf x)$ is a manifold of dimension
\[\mu(y)-\sum_{i=1}^{k}\mu(x^i)+k.\]
\end{lem}

Again, when $\mu(y)=\sum_{i=1}^{k}\mu(x^i)-k$, a map is defined by counting rigid solutions of $\overline\CL{}^{k+1}(y;\mathbf x)$
with respect to the Floer data $\frak D_{\frak h}$ in Definition~\ref{def:Floer data homotopy}
\[
{\frak h}_{\frak D}^k( \psi_{w^1}(x^1) \otimes \psi_{w^2}(x^2) \otimes \cdots  \otimes \psi_{w^k}(x^k))
=\sum_{y}(-1)^\dagger\#\left(\overline\CL{}^{k+1}(y;\mathbf x)\right)\psi_{w^0}(y),
\]
where $\dagger=\sum_{i=1}^k i \cdot \mu(x^i)$.
By the same identification as used in the definition ${\frak m}_k$ and ${\frak f}_k$, we will obtain
a $A_\infty$ homotopy
$$
{\frak h}^k :CW^*(L^{1},L^{0};H_{g_0}) \otimes \cdots \otimes CW^*(L^{k},L^{k-1};H_{g_0})
\to  CW^*(L^k,L^0;H_{g_0'})[-k].
$$

For the later purpose, we describe the structure of the zero dimensional component more carefully.
Recall that the fiberwise dimension of $\overline{\CL}{}^{k+1}(y;\mathbf x)$ for the current circumstance is
$-1$ and that we assume the cases of $s = 0, \, s=1$ are generic. Therefore there is no contribution
therefrom. Therefore there are a finite number of $0< s_1 < s_2 < \cdots < s_j < 1$ such that
$$
\#\left(\overline\CL{}^{k+1}(y;\mathbf x)\right) = \sum_{i=1}^j
\#\left(\overline\CL{}^{k+1}(s)(y;\mathbf x)|_{s = s_i}\right).
$$
Furthermore, generically the associated moduli space
$
\overline\CL{}^{k+1}(s)(y;\mathbf x)|_{s = s_i}
$
is \emph{minimally degenerate}, i.e., the cokernel of the associated linearized operator
has dimension 1. (See \cite{yijen} for the relevant parameterized gluing result.)
Denote by
$$
\text{\rm Sing}^{\rm I} \subset [0,1]
$$
the set of these points.

In order to investigate the algebraic relation for the $A_\infty$ homotopy for
the map $\{{\frak h} ^k\}_{k\in\N}$ we need to look up the codimension one strata
of the moduli space $\overline{\CL}{}^{k+1}(y;\mathbf x)$.
Especially we need to examine the structure of the boundary of one-dimensional components i.e., of those
satisfying
\be\label{eq:1dim}
\mu(y)-\sum_{i=1}^{k}\mu(x^j)+k = 1.
\ee

\begin{thm}\label{thm:homotopy}
 The maps $\{\frak h ^k\}_{k\in \N}$ satisfies the homotopy relation \eqref{eq:homotopy}.
\end{thm}
\begin{proof}
We first recall the current geometric circumstance. We are given two
Hamiltonian $H^0 \in \CH(\widehat{g}_i)$ and $H^1 \in \CH({\widehat g_\lambda})$, and two paths
$r \mapsto H_*(r) \in \CH(\widehat{g}(r))$ with $*=0, \, 1$ connecting the two Hamiltonians,
and then a homotopy $s \mapsto H_s$ interpolating the two paths with $s \in [0,1]$.

\[
\begin{tikzcd}
\CW\CF (M\setminus K;H_{g_0})\phantom{aaaaa} \arrow[r, bend left=40,
"\mathcal F_{\lambda''}(H_0(r))"{name=U, below}]
\arrow[r, bend right=40, "\mathcal F_{\lambda''}(H_1(r))"{name=D}]
&\phantom{aaaaaaa}\CW\CF(M\setminus K;H_{g''_0})
\arrow[Rightarrow, from=U, to=D,"\mathcal H (H_s(r))"] \end{tikzcd}
\]

Let us consider a sequence
\begin{align}\label{eqn:seq for homotopy relation}
\left(s_\delta^{(i)}, \Gamma(\vec{u}^{(i)}),(\Sigma_\alpha^{(i)},\rho_\alpha^{(i)})_{\alpha\in\frak A}\right)^{i\in \N}
\end{align}
constructed from $\vec{u}\in\overline\CL{}^{k+1}(y;\mathbf x)$ satisfying \eqref{eq:1dim},
i.e., when the fiberwise dimension becomes zero.
Here $s_\delta$ is the $\frak s$-image of the distinguished component $\delta$, see (\ref{eqn:ev_s}), and
$\Gamma(\vec{u})$ is a ribbon graph with decorations defined as in Definition~\ref{def:ribbon graph}.

Generically, the change of $\Gamma(\vec{u})$ occurs at finite number points $s=s_0$ in $[0,1]$
provided the corresponding moduli spaces are regular. We denote by
$$
\text{\rm Sing}^{\rm II} \subset [0,1]
$$
the set of these points. There are also a finite number of points $s$ at which
the fiber $ev_{[0,1]}^{-1}(s)$ contains an element that is minimally degenerate, i.e.,
the cokernel of the linearized operator has dimension 1. We denote by
$$
\text{\rm Sing}^{\rm III} \subset [0,1].
$$
Generically, the three subsets $\text{\rm Sing}^{\rm I},\, \text{\rm Sing}^{\rm II},\, \text{\rm Sing}^{\rm III}$
are pairwise disjoint. We denote the union by $\text{\rm Sing}([0,1]_s)$.

The sequence (\ref{eqn:seq for homotopy relation}) with fixed
type of domain configuration $\Gamma(\Sigma)$, see Definition~\ref{def:ribbon graph}.
By choosing a subsequence,
we may assume that the sequence $s_\delta^{(i)}$ is increasing.

Then  we have the following possible scenarios:
\begin{enumerate}
\item $s_\delta^{(i)}$ converges to $0 < s_0 < 1$, one of the points in $\text{\rm Sing}([0,1]_s)$.
\item $\lim_{i\to +\infty}s_{\delta}^{(i)}=0$.
\item $\lim_{i\to +\infty}s_{\delta}^{(i)}=1$.
\end{enumerate}

We examine the three cases more closely. The first case $s_\delta^{(i)} \to s_0$ is further divided into the
following four different scenarios by Gromov-Floer compactness:
\begin{enumerate}
\item[(1-i)] One of the component $\Sigma_\alpha^{(i)}$ with
$0 < \rho^{(i)}_\alpha < 1$ splits into two components at $s = s_0$ as $i\to +\infty$.
\item[(1-ii)] Two component $\Sigma_{\alpha_1}$ and $\Sigma_{\alpha_2}$ at $s = s_0$
 sharing one asymptotic matching condition satisfy
$0 < \lim\rho_{\alpha_1}^{(i)}= \lim\rho_{\alpha_2}^{(i)} < 1$ at $s=s_0$ when $i\to +\infty$.
\item[(1-iii)] $\lim_{i\to +\infty}\rho_{\alpha}^{(i)}=0$ .
\item[(1-iv)] $\lim_{i\to +\infty}\rho_{\alpha}^{(i)}=1$.
\end{enumerate}
Each of these cases is the analog to the corresponding scenario of
the study of $\frak f$-moduli space $\CN{}^{k+1}(y;{\bf x})$
in the previous section. One difference is that the parameter $s = s_0$ is
not transversal fiberwise but transversal as a parameterized problem. So among the two
components $\alpha_1, \ \alpha_2$ one is an $\frak m$-component and the other is
an $\frak h$-component, i.e., a fiberwise $\frak f$-component with minimal degeneracy. Then
by the same argument as in ${\overline \CN}{}^{k+1}(y;\mathbf x)$ applied to the parametric
case, the case (1-i) and (1-ii) cancel out.

For the convenience sake, we denote the associated
structure maps by $\frak m^\ast, \frak f^\ast$, and ${\frak h}^\ast$ respectively.
A codimension one phenomenon of (1-iii) occurs
when $\alpha$ is an index for a $\frak m$-component. We note that this component $\alpha$ is generically regular
and the underlying zero-dimensional $\frak h$-component containing $\alpha$ is also regular
as a parameterized moduli space. Since at $\rho = 0, \, 1$, the Floer datum is constant over $s \in [0,1]$.
The same configuration must occur for all $s$ near $s_0$. This contradicts to the fact that
$s = s_0$ is contained in the discrete set $\text{\rm Sing}([0,1]_s)$ unless $\frak m$-bubble
is attached at $s =0, \, 1$.
Then it contributes the following terms

\begin{align}\label{eqn:homotopy term (3)}
\sum_{m,n}(-1)^\dagger \frak h^{k-m+1}(x^1, \dots, x^{n}, \frak m^m_{g_0}(x^{n+1}, \dots,
x^{n+m}), x^{n+m+1}, \dots, x^{d} ),
\end{align}
where $\dagger = \sum_{i=1}^{n}\mu(x^i)-n.$

For the case (1-iv), it follows from \eqref{eq:s(alpha)} that a codimension one strata can be obtained when $\alpha$ is an index for the root component and is different from the distinguished component $\delta$.
Then it corresponds to
\begin{align}\label{eqn:homotopy term (4)}
\sum_{r,i}\sum_{s_1,\dots,s_r} (-1)^\clubsuit \frak m^r_{g_0''} \big(\frak f^{s_1}(x^{1},\dots,x^{s_1}),\dots,\frak f^{s_i-1}(\dots,x^{s_1+\cdots+s_{i-1}}),\nonumber\\
 \qquad \qquad \qquad{\frak h}^{s_i}(x^{s_1+\cdots+s_{i-1}+1},\dots, x^{s_1+\cdots+s_i}),\nonumber\\
\qquad \qquad \frak g^{s_i+1}(x^{s_1+\cdots+s_i+1},\dots),\dots,\frak g^{s_r}(x^{d-s_r+1},\dots,x^d)\big),
\end{align}
where $\clubsuit = \sum_{\ell=1}^{s_1+\dots+s_{i-1}} \mu (x^\ell)-\sum_{\ell=1}^{i-1}s_\ell$.
Among many terms from the cases (2) and (3), all of them are cancelled out except the following two terms
\begin{enumerate}
\item[($2'$)] $\lim_{i\to +\infty}s_{\delta}^{(i)}=0$, where $\delta$ is the minimum index with respect to $\lhd$.
\item[($3'$)] $\lim_{i\to +\infty}s_{\delta}^{(i)}=1$, where $\delta$ is the maximum index with respect to $\lhd$.
\end{enumerate}
because of the total order $\lhd$ and the definition of time-wise product $\frak s$.
These two cases give the following terms
\begin{align}\label{eqn:homotopy term (5)(6)}
-\frak f^d(x^1, \dots, x^d),\quad \frak g^d(x^1, \dots , x^d),
\end{align}
respectively.

The algebraic relation coming from the combination of (\ref{eqn:homotopy term (3)}),
(\ref{eqn:homotopy term (4)}), and (\ref{eqn:homotopy term (5)(6)})
is nothing but the $A_\infty$ homotopy relation \eqref{eq:homotopy}.
\end{proof}

This concludes the proof of the following theorem.

\begin{thm}\label{thm:homotopy-direct-g} Let $M$ be a closed manifold and $g, \, g'$
be two metrics on $M$ such that their cylindrical adjustments satisfy $g_0  \geq g_0'$.
Suppose $\lambda_1$ and $\lambda_2$ be homotopies connecting $g_0$ and $g_0'$ and
let $\Gamma: [0,1]^2 \to \mathcal C(N)$ be a homotopy between $\lambda_1$ and $\lambda_2$.
Then there exists an $A_\infty$ homotopy between $\mathcal F_{\lambda_1}$ and
$\mathcal F_{\lambda_2}$.
\end{thm}

We denote the resulting $A_\infty$ category by
\be\label{eq:WFgMK}
\CW\CF_g(M \setminus K) := \CW\CF(T^*N,H_{g_0}).
\ee
In the next section, we will prove the quasi-equivalence class of the
$A_\infty$ category $\CW\CF(T^*N;H_{g_0})$
is independent of the choice of metrics $g$.

\section{Independence of choice of metrics}\label{sec:homlimit}
In this section, we consider the wrapped Fukaya categories
$\mathcal{WF}(T^*N,g_0)$ and $\mathcal{WF}(T^*N,h_0)$ constructed
on the knot complement $M \setminus K$ for two metrics $g, \, h$ of $M$.

 Moreover, it follows from Proposition \ref{prop:equivalence} that
the two Riemannian metric $g_0$ and $h_0$ are {\em Lipschitz equivalent}, i.e., there exists a constant $C > 1$ such that
 \be\label{eq:LipschitzCondition}
 \frac{1}{C} h_0 \leq g_0 \leq C h_0,
 \ee
on $M \setminus K$.

For the convenience of notation, we denote the kinetic energy Hamiltonian $H_{g_0}$ associated
to $g$ also by $H(g_0)$ in this section.

Next we prove the following equivalence theorem.

\begin{thm}\label{thm:invariancemetric}
Let $(g,h)$ be a pair of Lipschitz equivalent metrics on an orientable tame manifold $N$.
Then two induced wrapped Fukaya categories $\mathcal{WF}_g(T^*N)$ and $\mathcal{WF}_h(T^*N)$ are quasi-equivalent.
\end{thm}
\begin{proof} By the Lipschitz equivalence \eqref{eq:LipschitzCondition}, we have
\begin{align}\label{eq:confinal}
\begin{cases}
H(g_0) \leq H(h_0/C);\\
H(h_0) \leq H(g_0/C)
\end{cases}
\end{align}
recalling that the Hamiltonian is given by the {\em dual} metric.

We have two $A_\infty$ functors between induced wrapped Fukaya categories
\begin{align}\label{eqn:two vertical sequences}
\Phi  &: \mathcal{WF}(T^*N;H(h_0)) \to \mathcal{WF}(T^*N;H(g_0/C));\nonumber \\
\Psi  &: \mathcal{WF}(T^*N;H(g_0)) \to \mathcal{WF}(T^*N,H(h_0/C))
\end{align}
which are defined by the standard $C^0$-estimates for the monotone homotopies.
Note that the morphism is naturally defined from the {\em smaller metric} to the {\em bigger metric}.

Now consider the composition of the functors
\begin{align*}
\Psi \circ \Phi  &: \mathcal{WF}(T^*N;H(h_0)) \to \mathcal{WF}(T^*N;H(g_0/C^2));\\
\Phi \circ \Psi  &: \mathcal{WF}(T^*N;H(g_0)) \to \mathcal{WF}(T^*N;H(g_0/C^2)).
\end{align*}
These are homotopic to natural isomorphisms induced by the rescaling of metrics
\begin{align*}
\rho_{C^2}&:  \mathcal{WF}(T^*N;H(h_0)) \to \mathcal{WF}(T^*N,H(h_0/C^2));\nonumber\\
\eta_{C^2}&:  \mathcal{WF}(T^*N;H(g_0)) \to \mathcal{WF}(T^*N;H(g_0/C^2)),
\end{align*}
respectively. This proves that $\Phi$ and $\Psi$ are quasi-equivalences.
\end{proof}

\begin{rem}
Another class of metrics we will study in \cite{BKO} is a complete hyperbolic metric $h$
on $N = M \setminus K$ for hyperbolic knots $K$. In such a case,
we can exploit hyperbolic geometry to directly construct another $A_\infty$ category $\CW\CF(\nu^*T;H_h)$
without taking a cylindrical adjustment. Since $h$ is not Lipschitz-equivalent to
a cylindrical adjustment $g_0$ on $M \setminus K$ of any smooth metric $g$ on $M$,
that category may not be quasi-equivalent to $\CW\CF(M \setminus K)$ constructed in the
present paper.
\end{rem}

\subsection{Wrap-up of the construction of wrapped Fukaya category $\CW\CF(M\setminus K)$}

In this section, we restrict ourselves to the case when $N$ is a knot complement $M \setminus K$ in a closed oriented $3$-manifold $M$, and wrap up the proofs of the main theorems stated in the introduction.

We first note that one class of Lipschitz equivalent metrics on $N$
considered in Theorem \ref{thm:invariancemetric}
is obtained by restricting any Riemannian metric $g_M$ to $M$.
Note that these metrics on $N$ is {\em incomplete}.

\begin{thm}\label{thm:invariant}
Let $M$ be a closed oriented $3$-manifold equipped with a metric and let $M \setminus K$ be a knot.
Then there exists an $A_\infty$ category $\CW\CF(M\setminus K)$ whose isomorphism class depends only
on the isotopy type of $K$ in $M$.
\end{thm}
\begin{proof} Let $g$ be a Riemannian metric on $M$ and restrict the metric to $M \setminus K$.
By the remark above, the quasi-isomorphism type of $A_\infty$ category $\CW\CF(T^*N,H_{g_0})$
does not depend on the choice of the metric.

It remains to show the isotopy invariance of $\CW\CF(T^*N,H_{g_0})$. More precisely, let $K_0, \, K_1$ be
isotopic to each other. Suppose $\phi^t: M \to M$ be an isotopy such that $K_1 = \phi^1(K_0)$.
Denote $K_t = \phi^t(K_0)$.
We fix a metric $g$ on $M$ and consider the family of metric $g_t: = \phi^t_*g$.
We also fix a pair of precompact domains $W_0 \subset W_0' \subset M \setminus K$
with smooth boundary and fix a cylindrical adjustment
$g'$  of $g$ outside $W_0'$ so that $g_0' = g|_{\del W_0} \oplus dr^2$ on $M \setminus W_0'$
which is interpolated in $W_0' \setminus W_0$ to $g$.
Then we consider the isotopy of pairs $(M \setminus K_t, g_t)$.
We denote $W_t = \phi^t(W_0)$ and $W_t' = \phi^t(W_0')$. Then choose a smooth family of cylindrical
adjustments $g_t'$ for $g_t$ outside $W_t' \subset M \setminus K_t$ with $g_t = \phi^t_*(g)$ on $W_t$
which are interpolated in between.

Then we construct an  $A_\infty$ functor
$$
\Phi: \CW\CF_g(T^*(M\setminus K_0)) \to \CW\CF_{\phi^1_*g}(T^*(M\setminus K_0))
$$
which are given by $L \to \phi^1(L)$ objectwise and whose morphism
is defined by the same construction performed in the previous construction of $A_\infty$ functor.
By considering the inverse isotopy, we also have
$$
\Psi: \CW\CF( T^*(M\setminus K_1); H_{g_0}) \to \CW\CF( T^*(M\setminus K_0);H_{g_0'}).
$$
Then we consider the isotopy which is a concatenation of $\phi^t$ and its inverse isotopy
which is homotopic to the constant isotopy.
Then the same construction of the homotopy as the one in Section \ref{sec:construction of functor}
can be applied to produce a $A_\infty$ homotopy between $\Psi \circ \Phi$ and the identity functor.

This proves $\CW\CF(T^*(M\setminus K_0),H_{g_0})$ is quasi-isomorphic to $\CW\CF(T^*(M\setminus K_0),H_{\phi^1_*g})$, which finishes the proof.
\end{proof}

\section{Construction of Knot Floer algebra $HW(\del_\infty(M\setminus K))$}
\label{sec:knot-algebra}

In this section, we give construction of Knot Floer algebra mentioned in Definition
\ref{defn:knotFloeralgebra}.

Denote by $\mathfrak G_{g_0}(T)$ the energy of the shortest geodesic cord of $T$
relative to the metric $g_0$. Then the following lemma is a standard fact
in Riemannian geometry since $T$ is a compact smooth submanifold and $g_0$ is
of bounded geometry.

\begin{lem}\label{lem:perturbed-cord} Denote by $\mathfrak G_{g_0}(T)$
the infimum of the energy of non-constant geodesics. Then $\mathfrak G_{g_0}(T) > 0$ and
all non-constant Hamiltonian chords $\gamma$ of $(\nu^*T;H_{g_0})$ have
$$
\CA_{H_{g_0}}(\gamma) = -E(c_\gamma) \leq  - \mathfrak G_{g_0}(T).
$$
\end{lem}

Now we perform the algebra version in the Morse-Bott setting
of the construction given in the previous sections associates an $A_\infty$ algebra
$$
CW_g(L, T^*(M\setminus K)): = C^*(T)\oplus \Z\{ \mathfrak X_{< 0}(H_{g_0};\nu^*T,\nu^*T)\}
$$
where $C^*(T)$ is chosen to be a cochain complex of $T$, e.g., the de Rham complex similarly as in \cite{fooo:book1,fooo:book2}.
We note that for $0 < \epsilon_0 < \mathfrak G_{g_0}(T)$ is sufficiently small, we have
$$
\mathfrak X_{\geq -\epsilon_0}(H_{g_0};\nu^*T, \nu^*T) \cong \T^2
$$
and
$$
\mathfrak X_{< -\epsilon_0}(H_{g_0};\nu^*T, \nu^*T)
$$
is in one-one correspondence with the set of non-constant geodesic cords of $(T, g_0)$ and so
$$
CW_{< -\epsilon_0}(L, T^*(M\setminus K); H_{g_0}) = \Z\langle \mathfrak X_{< 0}(H_{g_0};\nu^*T,\nu^*T)\rangle
$$
It follows from the lemma that the associated complex
$(CW_g(L, T^*(M\setminus K)), \frak m^1)$ has a subcomplex
$$
(CW_{\geq -\epsilon_0}(L, T^*(M\setminus K); H_{g_0}), \frak m^1) \cong (C^*(T), d)
$$
where $d$ is the differential on $C^*(T)$.  We denote the associated homology by
$$
HW_g(L, T^*(M\setminus K)):= HW (L, T^*(M\setminus K); H_{g_0}).
$$

Not to further lengthen the paper and since
the detailed construction is not explicitly used in the present paper, we
omit the details of this Morse-Bott construction till \cite{BKO} where we
provide them and describe its cohomology in terms of the Morse cohomology model of $C^*(T)$.
\emph{With this mentioned, we will pretend in the discussion below that $(\nu^*T, H_{g_0})$ is a
nondegenerate pair with $\nu^*T$ implicitly replaced by a $C^2$-small perturbation $L$ thereof.}

Let $ N(K), \, N'(K)$ be a  tubular neighborhood of $K$
such that
$$
N(K) \subset \Int N'(K).
$$
We denote $M \setminus N(K) = N^{\text{\rm cpt}}$ and similarly for $N'(K)$.
We denote $T = N(K)$ and $L = \nu^*T$.

We define the cylindrical adjustments $ g_0$ of the metric $g$ on $M$
with respect to the exhaustion \eqref{eq:NiNi'} by
$$
g_0 = \begin{cases} g_0 \quad & \text{on } N^{\prime,\text{\rm cpt}}\\
da^2 \oplus g_0|_{\del N^{\text{\rm cpt}}} & \text{on } N^{\text{\rm cpt}} \setminus K
\end{cases}
$$
which is suitably interpolated on $N^{\text{\rm cpt}} \setminus N^{\prime,\text{\rm cpt}}$ and fixed.

\begin{thm}\label{thm:limitLi} The quasi-isomorphism class of $A_\infty$ algebra
$$
(CW(\nu^*T, T^*(M \setminus K;H_{g_0})), {\frak m}), \quad \frak m
= \{\frak m ^k\}_{0 \leq k < \infty}
$$
does not depend on the various choices involved such as tubular neighborhood $N(K)$ and the metric $g$ on $M$.
\end{thm}
\begin{proof} The metric independence can be proved in the same as
in the proof given in the categorical level.

Therefore we focus on the choice of tubular neighborhood. Let $N_1, \, N_2$ be
two different tubular neighborhoods of $K$ and denote by $T_1, \, T_2$ be the boundaries
$T_1 = \del N_1$ and $T_2 = \del N_2$. Choose any diffeomorphism $\phi:M \to M$ such that
$\phi(T) = T'$ and that it is isotopic to the identity fixing $K$. Then the symplectomorphism
$(d\phi^{-1})*$, which is fiberwise linear (over the map $\phi$), induces a natural quasi-isomorphism
$$
CW_g(\nu^*T_1, T^*(M\setminus K)) \cong  CW_{\phi_*g}(\nu^*T_2, T^*(M\setminus K)).
$$
On the other hand, the latter is quasi-isomorphic to $CW_g(\nu^*T_2, T^*(M\setminus K))$
by the metric independence. Invariance under other changes can be proved similarly and so
omitted.

Denote by the resulting $A_\infty$ algebra by
$$
CW(\nu^*T, T^*(M \setminus K)): = CW(\nu^*T, T^*(M\setminus K); H_{g_0})
$$
whose quasi-isomorphism class is independent of the metric $g$.
Finally we prove the invariance thereof under the isotopy of $K$.
The proof is almost the same as that of the categorical context given in the proof of
Theorem \ref{thm:invariant}. For readers' convenience, we duplicate the proof here with
necessary changes made.
Suppose $K_0, \, K_1$ be isotopic to each other and let $\phi^t: M \to M$ be an isotopy
such that $K_1 = \phi^1(K_0)$ as before. Denote $K_t = \phi^t(K_0)$.
We fix a metric $g$ on $M$ and consider the family of metric $g_t: = \phi^t_*g$.
We also fix a pair of precompact domains $W_0 \subset W_0' \subset M \setminus K$
with smooth boundary and fix a cylindrical adjustment
$g'$  of $g$ outside $W_0'$ so that $g_0' = g|_{\del W_0} \oplus dr^2$ on $M \setminus W_0'$
which is interpolated in $W_0' \setminus W_0$ to $g$.
Then we consider the isotopy of pairs $(M \setminus K_t, g_t)$.
We denote $W_t = \phi^t(W_0)$ and $W_t' = \phi^t(W_0')$. Then choose a smooth family of cylindrical
adjustments $g_t'$ for $g_t$ outside $W_t' \subset M \setminus K_t$ with $g_t = \phi^t_*(g)$ on $W_t$
which are interpolated in between.

Then we construct an  $A_\infty$ map
$$
\Phi: CW(\nu^*T_0, T^*(M\setminus K_0); H_{g_0}) \to CW(\nu^*T_1, T^*(M\setminus K_1);H_{g'_0},
\quad g' = \phi^t_*(g)
$$
is defined by the same construction performed in the previous construction of $A_\infty$ functor.
By considering the inverse isotopy, we also have
$$
\Psi: CW(\nu^*T_1, T^*(M\setminus K_1); H_{g_0'}) \to CW(\nu^*T_0, T^*(M\setminus K_0);H_{g_0}).
$$
Then we consider the isotopy which is a concatenation of $\phi^t$ and its inverse isotopy
which is homotopic to the constant isotopy.
Then the same construction of the homotopy as the one in Section \ref{sec:construction of functor}
can be applied to produce a $A_\infty$ homotopy between $\Psi \circ \Phi$ and the identity map.
This finishes the proof.
\end{proof}

\begin{defn}[Knot Floer algebra]
We denote by
$$
HW^*(\del_\infty(M\setminus K)) = \bigoplus_{d=0}^\infty HW^d (\del_\infty(M\setminus K))
$$
the resulting (isomorphism class of the) graded group
and call it the knot Floer algebra of $K$ in $M$.
\end{defn}

The same argument also proves that the isomorphism class of the algebra
depends only on the isotopy class of the knot $K$.

\part{$C^0$-estimates for the moduli spaces}
\label{part:C0estimates}

The compactifications of the moduli spaces such as
$\CM^{\mathbf w}(x^0;\mathbf x)=\CM^{\mathbf w}(\mathbf x;\mathbf H,\mathbf J,\eta)$
(and also for $\mathcal N^{\bf w}(x^0;\mathbf x)$) are essential in the definition
of the wrapped Floer cohomology and its algebraic properties, the $A_\infty$ structure.
The main purpose of the present section is to establish the two $C^0$
estimates, Proposition  \ref{prop:vertical}, \ref{prop:horizontal}, \ref{prop:vertical-f}, and \ref{prop:horizontal-f} postponed in the previous sections.

For this purpose, we observe that thanks to the choice we made for ${\bf J}$
to satisfy \eqref{eq:choiceofJ} in Remark \ref{rem:quadratic}, $u$ satisfies \eqref{eq:duXHJ}
if and only if the composition $v(z) = \psi_{\eta(z)}^{-1}(u(z))$ satisfies the
\emph{autonomous} equation
\be\label{eq:dvXHJ}
\begin{cases}
(dv - X_H \otimes \beta)_{J_g}^{(0,1)}=0,\\
\text{$v(z)\in L^i$, for $z\in\partial \Sigma$ between $z_i$ and $z_{i+1}$ where $i\in \Z_{k+1}$.}\\
v\circ \epsilon^j(-\infty,t)= x^j(t), \text{ for }j=1,\dots,k.\\
v\circ \epsilon^0(+\infty,t)= x^0(t).
\end{cases}
\ee
Therefore it is enough to prove the relevant $C^0$-estimates for the map $v$ which
we will do in the rest of this part.

\section{Horizontal $C^0$ estimates for $\frak m$-components}
\label{sec:horizontalC0}

\begin{proof}[Proof of Proposition \ref{prop:horizontal}]
Recall that on $N^{\text{end}}$, we have the product metric $g = da^2 \oplus h$.
Because of this we have the Riemannian splittings
$N^{\text{\rm end}} \cong [0,\infty) \times T$ with $T = \del N^{\text{\rm end}}$ and
$$
T^*N^{\text{\rm end}} \cong T^* (\del N^{\text{\rm end}}) \oplus T^*[0,\infty).
$$
Furthermore the Sasakian almost complex structure has the splitting $J_g = i \oplus J_h$
where $J_h$ is the Sasakian almost complex structure on $\del N^{\text{\rm end}}$ and $i$ is the
standard complex structure on $T^*[0,\infty) \subset T^*\R \cong \C$.

Denote $q = (a,q_T) \in [0, \infty) \times T$. Then we have the orthogonal
decomposition $p = (p_a,p_T)$ and hence
$$
|p|^2_g = |p_T|_h^2 + |p_a|^2.
$$
Therefore the $(a,p_a)$-components of $X_H$ and $JX_H$ are given by
\be\label{eq:pia}
\pi_{T^*[0,\infty)}(X_H(q,p)) = p_a \frac{\del}{\del a}, \quad \pi_{T^*[0,\infty)}(JX_H) = p_a\frac{\del}{\del p_a}.
\ee

Let $z = x + \sqrt{-1} y$ be a complex coordinate of $(\Sigma, j)$ such that
$$
\beta = dt_e
$$
away from the singular points of the minimal area metric.
 If we write $\beta = \beta_x dx + \beta_y dy$, then
\eqref{eq:J-holeqn} is separable.

Another straightforward calculation shows that  the $(a,p_a)$-component of the equation becomes
\begin{align}\label{eq:a-component}
\begin{dcases}
\frac{\del a(v)}{\del x} - \frac{\del p_a(v)}{\del y} - \beta_x p_a(v)= 0 \\
\frac{\del p_a(v)}{\del x} + \frac{\del a(v)}{\del y} - \beta_y p_a(v)= 0
\end{dcases}
\end{align}

We note
$$
d(\beta \circ j) = - \left(\frac{\del \beta_x}{\del x} + \frac{\del \beta_y}{\del y}\right)\, dx \wedge dy.
$$
In particular, if $d(\beta \circ j) = 0$, this vanishes.

For any given one-form $\beta$, a straightforward calculation using these identities leads
to the following formula for the (classical) Laplacian
\be\label{eq:Deltav}
\Delta (a(v)) = p_a(v)\left(\frac{\del \beta_x}{\del x} + \frac{\del \beta_y }{\del y}\right)
- \beta_x\frac{\del (a(v))}{\del y} + \beta_y \frac{\del (a(v))}{\del x}
\ee
for any solution $v$ of \eqref{eq:J-holeqn}.

\begin{lem}\label{lem:Deltaa}
Then for any one-form $\beta$ on $\Sigma$,  $a(v)$ satisfies
\be\label{eq:Deltaau}
\Delta (a(v)) = - \beta_x\frac{\del (a(v))}{\del y} + \beta_y \frac{\del (a(v))}{\del x}
\ee
for any solution $u$ of \eqref{eq:J-holeqn}.
\end{lem}
\begin{proof} This immediately follows from \eqref{eq:Deltav}
since $\beta$ satisfies $0= d(\beta \circ j) = (\frac{\del \beta_x}{\del x} + \frac{\del \beta_y }{\del y})dx \wedge dy$.
\end{proof}

Therefore we can apply the maximum principle for $v$. Now let $L^0, \ldots, L^k$
be Lagrangian submanifolds contained in $\Int N^{\text{\rm cpt}}$. Then we have
$$
a(z) \leq a_0
$$
for $a_0 > 0$ for all $z \in \del \Sigma = D^2 \setminus \{z^0, \ldots, z^k\}$.
In particular, the end points of $x^i$ are contained in $\Int N^{\text{\rm cpt}}$.
The maximum principle applied to the function $t \mapsto a \circ x^i(t)$ on $[0,w^i]$
prevents the image of $x^i$ from entering in the cylindrical region $N^{\text{\rm end}}$.

This proves that the whole image of $x^i$ is also contained in $W$.
This finishes the proof of the proposition by applying the maximum principle to $v$
based on Lemma \ref{lem:Deltaa}.
\end{proof}

\begin{rem} We comment that the conformal rescaling of $u$ to $v$ defined by
$v(z) = \psi_{\eta(z)}^{-1}(u(z))$ does not change the horizontal part, i.e.,
$\pi\circ v = \pi \circ u$.
\end{rem}

\section{Vertical $C^0$ estimates for $\frak m$-components}
\label{sec:verticalC0}

We next examine the $C^0$-bound in the fiber direction of $T^*N$.

Writing $\rho = e^s \circ v$, a straightforward calculation (see \cite[(3.20)]{seidel:biased}) derives
\be\label{eq:Deltarho}
\Delta \rho = |dv - \beta \otimes X_H|^2 - \rho H''(\rho) \frac{d\rho \wedge \beta}{dx \wedge dy}
- \rho H'(\rho) \frac{d\beta}{dx \wedge dy}
\ee
for any complex coordinates $z = x+iy$ for $(\Sigma,j)$. We note that from this equation,
the (interior) maximum principle applies.

When $\del \Sigma \neq \emptyset$, we also need to examine applicability of strong
maximum principle on the boundary $\del \Sigma$. In \cite{abou-seidel}, Abouzaid-Seidel
used certain integral estimates to control $C^0$-bound instead of the strong maximum principle.
Here we prefer to use the strong maximum principle and so provide the full details of
this application of strong maximum principle, especially for the moving boundary case.

In this section we consider the case of \emph{fixed} Lagrangian boundaries.

For this purpose, the following $C^0$-bound is an essential step in the case of
noncompact Lagrangian such as the conormal bundles $L^i = \nu^*(\partial N^{\text{\rm cpt}})$.
For given $\mathbf x=(x^1,\dots,x^{k})$ where $x^j \in {\frak X}(w^jH;L^{j-1},L^j)$ with $j = 1, \dots, k$
and $x^0 \in  {\frak X}(w^jH;L^0,L^k)$, we define
\be\label{eq:CHLgamma}
{\frak{ht}}(\mathbf x;H,\{L^i\}): = \max_{0 \leq j \leq k} \|p \circ x^j\|_{C^0}
\ee

Proof of the following proposition is a consequence of  the strong maximum principle
based on the combination of the following
\begin{enumerate}
\item $\rho = r \circ v = e^s \circ v$ with $r = |p|_g$ satisfies \eqref{eq:Deltarho},
\item the conormal bundle property of $L^i$ and
\item the special form of the Hamiltonian $H = \frac{1}{2} r^2$, which is a radial function.
\end{enumerate}
(See \cite{EHS}, \cite{oh:jdg} for a similar argument in a simpler context of
\emph{unperturbed} $J$-holomorphic equation.)

\begin{proof}[Proof of Proposition \ref{prop:vertical}]
Since the interior maximum principle is easier and totally standard for
the type of equation \eqref{eq:Deltarho}, we focus on the boundary case.

Due to the asymptotic convergence condition and $\widehat{\Sigma}$ is compact,
the maximum of the function $z \mapsto p(v(z))$ is achieved. If it happens on
one of $\infty$'s in the strip-like end, we are done.

So it remains to examine that case where the maximum
is achieved at $z_0 \in \del \Sigma$. We will apply a strong maximum
principle to prove that the maximum cannot be
achieved beyond the height of the asymptotic chords in the fiber direction of
$T^*N$. However to be able to apply the strong maximum principle, we should
do some massaging the equation \eqref{eq:Deltarho} into a more favorable form.
Here the condition $i^*\beta = 0$ and $d\beta = 0$ near the boundary
enters in a crucial way.

Choose a complex
coordinate $z = s + it$ on a neighborhood $U$ of $z_0$ so that
$$
z(\del \Sigma \cap U) \subset \R \subset \C.
$$
First the last term of \eqref{eq:Deltarho} drops out by $d\beta = 0$ near $\del \Sigma$.
For the second term, we note
$$
d\rho \wedge \beta =
\left(\frac{\del \rho}{\del s} ds + \frac{\del \rho}{\del t}dt\right)
\wedge (\beta_s ds + \beta_t dt)
= \left(\frac{\del \rho}{\del s} \beta_t  - \frac{\del \rho}{\del t} \beta_s\right) ds \wedge dt
$$
and so the second term becomes
$$
-\rho H''(\rho) \left(\frac{\del \rho}{\del s} \beta_t  - \frac{\del \rho}{\del t} \beta_s\right).
$$
On the other hand, $\beta_s = 0$ since we imposed $i^*\beta = 0$ \eqref{eq:beta} and hence
\eqref{eq:Deltarho} is reduced to
\be\label{eq:Deltarho-2}
\Delta \rho = |dv - X_H(v) \otimes \beta|^2 - \frac{\del \rho}{\del s} \beta_t
\ee
on $\del \Sigma$.

Since $v(\del \Sigma) \subset \nu^*T_i$ and $z \in \overline{z^{i}z^{i+1}} \mapsto v(z)$ defines a
curve on $L^i$  and the function $\tau \mapsto |p(v(s + 0\sqrt{-1}))| $ achieves a maximum
at $z_0$ where $z_0 = s_0 + 0\sqrt{-1}$. In particular we have
$$
0 = \frac{\del r\circ v}{\del s}(s_0) = dr\left(\frac{\del v}{\del s}(z_0)\right).
$$
But we note
$$
\frac{\del v}{\del s}(z_0) \in T_{v(z_0)}L^i \cap \ker dr_{v(z_0)}
$$
and $L^i \cap \ker dr_{v(z_0)}$
is a Legendrian subspace of $Tr^{-1}(R_0)$ with $R_0 = |p(v(z_0))|$. Therefore
$$
J\frac{\del v}{\del s}(z_0) \in \xi_{v(z_0)} \subset \ker dr
$$
and so
\be\label{eq:dr=0}
dr(-J\frac{\del v}{\del s}(z_0)) = 0.
\ee
Substituting
$$
- J\frac{\del v}{\del s}(z_0) = \frac{\del v}{\del t}(z_0) - \beta_t X_H(v(z_0))
$$
into \eqref{eq:dr=0}, we have obtained
$$
dr \left(\frac{\del v}{\del t}(z_0)\right) = dr \left(\beta_t X_H(v(z_0))\right) = 0
$$
which in turn implies $\frac{\del \rho}{\del s}(z_0) = 0$ in \eqref{eq:Deltarho-2} and so
$$
(\Delta \rho)(z_0) \geq 0.
$$
This contradicts to the strong maximum principle, \emph{unless
$\Im v \subset r^{-1}(R_0)$}. The latter is possible only when
$r(x^j) \equiv R_0$ for all $j = 0, \cdots, k$ for some $R_0\geq 0$. But if that holds,
the proposition already holds and there is nothing to prove.

This completes the proof of the proposition.
\end{proof}

\section{$C^0$ estimates for moving Lagrangian boundary}
\label{sec:moving}

In this section, we first provide a direct proof of the uniform $C^0$ estimate for the
moving boundary condition \emph{without decomposing the isotopy in two stages} hoping that
such an estimate may be useful in the future, even though this $C^0$-estimate is not used
in the present paper. Then we will briefly indicate how the proofs of easier propositions, Propositions \ref{prop:vertical-f} and \ref{prop:horizontal-f} can be obtained from the scheme of the proof with minor modifications.

We consider a map $v:\Sigma\to T^*N$ satisfying a perturbed Cauchy-Riemann
equation
\be\label{eq:duXHJ-1}
\begin{cases}
(dv- X_H \otimes \beta)_{\mathbf J}^{(0,1)}=0,\\
v\circ \epsilon^j(-\infty,t)= x^j(t), \text{ for }j=1,\dots,k.\\
v\circ \epsilon^0(+\infty,t)= x^0(t).
\end{cases}
\ee
where $x^0\in{\frak X}(w^0 H;L^0,X^1)$ and $x^d\in{\frak X}(w^dH;L^1,X^d)$
and $x^d\in{\frak X}(w^jH;X^j,X^{j+1})$ for $j=1,\dots,d-1$, and
with moving boundary $\{L^t\}$, a Hamiltonian isotopy from $L^0$ to $L^1$:
\be\label{eq:moving-bdy}
v(z) \in
\begin{cases}X^i, \, \text{for $z\in \overline{z^i z^{i+1}} \subset \partial \Sigma$
with $1 \leq i \leq d-1$.} \\
L^0, \, \text{for $z\in \overline{z^0 z^1} \subset \partial \Sigma$}\\
L^1, \, \text{for $z\in \overline{z^d z^{d+1}} \subset \partial \Sigma$}
\end{cases}
\ee
\begin{prop}\label{prop:C0bound-moving} Let $L^0:=L_i = \nu^*(a^{-1}(a_i)), \, L^1:=L_j = \nu^*(a^{-1}(a_j))$
with $a_i < a_j$ and $\{x^j\}_{0 \leq j \leq k}$ be given as above.
Define the constant
$$
\frak{ht}\left(H;L_i,L_{i+1};\{X^k\}_{k=1}^d; \{x^j\}_{0 \leq j \leq k}\right)
= \max_{0 \leq j \leq k}\|p \circ x^j\|_{C^0}.
$$
Then
$$
\max_{z \in \Sigma}|p(v(z))| \leq \frak{ht}\left(H;L_i,L_{i+1};\{X^k\}_{k=1}^d; \{x^j\}_{0 \leq j \leq k}\right)
$$
for any solution $u$ of \eqref{eq:duXHJ}.
\end{prop}

We consider the case of moving the Lagrangian $L_i = \nu^*T_i$ to $L_{i+1} = \nu^*T_{i+1}$
for a given exhaustion sequence \eqref{eq:exhaustion} where $T_i = a^{-1}(i), \, T_{i+1} = a^{-1}(i+1)$.
The main point of this case is
that the Lagrangian $L_i$ \emph{moves outward} with respect to the $a$-direction in the cylindrical
region $[0, \infty) \times 0,\infty)$.

\begin{rem} It is very interesting to see that subclosedness of the one-form $\beta$ also enters
in our proof of the vertical $C^0$-bound. See the proof of Proposition \ref{prop:C0bound-moving}
below.
\end{rem}

\begin{proof}
In order to produce a Hamiltonian flow which send $L_i$ to $L_{i+1}$, we start
with a Hamiltonian function $F:T^*N\to\R$ whose restriction to $T^*N^{\text{\rm end}}$
agrees with the coordinate function  $p_a:T^*N^{\text{\rm end}}\to\R$. Then we deduce
$$
X_{F}|_{T^*N^{\text{\rm end}}}=\frac{\partial}{\partial a}.
$$
Since the metric $g$ is cylindrical on $N^{\text{\rm end}}$, the induced Sasakian metric $\widehat g$ on
$$
T^*N^{\text{\rm end}}\cong T^*[0,+\infty)\oplus T^*\T^2
$$
can be written as $\widehat g = \widehat g_a \oplus \widehat g_{\T^2}$.

Let $\phi_F^t$ be a time $t$ flow of $X_F$ on $T^*N$ then its restriction to $T^*N^{\text{\rm end}}$ becomes
$$
\Phi_{F}^t((a,p_a),(q_T,p_T))=((a+t,p_a),(q_T,p_T)).
$$
The local coordinate expression of the Lagrangian $L_i=\nu^*(b^{-1}(k_i))$, where $k_i<0$ see Section~\ref{sec:Tame manifolds and the Busemann function}, is given by
$$
\{((-k_i,p_a),(q_T,0)):p_a\in\R,\  q_T\in \T^2\}.
$$
For $s\in[0,1]$, let us denote by
$$
L_{i+s}:=\{((-k_i+s(k -k_{i+1}),p_a),(q_T,0)):p_a\in\R,\  q_T\in \T^2\},
$$
a $[0,1]$-parameterized family of Lagrangians interpolating $L^i$ and $L^{i+1}$.
By an easy observation we have
\begin{lem}\label{lem:flow of X_F}
The flow $\Phi_{F}^{s(k -k_{i+1})}$ of $X_F$ sends $L_i$ to $L_{i+s}$ for $s\in[0,1]$.
\end{lem}

For each given test Lagrangians $X^1, \ldots, X^k$, we consider a moduli space
$\CN^{\bf w}(x^0;\mathbf x)=\CN^{\bf w}(x^0;\mathbf x;\mathcal D_{\frak m})$
with
$$
\mathcal N^{\bf w}(x^0;\mathbf x):=\bigcup_{\mathcal D\in\mathcal{FD}_{\mathbf w}}\mathcal N^{\bf w}((x^0;\mathbf x);\mathcal D)
$$
which consists of a map $u:\Sigma\to T^*N$ satisfying a perturbed Cauchy-Riemann equation
\eqref{eq:duXHJ} with moving boundary condition \eqref{eq:moving-bdy}.

\begin{rem}
We would like to alert readers that as our proof clearly shows \emph{in general construction
of such a morphism is possible only in a certain direction
that favors application of strong maximum principle.}
\end{rem}

For each given quadruple $(H;L_i,L_{i+1},\{x^j\}_{0 \leq j \leq k})$, where $x^0\in{\frak X}(w^0 H;L_{i+1})$ and $x^j\in{\frak X}(w^jH;L_{i} )$ for $j=1,\dots,k$, we define a constant
$$
\frak{ht}(H;L_i,L_{i+1},\{x^j\}_{0 \leq j \leq k}) = \max_{0 \leq j \leq k}\|p \circ x^j\|_{C^0}.
$$
Recall that our test Lagrangians $X^j$ are either compact or with cylindrical end if noncompact.
Since the boundary condition is the fixed one on $\overline{z^j z^{j+1}}$, $u(\overline{z^j z^{j+1}}) \subset X^j$,
for $1 \leq j \leq k-1$, the strong maximum principle for \eqref{eq:dvXHJ} applies thereto.

It remains to check on the strip-like end near $z^0$ where one of the boundary involves
moving boundary $\{L_s\}$ from $L_i$ to $L_{i+1}$. The only place where this is not clear is the
part where $\chi'(\tau) \neq 0$, i.e., $\tau \in (-2,-1) \subset (-\infty,0]$ in the strip-like region.

In this part of the region, $\Delta \rho$ is given by the same formula as
\eqref{eq:Deltarho} but with the moving boundary condition.
Recall the $[0,1]$-parameterized Lagrangian
$$
L_{i+s}:=\{((-k_i+s(k -k_{i+1}),p_a),(q_T,0)):p_a\in\R,\  q_T\in T\}.
$$

As before let $z_0:= \epsilon^0(\tau_0,0)$ or $\epsilon^0(\tau_0,1)$ be a point in $\del \Sigma \cap \epsilon^0(Z^-)$ where a maximum of the radial function $r(z) = |p(v(z))|_{g}$ is achieved.
Without loss of any generality, we may assume that the point is $z_0 = \epsilon^0(\tau_0, 1)$
with $ -2 < \tau_0 < -1$.
Similarly as before, take a local neighborhood $U$ of $z_0$ with coordinates
$$
\varphi:U\cap Z^-\to\H=\{x+iy\in\C:y\geq0\}
$$
such that  $\varphi(z_0)=x_0+0i$ for some $x_0$ and $\varphi(U\cap \partial Z^-)\subset \R\subset\H$.
We note that the outward normal of $Z^-$ along $t =1$ is $\frac{\del}{\del t}$ and
that of $\H$ along $\del \H = \{y =0\}$ is given by $-\frac{\del}{\del y}$. Therefore for the holomorphic
map $\varphi$, we have $\frac{\del \tau}{\del x} < 0$. This in turn implies
\be\label{eq:dadx}
\frac{\del (a \circ v)}{\del x}
= \frac{\del (a \circ v)}{\del \tau} \frac{\del \tau}{\del x} \geq 0
\ee
along $\{y=0\}$ since $\frac{\del (a \circ v)}{\del \tau} \leq 0$ on $\del Z^-$ by the direction of
the moving boundary condition in \eqref{eq:moving-bdy}.

We also have $dr\left(\frac{\del v}{\del x}(z_0)\right) = 0$ and so
$
0 = \theta\left(J \frac{\del v}{\del x}(z_0)\right)
$
for $r^2 =|p|^2 = 2H_{g_0}$.
\begin{rem}
This time $\frac{\del v}{\del x}(z_0) \not \in T_{v(z_0)}L_{\rho(x_0)}$
and hence the  argument used in the later half of the fixed boundary case
cannot be applied to the current situation. This is the precisely the reason why
a direct construction of homotopy for the moving (noncompact) Lagrangian boundary
has not been able to be made but other measures such as the cobordism approach \cite{oh:cobordism}
or Nadler's approach \cite{nadler} are taken. As we shall see below, we have found a way
of overcoming this obstacle by a two different applications of
strong maximum principle after combining all the given geometric
circumstances and a novel sequence of logics.
\end{rem}

The condition
$v \circ \epsilon^0(\tau,0) \in L_{i+\rho(\tau)}$
at $z_0$ and Lemma \ref{lem:flow of X_F} imply
\be\label{eq:Phigt}
\left(\phi_F^{\rho(x)(k_i-k_{i+1})}\right)^{-1}(v \circ \epsilon^0(x,0)) \in L^i
\ee
where $\rho(x) = \rho(\tau(x))$.
By differentiating the equation for $x$ and recalling $z_0 = \epsilon^0(x_0,0)$, we have obtained that
$$
\frac{\del \tau}{\del x}\rho'(\tau(x))(k_i-k_{i+1}) X_{\overline F}\left(\left(\phi_F^{\rho(x)(k_i-k_{i+1})}\right)^{-1}(v (z_0))\right)
+ d\left(\phi_F^{\rho(x)(k_i-k_{i+1})}\right)^{-1}\left(\frac{\del v}{\del x}(z_0)\right)
$$
is a tangent vector of $L^i$ at $\left(\phi_F^{\rho(x)(k_i-k_{i+1})}\right)^{-1}(v(z_0))$. Here
$\overline F = -F \circ \phi_F^t$ is the inverse Hamiltonian which generates the inverse
flow $(\phi_F^t)^{-1}$. This implies
$$
- \frac{\del \tau}{\del x}\rho'(\tau(x))(k_i-k_{i+1}) X_F(v(z_0)) + \frac{\del v}{\del x}(z_0) \in T_{v(z_0)}L_{\rho(x_0)}.
$$
(See \cite{oh:fredholm} for the same kind of computations used for the Fredholm
theory for the perturbation of boundary condition.)
Recall $\theta \equiv 0$ on any conormal bundle and so on $L_{\rho(x)}$.
Therefore we obtain
$$
0 = \theta\left(-\frac{\del \tau}{\del x}\rho'(\tau(x))(k_i-k_{i+1})X_F(v(z_0))
+ \frac{\del v}{\del x}(z_0) \right).
$$
By evaluating the equation in \eqref{eq:moving-bdy} against $\frac{\del}{\del x}$,
 we get
$$
\frac{\del v}{\del x}(z_0) + J \left(\frac{\del v}{\del y}(z_0) -\beta_t X_H(v(z_0))\right)= 0.
$$
Therefore the equation $\theta\circ J = -dH_g$ gives rise to
\begin{align}
\frac12\frac{\del |p\circ v|^2}{\del y}(z_0) & = \frac{\del (H\circ v)}{\del y}(z_0)
=  - \theta\circ J(\frac{\del v}{\del y}) = \theta\left(\frac{\del v}{\del x}(z_0) \right)\nonumber \\
&= \frac{\del \tau}{\del x}\rho'(\tau(x))(k_i-k_{i+1})\theta(X_{F}(v(z_0)))\nonumber\\
&= \frac{\del \tau}{\del x}\rho'(\tau(x))(k_i-k_{i+1})p_a(v(z_0))\label{eq:dHdy}.
\end{align}

\begin{lem}\label{lem:pauz0>0} We have $p_a(v(z_0)) \geq 0$ and equality holds only when
$p_a \circ v \equiv 0$, i.e., when $\text{\rm Image } v$ is entirely contained in the zero section.
\end{lem}
\begin{proof}
In order to estimate $p_a(v(z_0))$, by a direct computation from (\ref{eq:a-component}), we derive
$$
\Delta (p_a\circ v)=p_a\circ v \left(\frac{\partial \beta_x}{\partial y}-\frac{\partial \beta_y}{\partial x}\right)
+\left(\beta_x \frac{\partial (p_a\circ v)}{\partial y}-\beta_y \frac{\partial (p_a\circ v)}{\partial x}\right),
$$
where $\beta=\beta_x dx+\beta_y dy$ near $z_0$.
The sub-closedness assumption of $\beta$ implies that $\frac{\partial \beta_x}{\partial y}-\frac{\partial\beta_y}{\partial x}\geq0$ and hence the (interior) maximum principle applies
for the positive value of $p_a(v)$ and the (interior) minimum principle for the negative values,
respectively.

Now we consider the function $p_a\circ v$ on $\del \Sigma$.
Recalling the property $i^*\beta = 0$ on $\del \Sigma$ which
implies $\beta_x = 0$, the above equation for $\Delta (p_a \circ v)$ becomes
\be\label{eq:Deltapau}
\Delta (p_a\circ v)=p_a \circ v \left(\frac{\partial \beta_x}{\partial y}-\frac{\partial \beta_y}{\partial x}\right)
-\beta_y \frac{\partial (p_a\circ v)}{\partial x}
\ee
on $\del \Sigma$. The inequality \eqref{eq:dadx} on $\partial Z^-\cap U$ implies
$$
\frac{\partial (p_a \circ v)}{\partial y}(z_0) \geq 0
$$
by the Cauchy-Riemann equation applied to the holomorphic function $a + i p_a$
of $x+iy$.
If we denote $f(z) = - p_a\circ v(z)$, this inequality becomes
\be\label{eq:dfdy}
- \frac{\partial f}{\partial y}(z_0) \leq 0.
\ee
We would like to show $f(z_0) = - p_a(v(z_0))\leq 0$.
Suppose to the contrary $f(z_0) \geq 0$, we obtain
$$
\Delta f(z_0) \geq 0
$$
from \eqref{eq:Deltapau} by combining $\frac{\del (p_a\circ v)}{\partial x}(z_0)= 0$.
Therefore since $-\frac{\del}{\del y}$ is the outward normal to
$\{y=0\}$, \eqref{eq:dfdy} contradicts to the (strong) maximum principle. Therefore
$$
p_a(v(z_0)) \geq 0
$$
and equality holds only when $p_a \circ v \equiv 0$. This finishes the proof.
\end{proof}

Substituting this lemma back into \eqref{eq:dHdy}, we obtained
$$
-\frac{\del  |p \circ v|^2}{\del y}(z_0) \leq 0
$$
and equality holds only when $p \circ v(z_0) = 0$.
This contradicts to the strong maximum principle applied to the radial function
$(r\circ v)^2 = |p\circ v|^2$ by the same way as for the
vertical $C^0$ estimates given in Section \ref{sec:verticalC0}, unless the whole image of $v$ is
contained in the level set $r^{-1}(0) = 0_N$.

But if the latter holds, all the Hamiltonian chords are constant chords and the solution $v$
must be constant maps. This is impossible because of the asymptotic condition
$$
v(\epsilon_-((-\infty,0]) \subset \nu^*T_i, \quad i = 1, \cdots, k
$$
and
$$
v(\epsilon_+([0,\infty)) \subset \nu^*T_{i+1}, \quad i= 0.
$$
This finishes the proof of Proposition \ref{prop:C0bound-moving}.
\end{proof}

Now we briefly indicate the main points of the proofs of Propositions \ref{prop:vertical-f}
and \ref{prop:horizontal-f}].

For these cases, the boundaries are fixed as in the proofs of
Proposition \ref{prop:horizontal}, \ref{prop:vertical} is that the Hamiltonian
involved is not $\tau$-independent due to the appearance of the elongation function $\chi$
in $H^\chi$. We have only to consider the equation on the strip-like region.
{}
Writing $\rho = e^s \circ v$, the calculation (see \cite[(3.20)]{seidel:biased}) derives
\be\label{eq:Deltarho-1}
\Delta \rho = |dv - \beta \otimes X_{H_\lambda^\chi}|^2 - \rho (H_\lambda^\chi)''(\rho)
\frac{d\rho \wedge \beta}{d\tau \wedge dt} + \chi'(\tau) \frac{\del H^s}{\del s}\Big|_{s = \chi(\tau)}
\ee
We note that from this equation, both the maximum principle and strong maximum principle apply
for this non-autonomous case because
$$
\chi'(\tau) \frac{\del H^s}{\del s}\Big|_{s = \chi(\tau)} \geq 0
$$
by the monotonicity hypothesis on $\lambda$. We refer to the proof of Proposition \ref{prop:C0bound-moving}
to see how the strong maximum principle applies.

\appendix

\section{Energy identity for Floer's continuation equation}

In this appendix, we recall the energy identity for the continuation equation
given in \cite{oh:chain} and given its derivation
which was originally given in \cite{oh:jdg} in the cotangent bundle context.

\begin{prop} Suppose that $u$ is a finite energy solution of \eqref{eq:CR-chi}. Then
\bea\label{eq:energy-homotopy}
\int \left|\frac{\del u}{\del \tau}\right|_{J_{\chi(\tau)}}^2 \, dt\, d\tau
& = &
 \mathcal A_{H^+}(z^+) - \mathcal A_{H^-}(z^-) \nonumber \\
&{}& -
\int_{-\infty}^\infty \chi'(\tau) \left(\int_0^1
\frac{\del H_s}{\del s}\Big|_{s = \chi(\tau)}(u(\tau,t))dt\right)  d\tau.
\eea
\end{prop}
\begin{proof} We derive
\beastar
\int \int \left|\frac{\del u}{\del \tau}\right|_{J^{\chi(\tau)}}^2 \, dt \, d\tau & = &
\int\int \omega\left(\frac{\del u}{\del \tau}, J^{\chi(\tau)}\frac{\del u}{\del \tau}\right) \, dt\, d\tau\\
& = & \int\int \omega\left(\frac{\del u}{\del \tau}, \frac{\del u}{\del t} - X_{H^{\chi(\tau)}}(u) \right) \, dt
\, d\tau \\
& = & \int u^*\omega + \int_{-\infty}^\infty \int_0^1 dH^{\chi(\tau)}\left(\frac{\del u}{\del \tau}\right)
\, dt\, d\tau.
\eeastar
The finite energy condition and nondegeneracy of $z^\pm$ imply that the first improper
integral $\int u^*\omega$ converges and becomes
\be\label{eq:int-omega}
\int u^*\omega = \int u^*(-d\theta) = - \int_0^1 (z^+)^*\theta + \int_0^1 (z^-)^*\theta.
\ee
On the other hand, the second summand becomes
\bea\label{eq:int-dH}
\int_{-\infty}^\infty \int_0^1 dH^{\chi(\tau)}\left(\frac{\del u}{\del \tau}\right)\, dt\, d\tau
& = & \int_{-\infty}^\infty \int_0^1 \frac{\del}{\del \tau} \left(H^{\chi(\tau)}(u(\tau,t))\right)\,
dt\, d\tau \nonumber\\
&{}& \quad - \int_{-\infty}^\infty \chi'(\tau) \left(\int_0^1
\frac{\del H_s}{\del s}\Big|_{s = \chi(\tau)}(u(\tau,t))dt\right)  d\tau \nonumber\\
& = & \int_0^1 H^{\chi(\tau)}(u(\infty,t)) \, dt -  \int_0^1 H^{\chi(\tau)}(u(-\infty,t)) \, dt \nonumber\\
&{}& \quad - \int_{-\infty}^\infty \chi'(\tau) \left(\int_0^1
\frac{\del H_s}{\del s}\Big|_{s = \chi(\tau)}(u(\tau,t))dt\right)  d\tau.\nonumber\\
&{}&
\eea
By recalling $u(\pm\tau,t) = z^\pm(t)$  and adding \eqref{eq:int-omega}, \eqref{eq:int-dH}, we
have obtained \eqref{eq:energy-homotopy}.
\end{proof}

\section{Gradings and signs for the moduli spaces}

\subsection{Lagrangian branes}\label{sec:Lagrangian brane}
Let us choose a quadratic complex volume form $\eta_M^2$ on the symplectic manifold $(W=T^*(M\setminus K),\omega_{std})$ with respect to the almost complex structure $J_g$ introduced in Definition \ref{def:almost_complex_structure}. The associated {\em (squared) phase map} is
\[
\alpha_M:{\rm Gr}(TM)\to S^1,\quad \alpha_M(TL_x)=\frac{\eta_M(v_1\wedge v_2\wedge v_3)^2}{|\eta_M(v_1\wedge v_2 \wedge v_3)|^2},
\]
where $v_1,v_2,v_3$ is a basis of $TL_x$ and ${\rm Gr}(TM)$ is the Lagrangian Grassmannian.
The vanishing condition of the relative first Chern class in Definition~\ref{def:admissible Lagrangian} implies that the Maslov class $\mu_L\in H^1(L)$ vanishes on $H_1(L)$.
So we have a lifting $\alpha^\#_L:L\to \R$ of $\alpha_M|_{TL}$ satisfying $\exp(2\pi i \alpha^\#_L(x))=\alpha_M(TL_x)$, which we call a {\em grading} on $L$.

For the orientation of various types of moduli spaces, we need to consider a {\em relative spin structure} of each Lagrangian $L$. Since we have $H^2(M;\Z_2)=0$, the situation is rather simple. The relative spin structure $P^\#$ is consist of the choice of orientation of $L$ and a Pin structure on $TL$, see \cite[Section8]{fooo:book2},\cite[Section11]{seidel:book} for more general cases. Note that the vanishing condition of the second Stiefel-Whitney class $w_2(L)$ in Definition~\ref{def:admissible Lagrangian} implies the existence of the Pin structures.
We call a triple $(L,\alpha^\#,P^\#)$ a  {\em Lagrangian brane}.

\subsection{Dimensions and orientations of moduli spaces}\label{sec:Dimensions and orientations of moduli spaces}
Let $(L^{0\#}, L^{1\#})$ be a pair of exact Lagrangian branes with a choice of Floer datum $(H,J)$.
Consider a Hamiltonian chord $x\in{\frak X}(H;L^0,L^1)$ and use a linearization $d\phi_H^t$ of the Hamiltonian flow of $H$ to transport the linear Lagrangian brane
\[
{\sf L}^{0\#}_{x(0)}=((TL^0)_{x(0)},\alpha^{0\#}({x(0)}),(P^{0\#})_{x(0)})
\]
to $TM_{x(1)}$.
Then the index and the orientation space are defined by comparing two linear Lagrangian branes ${\sf L}^{0\#}_{x(0)}$ and ${\sf L}^{1\#}_{x(1)}$, see \cite[11h]{seidel:book} for the construction.
Let us denote them by $\mu(x)=\mu(H;{\sf L}^{0\#}_{x(0)},{\sf L}^{1\#}_{x(1)})$ and $o(x)=o(H;{\sf L}^{0\#}_{x(0)},{\sf L}^{1\#}_{x(1)})$, respectively.

We have a more direct description of the grading in the cotangent bundle setup from \cite{oh:jdg}.
Note that the sequence of Riemannian metrics $\{g_0\}_{i\in\N}$ on the knot complement $M\setminus K$ give the canonical splittings of two Lagrangian subbundles
\[
T_{x(t)}M=H_{x(t)}\oplus V_{x(t)},
\]
where the vertical tangent bundle $V_{x(t)}$ is canonically isomorphic to $T^*_{\pi(x(t))}(M\setminus K)$, and the horizontal subbundle $H_{x(t)}$ with respect to the Levi-Civita connection of $g_0$ is isomorphic to $T_{\pi(x(t))}(M\setminus K)$ under $T\pi:TM\to T(M\setminus K)$.

We now consider a canonical class of symplectic trivialization
\[
\Phi:x^*TM\to [0,1] \times \C^3,
\]
satisfying $\Phi(H_{x(t)})\equiv \R^3$ and $\Phi(V_{x(t)})\equiv i\R^3$ for all $t\in[0,1]$. Since $[0,1]$ is contractible, such a trivialization exists.
For example, if $x\in{\frak X}(H_{g_0};\nu^*T_i, \nu^*T_i)$, then we have
\[
\Phi(T_{x(0)}\nu^*T_i)=U_\Phi \oplus U_\Phi^\perp,\quad \Phi(T_{x(1)}\nu^*T_i)=W_\Phi \oplus W_\Phi^\perp,
\]
where $U_\Phi, W_\Phi$ are 2-dimensional subspaces and $U_\Phi^\perp,W_\Phi^\perp$ are the corresponding annihilators.
Then the index $\mu(x)$ is defined by following the definition of the Maslov index in \cite{robbin-sal}
\[
\mu(B_\Phi(U_\Phi \oplus U_\Phi^\perp),W_\Phi \oplus W_\Phi^\perp),
\]
where $B_\Phi:[0,1] \to {\rm Sp}(6)$ is a symplectic path given by $B_\Phi:=\Phi \circ T\phi_{H_{g_0}}^t \circ \Phi^{-1}$.

Let $\Sigma$ be a $(d+1)$-pointed disk equipped with a fixed conformal structure, with strip-like ends, and with brane structures $(L^{0\#},\dots,L^{d\#})$  on each boundary component.
Consider Floer data $\mathcal D$ in Definition \ref{def:Floer data} for
\begin{align}\label{eqn:x^0 and x^k}
x^0\in{\frak X}(L^d,L^0),\quad {\bf x}=(x^1,\dots,x^d),\quad x^k\in{\frak X}(L^{k},L^{k-1}), \quad k=1,\dots,d.
\end{align}
Let $u\in \mathcal M(x^0,\bf x;\mathcal D)$ be a regular point,
then the linearization of $(dv - X_{\bf H}\otimes \beta)_{\bf J}^{(0,1)}$ become
\[
D_{\mathcal D,u}:(T\mathcal B_\mathcal D)_u\to(\mathcal E_\mathcal D)_u,
\]
where $\mathcal B_\mathcal D$ is a Banach manifold of maps in $W^{1,p}_{loc}(\Sigma, M)$ satisfying $W^{1,p}$-convergence on each strip-like ends,
and $\mathcal E_{\mathcal D}\to \mathcal B_{\mathcal D}$ is a Banach vector bundle whose fiber at $u$ is $L^p(\Sigma,\Omega^{0,1}_\Sigma\otimes u^*TM)$.

Let $\Sigma$ be a $(d+1)$-pointed disk equipped with a fixed conformal structure, with strip-like ends, and with a brane structure on each boundary component, set say $L^{0\#},\dots,L^{d\#}$.
Consider Floer data $\mathcal D$ in Definition \ref{def:Floer data} for $x^0\in{\frak X}(H;L^d,L^0)$ and $x^k\in{\frak X}(H;L^{k},L^{k-1})$, $k=1,\dots,d$.
Let $u\in \mathcal M(x^0;\bf x;\mathcal D)$ be a regular point, where ${\bf x}=(x^1,\dots,x^d)$, then we have
\begin{align*}
\dim_u \mathcal M(x^0;{\bf x};\mathcal D)&=\mu(x^0)-\mu(x^1)-\cdots-\mu(x^d);\\
(\Lambda^{\rm top}T\mathcal M(x^0;{\bf x};\mathcal D))_u
&\cong \Lambda^{\rm top}\ker(D_{\mathcal D,u}) \otimes \Lambda^{\rm top} \coker(D_{\mathcal D,u})\\
& \cong o(x^0)\otimes o(x^1)^\vee \otimes \cdots \otimes o(x^d)^\vee.
\end{align*}
Here, $D_{\mathcal D,u}$ is the linearized operator of $(dv - X_H\otimes \beta)_J^{(0,1)}$, where $({\bf H}, {\bf J},\beta)$ comes from the data $\mathcal D$, and $^\vee$ denotes the dual vector space.

\subsubsection{Sign convention for $A_\infty$ structure map}

We briefly explain the sign convention adopted in Proposition \ref{eqn:A_infty relation}. Let $\Sigma^1$ be a $(d+2-m)$-pointed disk equipped with Floer data $\mathcal D^1$ and with brane structures
\[
\vec L^{1\#}:=(L^{0\#},\dots,L^{n\#},L^{n+m\#},\dots, L^{d\#}),
\]
and $\Sigma^2$ be an $(m+1)$-pointed disk with $\mathcal D^2$ and with branes
\[
\vec L^{2\#}:=(L^{n\#},\dots,L^{n+m\#}).
\]
Now consider the gluing of $(\Sigma^1,\mathcal D^1,\vec L^{1\#})$ and $(\Sigma^2,\mathcal D^2,\vec L^{2\#})$ near the outgoing end of $\Sigma^2$ and the $(n+1)$-th incoming end to obtain a $(d+1)$-pointed disk $\Sigma$ with Floer data $\mathcal D:=\mathcal D^1\#_{n+1}\mathcal D^2$ and with the combined brane data
\[
\vec L^{1\#}\#_{n+1}\vec L^{2\#}=(L^{0\#},\dots, L^{d\#}).
\]

For given regular points
\begin{align*}
u^1&\in\mathcal M(x^0;x^1,\dots,x^n,\bar x,x^{n+m+1},\dots, x^d;\mathcal D^1);\\
u^2&\in\mathcal M(\bar x;x^{n+1},\dots,x^{n+m};\mathcal D^2),
\end{align*}
the standard gluing procedure gives rise to a regular point
\[
u\in\mathcal M(x^0;\dots,x^d;\mathcal D).
\]
This gluing process extends to a local diffeomorphism, let say $\phi$, between neighborhood of $(u^1,u^2)$ to that of $u$.
Then its linearization $D\phi$ fits into the following commutative diagram:
\[
\begin{tikzcd}[row sep=1.5cm, column sep=1cm]
   (\Lambda^{\rm top}T\mathcal M(x^0,\dots,x^d;\mathcal D))_u\arrow[r, "\cong"]
    & o(x^0)\otimes o(x^1)^\vee\otimes\cdots\otimes o(x^d)^\vee \\
  \parbox{4cm}{\centering $(\Lambda^{\rm top}T\mathcal M({\bf x}^1;\mathcal D^1))_{u^1} \otimes (\Lambda^{\rm top}T\mathcal M({\bf x}^2;\mathcal D^2))_{u^2} $}\arrow[r, "\cong"] \arrow[u, "(\Lambda^{\rm top}D\phi)_{(u^1,u^2)}"]
&o({\bf x^1})\otimes o({\bf x^2}) \arrow[u, "(-1)^\Delta"]
\end{tikzcd}
\]

The above diagram commutes when
\[
\Delta=\left(\sum_{k>n+m}\mu(x^k)\right)\cdot \dim\mathcal M({\bf x}^2;\mathcal D^2)_{u^2}.
\]
Here
\begin{align*}
{\bf x^1}&=(x^0,\dots,x^n,\bar x,x^{n+m+1},\dots, x^d);\\
{\bf x^2}&=(\bar x,x^{n+1},\dots,x^{n+m}).
\end{align*}
and hence
\begin{align*}
o({\bf x^1})&=o(x^0)\otimes o(x^1)^\vee\otimes\cdots\otimes o(x^n)^\vee\otimes o(\bar x)^\vee\otimes o(x^{n+m+1})^\vee\otimes\cdots\otimes o(x^d)^\vee;\\
o({\bf x^2})&=o(\bar x)\otimes o(x^{n+1})^\vee\otimes \cdots \otimes o(x^{n+m})^\vee.
\end{align*}

Now we consider the $\mathcal M$-parameterized Floer data used in Definition~\ref{def:universal Floer data}
\[\frak D=\frak D_{\frak m}=\bigsqcup_{r\in\mathcal M}\mathcal D(r),\]
on $(d+1)$-pointed disks, with branes structures $(L^{0\#},\dots,L^{d\#})$ on the ends.
Let us denote the corresponding parameterized moduli space by
\[
\mathcal M_{\frak D}(x^0;{\bf x}):=\bigsqcup_{r\in \mathcal M}\{r\}\times\mathcal M(x^0;{\bf x};\mathcal D(r)),
\]
where $(x^0;\bf x)$ are the same as in (\ref{eqn:x^0 and x^k}).
For a regular point $(r,u)$ of the parameterized moduli space, the {\em extended linearized operator} is given by
\[
D_{\mathcal D(r),u}:(T\mathcal M)_r \times (T\mathcal B_{\mathcal D(r)})_u \to (\mathcal E_{\mathcal D(r)})_u.
\]
We then have
\begin{align*}
\dim_{(r,u)} \mathcal M_{\frak D}(x^0;{\bf x})&=\dim_r\mathcal M+\mu(x^0)-\mu(x^1)-\cdots-\mu(x^d);\\
(\Lambda^{\rm top}\mathcal M_{\frak D}(x^0,{\bf x}))_{(r,u)}
&\cong (\Lambda^{\rm top}T\mathcal M)_r \otimes\Lambda^{\rm top}\ker(D_{\mathcal D(r),u})\\
&\cong (\Lambda^{\rm top}T\mathcal M)_r \otimes o(x^0)\otimes o(x^1)^\vee \otimes \cdots \otimes o(x^d)^\vee.
\end{align*}

Now consider two Floer data $\frak D^j$ parameterized over $\mathcal M^j$ for $j=1,2$.
The brane data and asymptotic data are given as before by $(\vec L^{1\#}, \vec L^{2\#})$ and $(\bf x^1,x^2)$.
Then the gluing induces a Floer data $\frak D$ parameterized over
\[
\mathcal M=\R^+ \times \mathcal M^1 \times \mathcal M^2,
\]
where $\R^+$ plays a role of gluing parameter.
Let us choose regular points
\begin{align*}
(r^1,u^1)&\in\mathcal M_{\frak D^1}(x^0,\dots,x^n,\bar x,x^{n+m+1},\dots, x^d);\\
(r^2,u^2)&\in\mathcal M_{\frak D^2}(\bar x,x^{n+1},\dots,x^{n+m}),
\end{align*}
then for each gluing parameter $\ell \in \R^+$ we have a regular point
\[
(\ell,u)\in\mathcal M_{\frak D}(x^0,\dots,x^d).
\]
Note that we obtain
\[
(\Lambda^{\rm top}T\mathcal M)_r \cong \R \otimes (\Lambda^{\rm top}T\mathcal M^1)_{r^1} \otimes (\Lambda^{\rm top}T\mathcal M^2)_{r^2},
\]
when $\ell$ is sufficiently large. By the same argument as above we have the following diagram:
\[
\begin{tikzcd}[row sep=1.5cm, column sep=1cm]
\parbox{4.5cm}{\centering $(\Lambda^{\rm top}T\mathcal M_{\frak D}(x^0,\dots,x^d))_{(r,u)}$}\arrow[r, "\cong"]
    & \parbox{3.5cm}{\centering $(\Lambda^{\rm top}T\mathcal M)_r \otimes o(x^0)\otimes o(x^1)^\vee\otimes\cdots\otimes o(x^d)^\vee$} \\
\parbox{4.5cm}{\centering $(\Lambda^{\rm top}T\mathcal M_{\frak D^1}({\bf x}^1))_{(r^1,u^1)} \otimes (\Lambda^{\rm top}T\mathcal M_{\frak D^2}({\bf x}^2))_{(r^2,u^2)}$ } \arrow[r, "\cong"] \arrow[u, "(\Lambda^{\rm top}D\phi)_{((r^1,u^1),(r^2,u^2))}"]
&\parbox{4.5cm}{\centering $\R \otimes (\Lambda^{\rm top}T\mathcal M^1)_{r^1} \otimes o({\bf x^1}) \otimes (\Lambda^{\rm top}T\mathcal M^2)_{r^2} \otimes o({\bf x^2})$}\arrow[u, "(-1)^*"]
\end{tikzcd}
\]

Here $o({\bf x^1})$ and $o({\bf x^2})$ are the same as before. The sign in the right vertical arrow comes from the Koszul convention
\begin{align}\label{eqn:Koszul sign for M}
\blacksquare_{\frak m}&=\dim_{r^2}\mathcal M^2 \cdot {\rm index}(D_{\mathcal D^1(r^1),u}) + \left(\sum_{k>n+m}\mu(x^k)\right)\cdot {\rm index} (D_{\mathcal D^2(r^2),u^2})\\
&\equiv m(d-m-1)+m\left(\sum_{k>n+m}\mu(x^k)\right)
\end{align}
and the sign difference between parameter spaces $\mathcal M$ and $\mathcal M^1 \times \mathcal M^2$
\begin{align}\label{eqn:sign for parameter space}
\blacktriangle_{(r^1,r^2)}=m(d-n)+m+n.
\end{align}
Recall from (\ref{eqn:structure map m}) that the structure map $\frak m^k$ already has the sign $\dagger=\sum_{i=1}^k i \cdot \mu(x^i)$.
By considering all the above sign effects we have
\begin{align*}
\dagger_{u^1}+\dagger_{u^2}+\blacksquare_{\frak m}+\blacktriangle_{(r^1,r^2)} \equiv \ddagger_u + \sum_{i=1}^{d}(i+1)\mu(x^i),
\end{align*}
where $\ddagger=\sum_{i=1}^{n}\mu(x^i)-n$. This explains the sign convention in Proposition \ref{eqn:A_infty relation}.

\subsubsection{Sign convection for $A_\infty$ functor}
For the sign convention in (\ref{eqn:functor_relation1}) and (\ref{eqn:functor_relation2}), we need to consider the previous argument with the parameter space $\overline{\mathcal N}{}^{k+1}$ in Definition~\ref{def:parameterN}. The notations in this section are the same as in Section~\ref{sec:construction of functor}.

Firstly, let us consider a regular point
\begin{align}\label{eqn:regular point N}
(r,u) \in \overline{\mathcal N}{}^{d+1}(y^0;x^1,\dots,x^d)
\end{align}
and its boundary strata
\begin{align*}
(r^0,u^0) &\in \overline{\mathcal N}{}^{d-m+2}(y^0;x^1,\dots, x^n,\tilde x,x^{n+m+1},\cdots,x^d);\\
(r^1,u^1) &\in \mathcal M^{m+1}(\tilde x;x^{n+1},\dots,x^{n+m})
\end{align*}
satisfying
\begin{align*}
\mu(y^0)&=\mu(x^1)+\cdots+\mu(x^n)+\mu(\tilde x)+\mu(x^{n+m+1})+\cdots+\mu(x^d)+d-m;\\
\mu(\tilde x)&=\mu(x^{n+1})+\cdots+\mu(y^{n+m})+m-2.
\end{align*}
As in the previous section we compare the sign difference between
\[
(\Lambda^{\rm top}\overline{\mathcal N}{}^{d+1})_r \otimes o(y^0) \otimes o(x^1)^\vee \otimes \cdots \otimes o(x^d)^{\vee}
\]
and
\begin{align*}
&(\Lambda^{\rm top}\overline{\mathcal N}{}^{d-m+2})_{r^0} \otimes o(y^0) \otimes o(x^1)^\vee \otimes \cdots \otimes o(\tilde x) \otimes \cdots \otimes o(x^d)^{\vee}\otimes \\
&(\Lambda^{\rm top}\mathcal M^{m+1})_{r^1} \otimes o(\tilde x) \otimes o(x^{n+1})^\vee \cdots o(x^{n+m})^\vee.
\end{align*}

The Koszul convention gives
\[
\blacksquare_{\frak f} \equiv m(d-m)+m\left(\sum_{k>n+m}\mu(x^k)\right)
\]
and a similar computation
\[
\blacksquare_{\frak f}+\blacktriangle_{(r^0,r^1)}+\dagger_{u^1}+\spadesuit_{u^0} \equiv (\ddagger_u+1)+\left(\sum_{i=1}^d(i+1)\mu(x^i)+d \right)
\]
verifies (\ref{eqn:functor_relation1}).

Now we consider another type of boundary strata of (\ref{eqn:regular point N}) consisting of
\begin{align*}
(r^0,u^0) &\in \mathcal M^{\ell+1}(y^0;y^1,\dots, y^\ell);\\
(r^1,u^1) &\in \overline{\mathcal N}{}^{s_1+1}(y^1;x^{1},\dots,x^{s_1});\\
\vdots\\
(r^\ell,u^\ell)&\in \overline{\mathcal N}{}^{s_\lambda+1}(y^\ell;x^{s_1+\cdots+s_{\ell-1}+1},\dots,x^{d});\\
\end{align*}
satisfying degree conditions
\begin{align*}
\mu(y^0)&=\mu(y^1)+\cdots+ \mu(y^\ell) + r-2;\\
\mu(y^1)&=\mu(x^1)+\cdots+ \mu(x^{s_1}) + s_1-1;\\
\vdots\\
\mu(y^\ell)&=\mu(x^{s_1+\cdots+s_{\ell-1}+1})+\cdots+ \mu(x^{d}) + s_\lambda-1
\end{align*}
Then the sign difference between
\[
(\Lambda^{\rm top}\overline{\mathcal N}{}^{d+1})_r \otimes o(y^0) \otimes o(x^1)^\vee \otimes \cdots \otimes o(x^d)^{\vee}
\]
and
\begin{align*}
&(\Lambda^{\rm top}\mathcal M^{\ell+1})_{r^0} \otimes o(y^0) \otimes o(y^1)^\vee \otimes \cdots\otimes o(y^\ell)^\vee \otimes \\
&(\Lambda^{\rm top}\overline{\mathcal N}{}^{s_1+1})_{r^1} \otimes o(y^1) \otimes o(x^1)^\vee \otimes \cdots\otimes o(x^{s_1})^\vee \otimes \\
&\cdots \otimes \\
&(\Lambda^{\rm top}\overline{\mathcal N}{}^{s_\lambda+1})_{r^\ell} \otimes o(y^\ell) \otimes o(x^{s_1+\cdots+s_{\ell-1}+1})^\vee \otimes \cdots\otimes o(x^{d})^\vee.
\end{align*}
Note that the Koszul sign convention induces
\[
\blacksquare_{\frak f'} \equiv (s_1'+\cdots + s_\lambda')\ell+\sum_{1\leq i \leq j \leq \ell}s_i' s_j' +\sum_{i=1}^{\ell-1}(s_1'+\cdots+s_i')\mu(y^{i+1}),
\]
where $s_r'=s_r-1$. The sign convention for the parameter space configuration $(r;r^0,\dots,r^\ell)$ is obtained by applying the sign rule in (\ref{eqn:sign for parameter space}) as follows:
\[
\blacktriangle_{(r^0,\dots,r^\ell)}=d\ell+d+\sum_{1\leq i \leq j \leq \ell}s_i' s_j'.
\]
Then by a tedious computation, we have
\[
\blacksquare_{\frak f'}+\blacktriangle_{(r^0,\dots,r^\ell)}+\dagger_{u^0}+\sum_{j=1}^{\ell}\spadesuit_{u^j}=\sum_{i=1}^d(i+1)\mu(x^i)+d.
\]
This verifies the signs in (\ref{eqn:functor_relation2}).
The sign for $A_\infty$ homotopy relation requires basically the same computation, we omit it.

\end{document}